\documentclass[12pt]{amsart}
\usepackage{amsmath}
\usepackage{amsthm}
\usepackage{amssymb}
\usepackage{amscd}
\usepackage{amsfonts}
\usepackage{epsfig}

\newtheorem{thm}{Theorem}[section]
\newtheorem{lem}[thm]{Lemma}
\newtheorem{prop}[thm]{Proposition}
\newtheorem{cor}[thm]{Corollary}
\newtheorem{defn}[thm]{Definition}
\newtheorem{fact}{Fact}

\theoremstyle{remark}
\newtheorem{rem}[thm]{Remark}

\def\ga{\mathfrak{a}}\def\gg{\mathfrak{g}}\def\gh{\mathfrak h}\def\gp{\mathfrak{p}}
\def\gr{\mathfrak{r}}\def\gq{\mathfrak{q}}\def\gm{\mathfrak{m}}\def\O{\mathcal{O}}

\def\BT{\mathcal{BT}}\def\F{\mathcal F}\def\G{\mathcal G}\def\K{\mathcal K}
\def\L{\mathcal L}\def\M{\mathcal M}\def\R{\mathcal{R}}\def\Q{\mathcal Q}\def\Y{\mathcal Y}

\def\Abar{\overline{A}}\def\Gbar{\overline{G}}\def\Mbar{\overline{M}}\def\Ybar{\overline{Y}}\def\Xbar{\overline{X}}
\def\Qbar{\overline{Q}}\def\Pbar{\overline{P}}\def\ebar{\overline{\eta}}\def\sbar{\overline\sigma}\def\Nbar{\overline{N}}

\def\a{\mathbb A}\def\d{\mathbb D}\def\p{\mathbb P}\def\q{\mathbb Q}\def\f{\mathbb F}
\def\c{\mathbb C}\def\g{\mathbb G}\def\h{\mathbb H}\def\r{\mathbb R}\def\z{\mathbb Z}

\def\zd{{\mathaccent94\z}}\def\md{\mathaccent94 m}\def\Yd{\mathaccent94 Y}
\def\Od{{\mathaccent94\O_L}}\def\Rd{{\mathaccent94 R}}\def\Qd{\mathaccent94 Q}
\def\Sd{\mathaccent94 S}\def\otd{\mathaccent94\otimes}

\def\lg{\operatorname{length}}\def\sy{\operatorname{sym}}\def\End{\operatorname{End}}
\def\Aut{\operatorname{Aut}}\def\Ad{\operatorname{Ad}}\def\PGL{\operatorname{PGL}}
\def\GL{\operatorname{GL}}\def\SL{\operatorname{SL}}\def\U{\operatorname{U}}
\def\Hom{\operatorname{Hom}}\def\Mat{\operatorname{Mat}}\def\Ext{\operatorname{Ext}}
\def\EXT{\operatorname{EXT}}\def\TORS{\operatorname{TORS}}\def\Lie{\operatorname{Lie}}
\def\Gal{\operatorname{Gal}}\def\Spec{\operatorname{Spec}}\def\Card{\operatorname{Card}}
\def\Df{\operatorname{Defo}}\def\tr{\operatorname{tr}}
\def\rk{\operatorname{rank}}\def\ch{\operatorname{char}}

\title{Construction of abelian varieties with given monodromy}
\author{Oliver B\"ultel}
\thanks{Mathematisches Institut der Universit\"at Heidelberg, Im Neuenheimer Feld 288, 69120 Heidelberg, Germany,
bueltel@mathi.uni-heidelberg.de, Subject Classification(2000): 14L05, 14K10, 14C25, 14C30}

\begin{document}

\maketitle

\begin{abstract} 
Let $\rho$ be a finite dimensional faithful representation of a semisimple algebraic group $G$. By means of a deformation argument we show that there exists a family of 
abelian varieties over a smooth projective curve $S/\f_p^{ac}$ with $\ell$-adic monodromy group covering $G$ and $\ell$-adic monodromy representation containing $\rho$.
\end{abstract}

\tableofcontents
\section{Introduction}

Recall that a complex abelian variety $Y/\c$ has a period lattice $TY$ which one might introduce to be the kernel of the universal covering map $\exp:\Lie Y\rightarrow Y$. As $TY$ is discrete and cocompact in $\Lie Y$ this means that the $\r$-vector space $\r\otimes TY$, being $\Lie Y$, has in fact a $\c$-structure, consequently we obtain a subgroup
\begin{equation}
\label{minus}
\c^\times\hookrightarrow\GL_\r(\r\otimes TY),
\end{equation}
which just arises by multiplying elements in $\Lie Y$ by elements in $\c^\times$. Now there exists a smallest $\q$-algebraic subgroup $H\subset\GL(VY/\q)$, where 
$VY=\q\otimes TY$, whose group of $\r$-valued points contains the above subgroup $\c^\times$, here we write $\GL(VY/\q)$ for the $\q$-algebraic group whose 
$Q$-points are $\GL_Q(Q\otimes VY)=$ set of $Q$-linear bijections on $Q\otimes VY$. This connected, reductive subgroup is called the Hodge group, it measures 
the individuality of $Y$, for example $H$ is abelian if and only if there is a semisimple commutative subalgebra in $\End^0(Y)$ that has degree $2\dim Y/\c=\dim_\q VY$, 
i.e. if and only if $Y$ ``has complex multiplication''. It is folklore that the subgroup \eqref{minus} gives a cocharacter $\mu:\g_m\rightarrow H\times\c$, which in 
$\GL_\c(\c\otimes TY)$ is conjugated to $z\mapsto(z,\dots,z,1,\dots,1)$, moreover there is no $\q$-simple factor of $H$ in which $\mu$ is trivial.\\
The significance of the Hodge group stems from its relation to the monodromy group of a family of complex abelian varieties: Let $Y$ be an abelian scheme over a 
smooth connected complex algebraic variety $S/\c$. Let $\xi$ be a base point on $S$ and let $Y_\xi/\c$ be the fibre over $\xi$. As $Y$ is a fibration there is an operation 
of $\pi_1(S,\xi)$ on $TY_\xi$ and $VY_\xi$. One simply defines the monodromy group $G_\xi$ of $Y$ to be the Zariski closure of the image of $\pi_1(S,\xi)$. By results 
of Y. Andr\'e \cite[Lemma 4/Theorem 1]{andre} there is a pathwise connected subset $S^\circ\subset S$ with meagre complement such that:
\begin{itemize}
\item
the Hodge group $H_\xi$ of $Y_\xi$ is locally constant on $S^\circ$,
\item
for all $\xi\in S^\circ$ the neutral component $G_\xi^\circ$ of the monodromy group is a normal subgroup of $H_\xi^{der}$,
\end{itemize}
where the first statement has to be interpreted by locally trivializing the locally constant sheaf $VY_\xi$, 
of course $G_\xi$ is locally constant with no restriction whatsoever. In addition we have a natural map:
$$\rho_\xi:G_\xi\hookrightarrow\GL(VY_\xi/\q) ,$$
which we want to call the monodromy representation. The whole scenario puts very strong limitations on the structure of the 
groups $H_\xi$ and $G_\xi^\circ$, leading to Deligne's well-known $A_l$, $B_l$, $C_l$, $D_l^\r$, $D_l^\h$ classification.\\
The purpose of this work is to show that over bases of positive characteristic $p$ the analogous $\ell$-adic monodromy groups $G_\ell$, where $\ell\neq p$, have a 
markedly different behaviour: If $Y$ is an abelian scheme over an algebraic variety $S$ over say the algebraic closure $\f_p^{ac}$ of the prime field, then the fibre 
$Y_\xi/\f_p^{ac}$ over a geometric point $\xi:\Spec\f_p^{ac}\rightarrow S$ can be used to define integral and rational Tate modules $T_\ell Y_\xi$ and $V_\ell Y_\xi$ 
over $\z_\ell$ and $\q_\ell$. On these there is an operation of $\pi_1^{et}(S,\xi)$, and the $\ell$-adic monodromy group $G_{\xi,\ell}\subset\GL(V_\ell Y_\xi/\q_\ell)$ 
is defined to be the Zariski closure of it, clearly independent of $\xi$. We define the $\ell$-adic monodromy representation to be the map
$$\rho_{\xi,\ell}:G_{\xi,\ell}\hookrightarrow\GL(V_\ell Y_\xi/\q_\ell)$$
with which $G_{\xi,\ell}$ is equipped by definition. By Deligne's theory of pure sheaves $G_{\xi,\ell}$ has a semisimple neutral component, \cite[Corollaire 1.3.9]{deligne4}. 
By \cite[Theorem 4.1]{larsen} the pair $(G_{\xi,\ell}\times_{\q_\ell}\c,\rho_{\xi,\ell}\times_{\q_\ell}\c)$ is independent of $\ell$, for imbeddings $\q_\ell\hookrightarrow\c$. 
Except for these fundamental results not much seems to be known. I am grateful to Prof. R. Pink for posing the stimulating question to me, on whether there were any 
necessary or sufficient criteria for groups with faithful group representation $(G_{\xi,\ell},\rho_{\xi,\ell})$ to arise from abelian schemes over varieties in characteristic 
$p$. One should certainly expect that one has more possibilities for $(G_{\xi,\ell},\rho_{\xi,\ell})$ than in characteristic $0$, but it is not so obvious how to construct 
examples where $G_{\xi,\ell}$ has a simple factor of exceptional type $G_2$.\\ 
Before we can state a result in this direction, we want to remark, that it is often easier to merely consider the projections of $G_{\xi,\ell}$ onto certain subspaces of 
$V_\ell Y_\xi$. Suppose a subfield $L\subset\End^0(Y)$ is given, and suppose $\gr|\ell$ is a prime of $\O_L$. Then it is traditional to consider the $L_\gr$-vector 
space $V_\gr Y_\xi=\q\otimes T_\gr Y_\xi$, where $T_\gr Y_\xi$ is the module $\lim\limits_{\leftarrow N}Y_\xi[\gr^N]$ which has rank $\frac{2\dim Y/S}{[L:\q]}$ over 
$\O_{L_\gr}$. Again we write $\rho_{\xi,\gr}:G_{\xi,\gr}\hookrightarrow\GL(V_\gr Y_\xi/L_\gr)$ for the smallest $L_\gr$-algebraic subgroup containing the projection 
of $G_{\xi,\ell}$ onto the direct summand $V_\gr Y_\xi$, or equivalently, containing the image of $\pi_1^{et}(S,\xi)$ in $\GL_{L_\gr}(V_\gr Y_\xi)$.\\
Furthermore we need a word on representations with pairing: Let $G^+$ be a semisimple algebraic group over a 
field $L^+$. By a unitary representation of $G^+$ we mean a triple $(B,(.,..),\rho^+)$ having the following properties:
\begin{itemize}
\item
$B$ is a free $L$-module of finite rank, where $L$ is a semisimple $L^+$-algebra of rank two.
\item
$(.,..)$ is a non-degenerate sesquilinear pairing on $B$, where the underlying involution is $x^*=\tr_{L/L^+}(x)-x$.
\item
$\rho^+:G^+\rightarrow\U(B/L)$ is a homomorphism of algebraic groups over $L^+$ with $\U(B/L)$ being the group of unitary isometries of $(B,(.,..))$
\end{itemize}
If $L$ is understood we denote the scalar extension $G^+\times_{L^+}L$ by the symbol $G$. Observe that every unitary representation of $G^+$ determines a usual 
one of $G$ as $\U(B/L)\times_{L^+}L$ is $\GL(B/L)$, we denote this representation by the symbol $(B,\rho)$. By a homomorphism between unitary representations 
$(B_1,(.,..),\rho_1^+)$ and $(B_2,(.,..),\rho_2^+)$ we mean a $G$-homomorphism between $(B_1,\rho_1)$ and $(B_2,\rho_2)$, i.e. we forget the pairing. However, every element $f\in\Hom_G(B_1,B_2)$ has a transpose $f^*\in\Hom_G(B_2,B_1)$ defined by $(f(x_1),x_2)=(x_1,f^*(x_2))$ and $f$ is called an isometry (resp. similarity) if 
$f^*\circ f$ and  $f\circ f^*$ are the identities (resp. homotheties) of $B_1$ and $B_2$. Let us call $(B,(.,..))$ positive definite if $L$ is a CM field and $(x,x)$ is a totally positive 
element of $L^+$ for all $x\in B-\{0\}$. Clearly every positive definite representation is isometric to an orthogonal direct sum of irreducible ones. Notice that a faithful positive definite representation of $G^+$ exists if and only if it is $L_\r^+$-anisotropic, in which case the adjoint representation $(\Lie G,(.,..),\Ad)$ is positive definite, where 
$(x,y)=-\tr(\Ad(x)\circ\Ad(y^*))$. We need a technical condition on such objects:

\begin{defn}
\label{vzwei}
Assume that our fixed CM field $L$ has degree $2r$ over $\q$ and is unramified at a fixed odd prime $p$. Assume further that $p$ is inert in $L^+$ and that it splits in $L$. 
Write $\gq$, $\gq^*$, and $\gq^+=\gq\gq^*$ for the three primes of $L$ and $L^+$ over $p$, each of them has norm $p^r$. Let $v$ be an element of $L^\times$ with 
$v^*=-v$. Assume further that the number of $\q_p$-linear embeddings $\iota:L_\gq\hookrightarrow\q_p^{ac}$ with $\Im(\iota(v))>0$. (resp. $\Im(\iota(v))<0$) is equal 
to $a$ (resp. $r-a$), where ``$\Im(z)$'' denotes the imaginary part of $z$ with respect to a choice of embedding $\q_p^{ac}\hookrightarrow\c$. This is independent of the 
choice, provided only one decrees $a\leq r-a$.\\
Now we call a positive definite representation $B/L$ of $G^+/L^+$ $v$-$\gq$-flexible if:
\begin{itemize}
\item[(X.1)]
The dimension of $B$ is strictly smaller than $\frac{r}{a}$.
\item[(X.2)]
There is an element $u\in G(L_\gq)$ that satisfies $0=(\rho(u)-1)^{a+1}$ and is not killed by any of the projections onto the $L^+$-simple factors.
\end{itemize}
\end{defn}

In this work the following result is proved:

\begin{thm}
\label{newnull}
Fix $L/L^+$, $v$ and $\gq$ with properties as above and let $B/L$ be a faithful unitary representation of the semisimple group $G^+/L^+$. Suppose 
that there exists a $v$-$\gq$-flexible representation $B'/L$ of $G^+/L^+$ such that $B'$ contains all of $B$, $\Lie G$, and a copy of the trivial representation.\\
Then there exists a polarized abelian scheme $(Y,\lambda)$ over a projective and smooth curve $S/\f_p^{ac}$, together 
with a Rosati invariant map $L\hookrightarrow\End^0(Y)$, such that there is an isometry of unitary representations: 
\begin{equation}
\label{inverse}
(G_{\xi,\gr^+},\rho_{\xi,\gr^+},\frac{\Psi_\lambda(.,..)}{v})\cong(G^+\times_{L^+}L_{\gr^+}^+,\rho^+\times_{L^+}L_{\gr^+}^+,(.,..))
\end{equation}
for all primes $\gr^+|\ell\neq p$ of $\O_{L^+}$ and for all geometric points $\xi$ of $S$. Here it is understood that 
$\tr_{L/\q}\Psi_\lambda=\psi_\lambda:V_{\gr^+}Y_\xi\times V_{\gr^+}Y_\xi\rightarrow\q_\ell(1)$ is a multiple of the associated Weil-pairing.
\end{thm}

\subsubsection*{Comments and Examples:} Suppose, for instance that $G^+/L^+$ 
is any $L_\r^+$-anisotropic, $L_\gq$-isotropic, $L^+$-simple group of adjoint type. Then the representation $B':=L\oplus\Lie G$ is $v$-$\gq$-flexible if $v$ 
and $\gq$ have the parameter $a=2$ (choose $u$ in the longest root eigenspace) and $r=2\dim_LG+3$. Hence our theorem predicts a family where the $\gr^+$-adic 
monodromy $G_{\xi,\gr^+}\subset\GL(V_{\gr^+}Y_\xi/L_{\gr^+}^+)$ looks like $G^+\subset\GL(\Lie G/L^+)$ regarded over $L_{\gr^+}^+$. In fact, for many 
representations the assumption ``$B'\supset L\oplus\Lie G$'' is not needed. For example this is the case if $B$ is the $7$-dimensional standard 
representation of (a $L_\gq$-isotropic, $L_\r^+$-compact $L^+$-form of) the $G_2$-group, so the statement of the theorem simply holds if $B'=B$ is 
$v$-$\gq$-flexible. In particular, there exists an abelian $7\times(1\times7+1)=56$-fold with this kind of additional structure (as 
$a=1$).\\
Despite appearance of the contrary the CM extension $L$ is important in all this. In particular our methods cannot decide whether there exist abelian varieties with 
multiplication by a real field $L^+$ and such that the $\gr^+$-adic monodromy is the group $\SL_2\times L_{\gr^+}^+\subset\GL(\sy^3/L_{\gr^+}^+)$. Although 
``$\sy^3\oplus\sy^3$'' can be achieved by the means of the theorem \ref{newnull}, the result tends to have too large an arithmetic monodromy to be cut into two parts.

\subsubsection*{The tools of our proof:} First of all it is clear that when 
starting to look for families $Y/S$ one should invoke an integral model $\M^{(1)}$ of some PEL-Shimura variety $M^{(1)}$ of which the special fibre 
$\Mbar^{(1)}$ accomodates triples $(Y,\lambda,\dots)$ over characteristic $p$ varieties $S/\f_p^{ac}$, regardless of whether they satisfy something to the 
effect of $\rho_{\xi,\gr^+}\cong\rho^+\times_{L^+}L_{\gr^+}^+$ or not. The PEL Shimura varieties of this kind are associated to groups of unitary similitudes 
of a skew-Hermitian $L$-vector space, however an important technical point of our theory is that there is a certain flexibility of picking up the Kottwitz 
determinant condition in the definition of $M^{(1)}$, hence in the definition of the corresponding unitary group. For the precise description of the group, 
which we prefer to choose, we refer to section \ref{group}. The corresponding symmetric Hermitian domain $X^{(1)}$ is a product of $w=an$ copies of the 
$n-1$-dimensional complex ball, where in fact $n=\dim_LB'$. In terms of the skew-Hermitian form, with respect to which one forms the congruence subgroup of
unitary similitudes, this condition reads equivalently: there are $w$ real places of $L^+$ at which the form has one of the signatures 
$(n-1,1)$ or $(1,n-1)$, while at all the other real places it is positive definite or negative definite.\\ 
The whole point lies now in the very simple observation that unitary groups of forms with this particular signature, i.e. $(0,n)$, $(1,n-1)$, $(n-1,1)$ or $(n,0)$, have more than 
one symplectic representation with Hodge weight in the set $\{(-1,0),(0,-1)\}$, amongst them there is the standard representation that we denote by $g^{(1)}$, but there are 
further more exotic ones which will be denoted by $g^{(0)}$, $\dots$, $g^{(n)}$. These maps are needed in a very essential way in our argument and it is for this reason that 
we pick the particular Kottwitz condition alluded to above. The representations $g^{(k)}$ give highly non-trivial maps from a certain auxiliary covering PEL Shimura variety $M^{(0\times1)}\rightarrow M^{(1)}$ to Siegel space, see \cite[Introduction]{ball} or section \ref{exotic} of this work. In zero characteristic the map $g^{(k)}$ is obtained by 
applying the standard functoriality properties of canonical models, but in positive characteristic it is difficult to get and its construction occupies the bulk of chapter \ref{mock2}.\\ 
Once integral versions, and hence reductions $\pmod p$ of $g^{(k)}$ have been established, there are applications for finding upper bounds for the connected Zariski 
closures $G_{\xi,\gr}^\circ$ of the $\gr$-adic image of the $\pi_1^{et}(S,\xi)$-operation that corresponds to some given $S\rightarrow\Mbar^{(0\times1)}$, as we shall 
see shortly. Each $g^{(k)}$ goes with an abelian scheme $Y^{(k)}$ on $\M^{(0\times1)}$, obtained by pull-back of the universal one on Siegel space. Each $Y^{(k)}$ 
determines a $\pi_1^{et}(S,\xi)$-representation on the $L_\gr$-vector space $V_\gr Y_\xi^{(k)}=$ rational Tate module of $Y_\xi^{(k)}$. A slight variant of the result 
\cite[Lemma 2.2]{ball} shows the existence of a $\pi_1(S,\xi)$-equivariant isomorphism
$$T_\gr Y_\xi^{(k)}\cong{T_\gr Y_\xi^{(0)}}^{\otimes_{\O_{L_\gr}}1-k}\otimes_{\O_{L_\gr}}\bigwedge_{\O_{L_\gr}}^kT_\gr Y_\xi^{(1)}.$$
Suppose for a while we could prove that there exists a semisimple $L_\gr$-algebraic subgroup $H_\gr\subset\GL(V_\gr Y_\xi^{(1)}/L_\gr)$, which satisfies:
\begin{equation}
\label{mockk}
L_\gr\otimes_L\End_L^0(Y^{(k)}\times_{\M^{(0\times1)}}S)=\End_{H_\gr}V_\gr Y_\xi^{(k)},
\end{equation}
a representation theoretic lemma shows that the semisimple group $G_{\xi,\gr}^\circ$ is (almost always, see lemma \ref{R3}) contained in $H_\gr$. Let us call \eqref{mockk} 
a mock $H_\gr$-structure over the pointed variety $S$. Observe that \eqref{mockk} would not be very sensible unless the $\End_L^0(Y_\xi^{(k)})$'s are already quite large. Actually the base point $\xi\in S$ that our theory comes up with is lying in the ``most supersingular locus'' of the moduli space $\Mbar^{(0\times1)}$. Such points give rise to 
the largest possible endomorphism algebras $\End_L^0(Y_\xi^{(k)})\cong\Mat(\binom{n}{k},L)$, which gives a lot of room for adjustment when trying to build up $S$.\\
The next tool which our proof uses is the deformation theory of a ``most supersingular'' $\f_p^{ac}$-valued point $\xi$, say over the scheme $S=\Spec\f_p^{ac}[[t]]$. 
Ironically, the concept of $\ell$-adic monodromy is completely meaningless in this context because $\Spec\f_p^{ac}[[t]]$ is simply connected. However \eqref{mockk} 
is meaningful over arbitrary bases. Furthermore, as the functors $S\mapsto\End_L(Y^{(k)}\times_{\M^{(0\times1)}}S)$ are representable by schemes locally of finite type 
over $\M^{(0\times1)}$, any mock $H_\gr$-structure over $\Spec\f_p^{ac}[[t]]$ descends to a mock $H_\gr$-structure over some, possibly very large, finitely generated 
$\f_p^{ac}$-algebra contained in $\f_p^{ac}[[t]]$, and over this basis $\ell$-adic monodromy does have a lot of content.\\
To make this project work a lot of compatibilities have to be checked. Most of them are rather straightforward and can be deduced with a close analysis of certain exotic 
Hodge cycles in the style Blasius is doing in \cite{blasius}, but one of these compatibilities is more intricate: As $\Mat(\binom{n}{k},L)$ acts on the $\binom{n}{k}$-dimensional 
$L_\gr$-vector space $V_\gr Y_\xi^{(k)}$ we get, by Morita equivalence, an induced $L$-structure up to scalars in $L_\gr^\times$ on it. For our purposes it turns out to 
be significant that the above isomorphism between $V_\gr Y_\xi^{(k)}$ and ${V_\gr Y_\xi^{(0)}}^{\otimes_{L_\gr}1-k}\otimes_{L_\gr}\bigwedge_{L_\gr}^kV_\gr Y_\xi^{(k)}$, 
respects this $L$-structure up to scalars in $L_\gr^\times$, see lemma \ref{wsechs}. Therefore we have canonical identifications of $\End_L^0(Y_\xi^{(k)})$ 
with $\End_L(\bigwedge_L^kB)$, where $B$ is an appropriatly chosen $L$-vector space, allowing us to free the left hand side of condition \eqref{mockk} 
from the $L_\gr$-coefficients. Hence the task consists now of exhibiting a map of pointed $\f_p^{ac}$-schemes $S\rightarrow\Mbar^{(0\times1)}$ with:
\begin{equation}
\label{mock4}
\End_L(\bigwedge_L^kB)=\End_L^0(Y_\xi^{(k)})\supset\End_L^0(Y^{(k)}\times_{\M^{(0\times1)}}S)=\End_G(\bigwedge_L^kB)
\end{equation}
We sketched above that we want to use a $S=\Spec\f_p^{ac}[[t]]$-valued point over which the subalgebras
\begin{equation}
\label{mock5}
\End_L^0(\tilde Y_\xi^{(k)})=\End_L^0(Y_\xi^{(k)})\cap(\End_{L_\gq}^0(\tilde Y_\xi^{(k)}[\gq^\infty])\oplus\End_{L_{\gq^*}}^0(\tilde Y_\xi^{(k)}[\gq^{*\infty}]))
\end{equation}
coincide with the ones on the right hand side of \eqref{mock4}, where $Y^{(k)}\times_{\M^{(0\times1)}}\f_p^{ac}[[t]]=\tilde Y_\xi^{(k)}$. It goes without saying that 
we want to produce these deformations by means of the Serre-Tate theorem, but how do we finess the requested $p$-divisible groups $\tilde Y_\xi^{(1)}[\gq^\infty]$ 
over $\f_p^{ac}[[t]]$? This is cumbersome to achieve, and it is this part of our theory where the very unpleasant condition (X.2) appears, see chapter \ref{wsieben} 
and in particular subsection \ref{freakish} for details.\\
Observe once more that it is necessary to exercise a lot of care when computing the intersection in \eqref{mock5}. The reason is that to a deformation 
$\tilde Y_\xi^{(1)}[\gq^\infty]$ of $Y_\xi^{(1)}[\gq^\infty]$ one calculates easily the $L_\gq$-algebra $\End_{L_\gq}^0(\tilde Y_\xi^{(k)}[\gq^\infty])$ for all $k$, but difficulties 
show up when one wants to understand which elements of $\End_{L_\gq}^0(\tilde Y_\xi^{(k)}[\gq^\infty])\oplus\End_{L_{\gq^*}}^0(\tilde Y_\xi^{(k)}[\gq^{*\infty}])$ 
are in $\End_L^0(Y_\xi^{(k)})$. To accomplish this we need to know a crystalline version of lemma \ref{wsechs}: The $L$-structure up to 
$K(\f_p^{ac})^\times$-scalars of certain eigenspaces of the Dieudonn\'e modules of $Y_\xi^{(k)}[\gq^\infty]$ are compatible.

\subsubsection*{Related results:} Let $\psi$ be a non-trivial $\ell$-adic additive character of $\f_p$, and let $\chi$ be the Legendre symbol. By 
push-out there are associated sheaves $\L_{\psi(t)}$ and $\L_{\chi(t)}$ on the affine $t$-line $\a^1$ over say $\f_p^{ac}$ and on $\g_m\stackrel{j}
{\hookrightarrow}\a^1$. Then the Fourier transform $\F=NFT_\psi(\L_{\psi(t^7)}\otimes_{\q_\ell^{ac}}j_*\L_{\chi(t)})$ on $\a^1$ is lisse and pure of rank 
seven and weight one. By a fascinating result of Katz its geometric monodromy group is $G_2$, provided that $p\notin\{2,3,7,13\}$, cf. 
\cite[Chapter 9]{katz3},\cite[section 4]{katz4}. Here observe that the pendant of our field $L$ is $\q[\exp\frac{2i\pi}{p}]$. Hence the abelian scheme over 
$\a^1$ corresponding to $\F$ has dimension $7\times\frac{p-1}{2}$.\\ 
Oort and van der Put succeeded in constructing $5$-dimensional abelian varieties $Y/\f_p^{ac}((t))$, such that the endomorphism algebra of
$Y\times_{\f_p^{ac}((t))}\f_p^{ac}((t))^{sep}$ is equal to the quaternion algebra over $\q$ which is non-split only at $p$ and $\infty$, again a phenomenon which 
does not exist in characteristic $0$, cf.\cite[Example 1.5.1]{oort1}. In fact their method of construction inspired much of this work intuitively: We start out from a 
base point in $\Mbar^{(0\times1)}$ that is contained in the closed subset that plays the role of the ``most supersingular locus'', the basic points in the sense of 
\cite[Proposition 1.12]{richartz}. Oort and van der Put start from a base point too, however they use a multiplicatively degenerated abelian variety instead. We also 
adopt Oort and van der Put's use of deformation theoretic methods over bases like $\f_p^{ac}[[t]]$ and $\f_p^{ac}((t))$ to manufacture families with a prescribed 
``degree of individuality'', in their case provided by a clever use of the Mumford-Faltings-Chai construction and in this work provided by the Serre-Tate theorem.

\subsubsection*{Acknowledgements:} A major part of this work was written up during a visit to the university of Rennes in spring 2002. I thank Prof. P. Berthelot 
very heartily for the providing of this opportunity. Further thanks go to: Prof. R. Coleman, K. Buzzard, M. Dettweiler, Prof. S. Edixhoven, Prof. E. Freitag, Prof. N. Katz, 
Prof. M. Kisin, Prof. W. Kohnen, Prof. P. Lochak, A. Miller, Prof. R. Noot, F. Paugam, Prof. R. Taylor, Prof. O. Venjakob, M. Volkov, T. Wedhorn, Prof. R. Weissauer, 
and S. Wortmann for general interest and comments. I also thank the audience of my talk in the Oberwolfach-workshop 'Arithmetic Algebraic Geometry' in 2004, in 
which I announced the corollary \ref{R5}.

\section{$\sy$-Structures} 
\label{wsieben}
In this chapter we introduce certain $p$-divisible groups which, as we will prove much later, are essentially the 
$p$-divisible groups of the abelian varieties in theorem \ref{newnull}. We study first properties of an equicharacteristic 
deformation, with additional structure of a very particular kind, which we call a sufficient deformation of a $\sy$-structure.

\subsection{Windows and Frames}

Let $k$ be a perfect field of characteristic $p$. Let $W(k)$ be its Wittring, let $K(k)$ be its fraction field, and denote the absolute Frobenius on these by $\tau$. 

\subsubsection{Graded slopes}\label{slop} By an isocrystal we mean a finite-dimensional $K(k)$-vector space $M$ together with a $\tau$-linear bijection $\phi:M\rightarrow M$. Let us briefly recall slopes: There exists a positive integer $d$ and a basis $x_i$ of $K(k^{ac})\otimes_{K(k)}M$ such that $\phi^d(x_i)=p^{dz_i}x_i$ with certain uniquely determined rationals $z_i\in\frac{1}{d}\z$ which are called the slopes of $M$. If $M$ has only one single slope it is called isoclinal, if that slope is $0$ it is called \'etale.\\
By an effective isocrystal we mean a free $W(k)$-module $M$ with map $\phi$, such that $\q\otimes M$ is an isocrystal. A Dieudonn\'e module is by definition an effective isocrystal that satisfies $pM\subset\phi M$, or equivalently one that has an additional $\tau^{-1}$-linear operator $\psi$ with $\psi\circ\phi=\phi\circ\psi=p$. The slopes of an effective isocrystal are always non-negative. The slopes of a Dieudonn\'e module are always in the interval $[0,1]$ and if $1$ is omitted the Dieudonn\'e module is called unipotent. If $M=\d^*(\Gbar)$ is the contravariant Dieudonn\'e module of a $p$-divisible group $\Gbar/k$ as explained in section \ref{contrava} below, then the slopes of $M$ are just the slopes of $\Gbar$ in the usual sense.\\ 
Let $r$ be a positive integer. A $\z/r\z$-grading on an isocrystal, effective isocrystal, or Dieudonn\'e-module $M$ is a 
decomposition as a direct sum $M=\bigoplus_{\sigma\in\z/r\z}M_\sigma$ with $\phi M_\sigma\subset M_{\sigma+1}$. Notice that the 
functor: $\{\z/r\z$-graded isocrystals$\}\rightarrow\{$finite-dimensional $K(k)$-vector spaces with a $\tau^r$-linear bijection$\}$ that takes a pair 
$(\bigoplus_{\sigma\in\z/r\z}M_\sigma,\phi)$ to $(M',\phi')=(M_0,\phi^r)$ is an equivalence of categories. In the other direction we start with $(M',\phi')$ and we define: 
\begin{equation*}
M_\sigma=K(k)\otimes_{\tau^{\tilde\sigma},K(k)}M',
\end{equation*}
where $\tilde\sigma$ is the representative of $\sigma$ in $\{0,\dots,r-1\}$, and putting:
\begin{equation*}
\phi(\alpha\otimes_{\tau^{\tilde\sigma},K(k)}x)=\begin{cases}\tau(\alpha)\otimes_{\tau^{\tilde\sigma+1},K(k)}x&\sigma\neq-1\\\tau(\alpha)\phi'(x)&\sigma=-1\end{cases}.
\end{equation*}
gives $M_\sigma$ the structure of a graded isocrystal. Observe however, that effective graded isocrystals cannot be recovered from the 
degree $0$ submodule because the maps $\phi:M_\sigma\rightarrow M_{\sigma+1}$ might not be surjective.\\
If $M$ is graded then the multiplicities $m_1,m_2,m_3,\dots$ of the slopes $z_1,z_2,\dots$ are all divisible by $r$. 
We call $\frac{m_1}{r}\times rz_1,\frac{m_2}{r}\times rz_2,\frac{m_3}{r}\times\dots$ the graded slopes of $M$, equivalently 
the graded slopes $rz_i$ are just the $p$-adic valuations of the eigenvalues of the operator $\phi'$ on $M'$, counted 
with multiplicity. The graded slopes of a graded (unipotent) Dieudonn\'e module are always in the interval $[0,r]$ ($[0,r[$).\\
If $k$ is algebraically closed and if all the graded slopes of an isocrystal $M_\sigma$ are equal to some integer $z$, 
then the following recipe will be used time and again: Consider the $K(\f_{p^r})$-vector space (``Li's skeleton''):
\begin{equation*}
I=\{x\in\q\otimes M_0|\phi^r(x)=p^zx\}.
\end{equation*}
The functor $M_\sigma\mapsto I$ defines again an equivalence of categories 
$\{\z/r\z$-graded isocrystals of graded slope $z\}\rightarrow\{$finite-dimensional $K(\f_{p^r})$-vector spaces$\}$. In particular the graded endomorphism algebra 
of a $M_\sigma$ that involves integral graded slopes only is a product of Matrixalgebras over $K(\f_{p^r})$ then. For non-integral graded slopes this does not hold.

\subsubsection{Generalities} For the following we have to fix a frame 
$A$ in the sense of \cite{zink1}, the precise definition is irrelevant to us, as we only have to know that the following are examples of frames:
\begin{itemize}
\item
$A=W(k)[[t]]$
\item
$A=W(k)((t))\otd\z_p$
\item
$A=W(k)$
\end{itemize}
Write $R$ for $A/pA$ being one of $k[[t]]$, $k((t))$, or $k$. Observe that the absolute Frobenius on $R$ has a natural lift $\tau$ to each of the 
above three frames by sending $t$ to $t^p$. We also need to consider a notion of grading in Zink's category of $A$-windows, as introduced in \cite{zink1}:

\begin{defn}
A $\z/r\z$-graded $A$-window is a finitely generated projective graded 
$A$-module $\bigoplus_{\sigma\in\z/r\z}M_\sigma$, together with a graded 
submodule $M_1=\bigoplus_{\sigma\in\z/r\z}M_{\sigma,1}$, and a $\tau$-linear 
operator $\phi:M_\sigma\rightarrow M_{\sigma+1}$ such that
\begin{itemize}
\item[(W.1)]
$pM_\sigma\subset M_{\sigma,1}$ and the $R$-module 
$M_\sigma/M_{\sigma,1}$ is projective,
\item[(W.2)]
$\phi M_{\sigma,1}\subset pM_{\sigma+1}$,
\item[(W.3)]
the $A$-module $M_{\sigma+1}$ is generated by $\frac{1}{p}\phi M_{\sigma,1}$, 
\item[(W.4)]
the operator $\psi^\sharp:M_\sigma\rightarrow A\otimes_{\tau,A}M_{\sigma-1}$ 
described below is nilpotent $\pmod p$.
\end{itemize}
\end{defn}

The $\tau$-linear map $M_1\ni x\mapsto\frac{1}{p}\phi(x)\in M$ will be denoted by $\phi_1$. By $A$-linearization of $\phi$ and $\phi_1$ we obtain 
$\phi^\sharp:A\otimes_{\tau,A}M\rightarrow M$, and $\phi_1^\sharp:A\otimes_{\tau,A}M_1\rightarrow M$. By a graded normal decomposition we mean 
lifts $L_\sigma\subset M_\sigma$ of $M_{\sigma,1}$ together with complements $T_\sigma\subset M_\sigma$ to $L_\sigma$. To any graded normal 
decomposition one obtains an isomorphism $U^\sharp:A\otimes_{\tau,A}M_\sigma\rightarrow M_{\sigma+1}$, where $U$ is defined by 
$L_\sigma\oplus T_\sigma\ni l+t\mapsto\phi_1(l)+\phi(t)$. In loc.cit. Zink defines the operator $\psi^\sharp$ to be $U^{\sharp-1}$ followed by the 
map $l+t\mapsto l+pt$. It satisfies $\psi^\sharp\circ\phi^\sharp=\phi^\sharp\circ\psi^\sharp=p$ and does not depend on the choice of the (graded) 
normal decomposition. If $M$ is graded then $\psi^\sharp$ is homogeneous of degree $-1$, i.e. maps $M_{\sigma+1}$ to $A\otimes_{\tau,A}M_\sigma$.

\begin{rem}
\label{E6}
A standard technique allows to deform a $\z/r\z$-graded Dieudonn\'e module
$M_\sigma$ as follows: Write $\tilde M_\sigma$ for 
$W(k)[[t]]\otimes_{W(k)}M_\sigma$ and pick arbitrary $W(k)[[t]]$-linear maps 
$u_\sigma:\tilde M_\sigma\rightarrow\tilde M_\sigma$ with 
$u_\sigma\equiv 1_{M_\sigma}\pmod t$. Define a $\tau$-linear homogeneous map 
$\tilde\phi:\tilde M_{\sigma-1}\rightarrow\tilde M_\sigma $ by 
$x\mapsto u_\sigma(\phi(x))$. Corresponding to our Frobenius lift $\tau$ there 
exists a unique homomorphism $\delta:W(k)[[t]]\rightarrow W(W(k)[[t]])$, 
called the Cartier map, such that an element $\alpha$ is sent to a Witt vector
of which the $n$th ghost coordinate is equal to $\tau^n(\alpha)$. The element 
$t$, for instance, is sent to its Teichmuller lift $[t]=(t,0,0,\dots)$. Write:
\begin{equation}
\label{car}
\varkappa:W(k)[[t]]\rightarrow W(k[[t]]).
\end{equation}
for the composition of $\delta$ with the map to $W(k[[t]])$ which is 
$\pmod p$-reduction of the Witt components. By the theory in 
\cite[section (2.2)]{katz1} the object 
$W(k[[t]]^{perf})\otimes_{\varkappa,W(k)[[t]]}\tilde M_\sigma$, that one gets 
after extension of scalars to $k[[t]]^{perf}$, describes a ``locally free 
finite rank Frobenius crystal on $k[[t]]^{perf}$'', in our case with 
$\z/r\z$-grading.\\ 
The theory of Zink allows an improvement: Suppose that decreeing 
$\tilde M_{\sigma,1}:=W(k)[[t]]\otimes_{W(k)}M_{\sigma,1}$, where 
$M_{\sigma,1}:=\{x\in M_\sigma|\phi(x)\in pM_{\sigma+1}\}=\psi M_{\sigma+1}$,
makes the structure $(\tilde M_\sigma,\tilde M_{\sigma,1},\tilde\phi)$ into a 
graded $W(k)[[t]]$-window. This boils down to the nilpotence condition (W.4), 
which is satisfied if and only if the Dieudonn\'e module $W(k((t))^{perf})
\otimes_{\varkappa,W(k)[[t]]}\tilde M_\sigma$ is unipotent. Now one can use 
$(\tilde M_\sigma,\tilde M_{\sigma,1},\tilde\phi)$ to describe a ``locally 
free finite rank $\z/r\z$-graded Frobenius crystal on the ring $k[[t]]$'', not 
merely on the perfection. This is an important aspect as passage from $k[[t]]$
to $k[[t]]^{perf}$ destroys many interesting properties of the ``locally free 
finite rank Frobenius crystals''.
\end{rem}

\subsection{Formal $\sy$-structures}
We need the following integral version of a special case of an ``additional structure'' on a window. In the setting of isocrystals this concept has already been studied 
in much greater generality, see \cite{richartz} for example. Fix once and for all a finite sequence $b_1,\dots,b_c$ of non-negative integers, not all of which are zero.
\begin{defn}
\label{symD}
A $\z/r\z$-graded $A$-window $M_\sigma$ is endowed with a formal $\sy$-structure if a tuple $(N_\sigma,N_{i,\sigma},\zeta_\sigma)$, where $i\in\{1,\dots,c\}$, is given such that:
\begin{itemize}
\item[(S.1)]
$N_\sigma$ is a $\z/r\z$-graded $A$-window with $\rk_AN_\sigma=2$.
\item[(S.2)]
$N_{1,\sigma},\dots,N_{c,\sigma}$ are some further auxiliary $\z/r\z$-graded effective isocrystals which have $\rk_{W(k)}N_{i,\sigma}=1$.
\item[(S.3)]
$\zeta_\sigma:\q\otimes M_\sigma\rightarrow\q\otimes\bigoplus_{i=1}^cN_{i,\sigma}\otimes_{W(k)}\sy_A^{b_i}N_\sigma$
is a $\z/r\z$-homogeneous quasi-isogeny of $A$-modules with $\phi$-operation.
\item[(S.4)]
For all $\sigma\in\z/r\z$ with $\rk_RN_\sigma/N_{\sigma,1}=1$ the $A$-module $\zeta_\sigma(M_\sigma)$ 
is equal to $\bigoplus_{i=1}^cN_{i,\sigma}\otimes_{W(k)}\sy_A^{b_i}N_\sigma$.
\end{itemize}
\end{defn}
If $N_\sigma$ (resp. $N_{i,\sigma}$) is as in (S.1) (resp. (S.2)) we write $z$ (resp. $z_i$) for the graded slope of $\bigoplus_\sigma\det(N_\sigma)$ 
(resp. $\bigoplus_\sigma N_{i,\sigma}$). We denote the set $\{\sigma\in\z/r\z|\rk_RN_\sigma/N_{\sigma,1}=1\}$ by $\Sigma$, its cardinality satisfies 
$0\leq z-\Card(\Sigma)\equiv0\pmod2$.\\ 
Observe that any $W(k)[[t]]$-window gives rise to a $W(k)((t))\otd\z_p$-and a $W(k)$-one by extension of scalars, these will be referred to as the 
generic and the special fibre. It is clear that a $\sy$-structure on some graded $W(k)[[t]]$-window is inherited by its generic and special fibre.

\subsubsection{Examples of formal $\sy$-structures} 
\label{freakish}
For the rest of this section the frame is $A=W(k)$, we begin with a numerical condition which guarantees the existence of an 
abundance of examples of formal $\sy$-structures. We fix $z,r\in\z$, and $a\in\frac{1}{2}\z$ all positive and subject to:
\begin{itemize}
\item
$n:=\sum_{i=1}^c1+b_i<\frac{r}{a}$
\item
$\max\{b_i|i=1,\dots,c\}\leq\frac{2a}{z}$
\item
$a\in\z$, if $z$ is even and $2a\equiv b_1\equiv\dots\equiv b_c\pmod2$, if $z$ is odd
\end{itemize}
If the above holds it is possible to choose the following:
\begin{itemize}
\item[(I)]
for every index $i\in\{1,\dots,c\}$ a representation of $a-b_iz/2$ as an arbitrary sum $\sum_{j=1}^zf_{i,j}$, where the $f_{i,j}$ are non-negative integers,
\item[(II)]
integers $\sigma_1<\dots<\sigma_z<\sigma_1+r$, and a proper subset $\Omega\subset\z/r\z$ such that its intersection with each of the intervals
$\{\sigma_j,\dots,\sigma_{j+1}-1\}$ (along with $\{\sigma_z,\dots,\sigma_1+r-1\}$ of course) contains exactly $w_j$ many elements where
$$w_j=\frac{1}{2}\sum_{i=1}^c(b_i+1)(b_i+2f_{i,j}).$$
\end{itemize}
Let $w$ be $\sum_{j=1}^zw_j$. Observe that $\Card(\Omega)=w=an<r$. We want to start with a graded Dieudonn\'e module $N_\sigma$, such that
\begin{eqnarray*}
&&\dim_kN_\sigma/N_{\sigma,1}=\begin{cases}1&\sigma\equiv\sigma_j\pmod r\\2&\text{ otherwise}\end{cases}\\
&&\rk_{W(k)}N_\sigma=2,
\end{eqnarray*}
where, as usual, $N_{\sigma,1}=\{x\in N_\sigma|\phi(x)\in pN_{\sigma+1}\}$. Notice that these conditions imply $\Sigma=\{\sigma_1,\dots,\sigma_z\}$, 
in particular $\Card(\Sigma)=z$. Let $N_{\sigma_j}=L_j\oplus T_j$ ($j=1,\dots,z$) be a graded normal decomposition. Let us 
endow the graded $W(k)$-modules $N_{i,\sigma}=W(k)$ with the structure of an effective isocrystal by the following formula:
$$\phi(x)=\begin{cases}p^{f_{i,j}}\tau(x)&\sigma\equiv\sigma_j\pmod r\\\tau(x)&\text{ otherwise}\end{cases},$$
where $x$ is to be regarded in $N_{i,\sigma}$, and $\phi(x)$ is to be regarded in $N_{i,\sigma+1}$. We next want to establish a graded Dieudonn\'e 
module $M_\sigma\subset\bigoplus_{i=1}^cN_{i,\sigma}\otimes_{W(k)}\sy_{W(k)}^{b_i}N_\sigma$. Moreover, we want $M_\sigma$ to satisfy:
\begin{eqnarray*}
&&\dim_kM_\sigma/M_{\sigma,1}=\begin{cases}n-1&\sigma\in\Omega\\n&\text{ otherwise}\end{cases}\\
&&\rk_{W(k)}M_\sigma=n
\end{eqnarray*}
We start with an element $\sigma_j\in\Sigma$ and set 
\begin{eqnarray*}
&&M_{\sigma_j}=\bigoplus_{i=1}^cN_{i,\sigma_j}\otimes_{W(k)}\sy_{W(k)}^{b_i}N_{\sigma_j}\\
&&=\bigoplus_{i=1}^cN_{i,\sigma_j}\otimes_{W(k)}(\bigoplus_{b=0}^{b_i} L_j^{\otimes b}\otimes_{W(k)}T_j^{\otimes b_i-b}).
\end{eqnarray*}
The lattice $M_{\sigma'}$ for some $\sigma'\in\{\sigma_j,\dots,\sigma_{j+1}\}$ will be manufactured from
$$\bigoplus_{i=1}^cN_{i,\sigma_j}\otimes_{W(k)}(\bigoplus_{b=0}^{b_i}p^{-e_{b,i}} L_j^{\otimes b}\otimes_{W(k)}T_j^{\otimes b_i-b})$$
by applying $\phi^{\sigma'-\sigma_j}$ with a choice of integers $e_{b,i}$ 
satisfying
\begin{eqnarray*}
&&0\leq e_{b,i}\leq b+f_{i,j}\\
&&e_{b,i}\leq e_{b+1,i}\leq e_{b,i}+1.
\end{eqnarray*}
We proceed inductively starting with $e_{0,i}=\dots=e_{b_i,i}=0$ at $\sigma'=\sigma_j$. If $\sigma'-1$ is not in $\Omega$, we keep all of $e_{0,i},\dots,e_{b_i,i}$, and otherwise we increase any of $e_{0,i},\dots,e_{b_i,i}$ by one. If we specialize the construction to $\sigma'=\sigma_{j+1}$ we find that the integers $e_{b,i}$ reach the maximal value 
$b+f_{i,j}$, due to $\sum_{i=1}^c\sum_{b=0}^{b_i}b+f_{i,j}=w_j$ and (II). The result is that $M_{\sigma_{j+1}}$ is the image under $\phi^{\sigma_{j+1}-\sigma_j}$ of the lattice
$$\bigoplus_{i=1}^cp^{-f_{i,j}}N_{i,\sigma_j}\otimes_{W(k)}(\bigoplus_{b=0}^{b_i}(p^{-1}L_j)^{\otimes b}\otimes_{W(k)}T_j^{\otimes b_i-b}).$$
In order to guarantee the consistency of the construction we have to check that $M_{\sigma_{j+1}}$ agrees with 
$\bigoplus_{i=1}^cN_{i,\sigma_{j+1}}\otimes_{W(k)}(\sy_{W(k)}^{b_i}N_{\sigma_{j+1}})$. This follows easily from $\phi(p^{-1}L_j\oplus T_j)=M_{\sigma_j+1}$, 
and $\phi(p^{-f_{i,j}}N_{i,\sigma_j})=N_{i,\sigma_j+1}$, together with the fact that $\phi$ induces bijections $N_{\sigma'}\rightarrow N_{\sigma'+1}$ and 
$N_{i,\sigma'}\rightarrow N_{i,\sigma'+1}$, whenever those degrees $\sigma'$ stay in the range from $\sigma_j+1$ to $\sigma_{j+1}-1$, i.e. outside $\Sigma$.

\begin{rem}
\label{E17}
Let $N_\sigma$ be as above and let $\{z',z''\}$ be the graded slopes of it. According to \cite[Theorem 4.2]{richartz} they satisfy:
\begin{itemize}
\item[(N.1)]
$z'+z''=z$,
\item[(N.2)]
$z',z''\in\{\frac{z}{2}\}\cup\{0,1,\dots,z\}$.
\end{itemize}
Moreover, for fixed $\Sigma$ every $\{z',z''\}$ with the above properties arises from some $N_\sigma$, see \cite[Theorem 5.1]{kottwitz3} for a more general result.
\end{rem}

We finish this subsection with a lemma on slopes:
 
\begin{lem}
\label{slope}
Suppose that $M_\sigma$ has a $\sy$-structure as above, and let 
$\bigoplus_\sigma N_\sigma$ have the graded slopes $\{z',z''\}$. Then 
$\bigoplus_\sigma M_\sigma$ has graded slopes $z_i+bz'+(b_i-b)z''$ where $i$ 
runs through $\{1,\dots,c\}$, and $b$ runs through $\{0,\dots,b_i\}$. The 
$W(k)$-rank of $M_\sigma$ is equal to $\sum_{i=1}^c1+b_i$.
\end{lem}
\begin{proof}
We can assume that $k$ is algebraically closed. Let $\q\otimes N_0$, and 
$\q\otimes N_{i,0}$ have bases $\{x',x''\}$ and $\{x_i\}$, with 
$\phi^{rd}x'=p^{z'd}x'$, $\phi^{rd}x''=p^{z''d}x''$, and 
$\phi^{rd}x_i=p^{z_id}x_i$. Now just note that 
$\q\otimes\bigoplus_{i=1}^cN_{i,0}\otimes_{W(k)}\sy_{W(k)}^{b_i}N_0$ has 
a basis consisting of the elements $\{x_i{x'}^b{x''}^{b_i-b}\}$ and that
$\phi^r(x_i)\phi^r(x')^b\phi^r(x'')^{b_i-b}$ is 
$\phi^r(x_i{x'}^b{x''}^{b_i-b})$. We get a $\sum_{i=1}^c1+b_i$-element basis 
of $\q\otimes M_0$ on which, by (S.3), $\phi^{rd}$ has eigenvalues with the 
valuations $d(z_i+bz'+(b_i-b)z'')$.
\end{proof}

\begin{rem}
\label{neweins}
Assume that $\bigoplus_\sigma M_\sigma$ has pure graded slope $a$. Then 
$\bigoplus_\sigma N_\sigma$ is isoclinal too, and necessarily has pure graded 
slope $z/2$, by (N.1). Consequently the integers $z_i$ have to be equal to 
$a-b_iz/2$. In particular $a$ must always be in $\frac{1}{2}\z$ and moreover 
one has:
\begin{itemize}
\item
$a\in\z$, if $z$ is even
\item
$2a\equiv b_1\equiv\dots\equiv b_c\pmod2$, if $z$ is odd.
\end{itemize}
Notice that in the first case all graded slopes are integral.
\end{rem}

\subsubsection{Results of Wedhorn}\label{bizarre} Let $M_\sigma$ be a $\z/r\z$-graded Dieudonn\'e module equipped 
with a $\sy$-structure $(N_\sigma,N_{i,\sigma},\zeta_\sigma)$. We want to exhibit lifts to non-isotrivial $W(k)[[t]]$-windows 
which preserve the $\sy$-structure. We begin with a description of a lift of $N_\sigma$ where we follow work of Wedhorn. If
\begin{itemize}
\item
$\dim_kN_\sigma/N_{\sigma,1}\geq1$, 
\item
and all graded slopes of $N_\sigma$ are $\neq0$,
\end{itemize}
then, by \cite[Proposition 4.1.4]{wedhorn}, there exists a deformation sequence. This is a family of 
elements $e_\sigma\in N_\sigma-N_{\sigma,1}$ such that for every $\sigma$ one and only one of
\begin{itemize}
\item
$\phi(e_{\sigma-1})\equiv e_\sigma\pmod p$
\item
$\phi(e_{\sigma-1})\in N_{\sigma,1}$
\end{itemize}
holds. Observe that if the second alternative holds then necessarily $\sigma\in\Sigma$, and $\{e_\sigma,\phi(e_{\sigma-1})\}$ 
is a basis of the $W(k)$-module $N_\sigma$. Now define $W(k)$-linear maps $v_\sigma:N_\sigma\rightarrow N_\sigma$ by:
\begin{itemize}
\item
if $\phi(e_{\sigma-1})\equiv e_\sigma\pmod p$, then $v_\sigma=0$
\item
if $\phi(e_{\sigma-1})\in N_{\sigma,1}$ then
\begin{eqnarray*}
v_\sigma(\phi(e_{\sigma-1}))=&e_\sigma&\\
v_\sigma(e_\sigma)=&0&
\end{eqnarray*}
\end{itemize}
We proceed as follows: The $W(k)[[t]]$-modules $\tilde M_\sigma$, $\tilde M_{\sigma,1}$, $\tilde N_\sigma$, and $\tilde N_{\sigma,1}$ 
are obtained by scalar extension, as in remark \ref{E6}. Now let $x$ be an element in one of $\tilde N_{\sigma-1}$ or $\tilde M_{\sigma-1}$. 
To define its image $\tilde\phi(x)$ under the $\tau$-linear operator, we distinguish two cases. If $\sigma\in\Sigma$ we set:
$$\tilde\phi(x)=u_\sigma(\phi(x)),$$
where $u_\sigma$ shall act $W(k)[[t]]$-linearly on $\tilde N_\sigma$ according to the formula 
$u_\sigma=1_{\tilde N_\sigma}+t\otimes_{W(k)}v_\sigma$, with $v_\sigma$ as above, and $u_\sigma$ shall act on $\tilde M_\sigma$ according to
$$u_\sigma(\zeta_\sigma^{-1}(x_i{x'}^b{x''}^{b_i-b}))=\zeta_\sigma^{-1}(x_iu_\sigma(x')^bu_\sigma(x'')^{b_i-b})$$
for any bases $\{x',x''\}$ of $N_\sigma$ and $x_i\in N_{i,\sigma}-pN_{i,\sigma}$. This is meaningful because of condition 
(S.4). If, however, $x$ is in one of $\tilde N_{\sigma-1}$ or $\tilde M_{\sigma-1}$ with $\sigma\notin\Sigma$ we set:
$$\tilde\phi(x)=\phi(x).$$
In so doing we obtain a very particular deformation of the graded Dieudonn\'e module $M_\sigma$, which we 
will refer to as a sufficient deformation with $\sy$-structure $\zeta_\sigma$. Let us state thoroughly what this shows:

\begin{lem}
\label{suff}
Let $(N_\sigma,N_{i,\sigma},\zeta_\sigma)$ be a $\sy$-structure on a $\z/r\z$-graded Dieudonn\'e module $M_\sigma$. Assume 
$\dim_kN_\sigma/N_{\sigma,1}\geq1$ for all $\sigma$, and assume that the graded slopes of $\bigoplus_\sigma N_\sigma$ are non-zero. Then one has:
\begin{itemize}
\item[(i)]
If $\bigoplus_\sigma\tilde N_\sigma$ is a deformation corresponding to a deformation sequence, then the graded slopes of its generic fibre are $\{0,z\}$.
\item[(ii)]
If there exists a $\sigma_0$ with $M_{\sigma_0,1}=pM_{\sigma_0}$ then the sufficient deformation $\tilde M_\sigma$, that corresponds to (i) 
satisfies the nilpotence condition (W.4) and is therefore a graded $W(k)[[t]]$-window with $\sy$-structure in the sense of definition \ref{symD}.
\end{itemize}
\end{lem}
\begin{proof}
In the generic fibre of $\bigoplus_\sigma\tilde N_\sigma$ the graded slope $0$ must occur, due to \cite[Proposition 4.1.5]{wedhorn} (the image of 
$e_\sigma$ in $\tilde N_\sigma/\tilde N_{\sigma,1}$ is an eigenvector of $\tilde\phi^r$ with eigenvalue equal to $t$ to the power of number of 
occurences of $\phi(e_{\sigma-1})\in N_{\sigma,1}$). By refering to (N.1), we find that the other graded slope has to be $z$. This shows (i).\\

In order to justify (ii) notice that the construction of the underlying $W(k)[[t]]$-modules 
$\tilde M_\sigma$ and $\tilde M_{\sigma,1}$ shows $\tilde M_{\sigma_0,1}=p\tilde M_{\sigma_0}$. This means that 
$\psi^\sharp:\tilde M_{\sigma_0+1}\rightarrow W(k)[[t]]\otimes_{\tau,W(k)[[t]]}\tilde M_{\sigma_0}$ is divisible by $p$, implying the nilpotence condition (W.4).
\end{proof}

\begin{rem}
Assume that $\bigoplus_\sigma M_\sigma$ has pure graded slope $a$. Then, the graded slope of $\bigoplus_\sigma N_{i,\sigma}$ is 
$z_i=a-b_iz/2$, according to remark \ref{neweins}. As the slopes of the generic fibre of $\bigoplus_\sigma\tilde N_\sigma$ are $\{0,z\}$ we can 
use lemma \ref{slope} to deduce that the slopes of the generic fibre of $\bigoplus_\sigma\tilde M_\sigma$ are the numbers $a-b_iz/2,\dots,a+b_iz/2$ 
where $i$ runs through $\{1,\dots,c\}$. Observe that this implies the inequality
$$\max\{b_i|i=1,\dots,c\}\leq\frac{2\min\{a,r-a\}}{z},$$
which was also present in subsection \ref{freakish}.
\end{rem}

Therefore we get examples of $\sy$-structures on $W(k)[[t]]$-windows with non-constant Newton polygons. Observe that the non-constancy of the Newton polygon 
implies that the window is not isotrivial. In fact, the converse holds also, provided only that the special fibre is isoclinal, see \cite[Theorem 2.7.1]{katz1} for a proof of 
this in the language of ``locally free finite rank Frobenius crystals over $k[[t]]$''. In the next section we give a precise measure for the non-isotriviality of $\tilde M_\sigma$.

\subsection{Results of Dwork} 
\label{vdrei}
In order to state the main result of this subsection we employ connections. Let $A$ be one of the frames $W(k)[[t]]$ or $W(k)((t))\otd\z_p$. If $M$ is a 
finitely generated projective $A$-module, then a quasi-connection is a $K(k)$-linear map $\nabla:\q\otimes M\rightarrow\q\otimes M$ that satisfies 
\begin{equation}
\label{der}
\nabla(\alpha x)=\frac{\partial\alpha}{\partial t}x+\alpha\nabla(x),
\end{equation}
and $\nabla$ is called a connection if it preserves $M$. Recall that the sets of quasi-connections and connections form 
principal homogeneous spaces under $\End_A^0(M)$ and $\End_A(M)$. Recall also that a connection is called nilpotent 
if some power kills $M/pM$. The following result seems to be well known, for lack of reference we provide a proof:

\begin{lem}
\label{das}
Let $(M,M_1,\phi)$ be a $A$-window. Then there exists one and only one quasi-connection $\nabla$ with
\begin{equation}
\label{die}
\nabla(\phi(x))=pt^{p-1}\phi(\nabla(x)),
\end{equation}
moreover $\nabla$ is a connection, it is nilpotent and, if the $A$-window has a $\z/r\z$-grading, then $\nabla$ preserves this.
\end{lem}
\begin{proof}
Suppose that $\nabla$ is just any quasi-connection. Let us define a quasi-connection $\nabla^\tau$ on the $A$-module $A\otimes_{\tau,A}M$ by the formula
\begin{equation*}
\nabla^\tau(\alpha\otimes_{\tau,A}x)=\frac{\partial\alpha}{\partial t}\otimes_{\tau,A}x+pt^{p-1}\alpha\otimes_{\tau,A}\nabla(x).
\end{equation*}
Now we use $\phi^\sharp:A\otimes_{\tau,A}M\rightarrow M$ to get a quasi-connection $F(\nabla)$ on $M$ by transport of structure, i.e. 
$F(\nabla)=\phi^\sharp\circ\nabla^\tau\circ{\phi^\sharp}^{-1}$, here notice that the existence of $\psi^\sharp$ implies that ${\phi^\sharp}^{-1}$, 
is a well-defined map from $\q\otimes M$ to $\q\otimes A\otimes_{\tau,A}M$. It is easy to see that the Leibniz-formula \eqref{der} holds 
for $F(\nabla)$. We claim that $F(\nabla)$ preserves $M$, provided that $\nabla$ preserves it, i.e. if $\nabla$ is in fact a connection: 
indeed by (W.3) it is enough to verify this for elements of the form $\alpha\phi_1(x)$, where $\alpha\in A$ and $x\in M_1$. Notice that
\begin{eqnarray*}
&&\phi^\sharp\circ\nabla^\tau\circ{\phi^\sharp}^{-1}(\alpha\phi_1(x))=\\
&&\phi^\sharp\circ\nabla^\tau(\alpha\otimes_{\tau,A}\frac{x}{p})=\\
&&\phi^\sharp(\frac{\partial\alpha}{\partial t}\otimes_{\tau,A}\frac{x}{p}+
\alpha t^{p-1}\otimes_{\tau,A}\nabla(x))=\\
&&\frac{\partial\alpha}{\partial t}\phi_1(x)+
\alpha t^{p-1}\phi(\nabla(x)),
\end{eqnarray*}
showing $F(\nabla)(M)\subset M$. Observe finally that for a $A$-linear map $u:\q\otimes M\rightarrow\q\otimes M$ we 
get $F(\nabla+u)=F(\nabla)+\phi^\sharp\circ t^{p-1}\otimes_{\tau,A}u\circ\psi^\sharp$. We deduce that $F$ is a contractive 
map for the $p$-adic topology, because $\psi^\sharp$ is $p$-adically nilpotent. As $\End_A^0(M)$ and $\End_A(M)$ are 
both $p$-adically complete, it follows that in the set of quasi-connections and connections $F$ has a unique fixed point.\\

The nilpotency of $\nabla$ follows from $\nabla^p(M)\subset A\phi M$, and $\nabla^p(A\phi M)\subset pM$
\end{proof}

We call the connection on a $A$-window fulfilling equation \eqref{die} the Dieudonn\'e connection. 

\subsubsection{Base change to $K(k)\{\{t\}\}$} We next invoke Dwork's trivialisation of a 
graded window $\tilde M_\sigma$, to this end we introduce the $W(k)[[t]][\frac{1}{p}]$-algebra
$$K(k)\{\{t\}\}=\{\sum_ia_it^i|a_i\in K(k), |a_i|_pC^i\rightarrow0\,\forall C<1\}.$$
We need the following:

\begin{lem}
\label{ff}
The ring $K(k)\{\{t\}\}$ is faithfully flat over $W(k)[[t]][\frac{1}{p}]$.
\end{lem}
\begin{proof}
By Weierstrass preparation \cite[Chapter 5, Theorem 11.2]{lang}, every non-zero ideal $\ga$ of $W(k)[[t]][\frac{1}{p}]$ is generated by a polynomial of the 
form $t^d+\sum_{i=0}^{d-1}a_it^i$ with $a_0,\dots,a_{d-1}\in pW(k)$. Such polynomials are not units in $K(k)\{\{t\}\}$, because all of their zeros have $p$-adic 
absolute value strictly less than $1$. It follows that $\ga K(k)\{\{t\}\}\neq K(k)\{\{t\}\}$ for all $\ga\neq W(k)[[t]][\frac{1}{p}]$, hence the faithfulness follows.\\
Observe that $W(k)[[t]][\frac{1}{p}]$ is a principal ideal domain. So $K(k)\{\{t\}\}$ is flat if and only if it is torsion-free. However, this can be deduced from $K(k)[[t]]\supset K(k)\{\{t\}\}$.
\end{proof}

Now let $(\tilde M_\sigma,\tilde M_{\sigma,1},\tilde\phi)$ be a $\z/r\z$-graded $W(k)[[t]]$-window. Let us assume that the special 
fibre $(M_\sigma,M_{\sigma,1},\phi)$ is isoclinal with integral graded slope $a$. The $\sigma=0$ eigenspace can be written as
\begin{equation}
\label{newzwei}
\q\otimes M_0\cong I\otimes_{K(\f_{p^r})}K(k),
\end{equation}
where $I$ is the skeleton, as in subsection \ref{slop}. The lemma of Dwork \cite[Proposition 3.1]{katz2} reads: There is a unique isomorphism of $K(k)\{\{t\}\}$-modules:
$$\tilde M_0\otimes_{W(k)[[t]]}K(k)\{\{t\}\}\cong I\otimes_{K(\f_{p^r})}K(k)\{\{t\}\}$$
such that
\begin{itemize}
\item
the Dieudonn\'e connection $\nabla$ on the left matches the $\frac{\partial}{\partial t}$-operator on the right 
\item
reducing $\pmod t$ gives one back the isomorphism \eqref{newzwei}
\end{itemize}
We will call this isomorphism the Dwork trivialization. It is given by Taylor's formula
$$\tilde M_0\otimes_{W(k)[[t]]}K(k)\{\{t\}\}\ni x\mapsto\Theta(x)=\sum_{i=0}^\infty\nabla^i(x)|_{t=0}\frac{t^i}{i!},$$
where $|_{t=0}$ means the reduction $\pmod t$. This formula implies readily $\Theta(\tilde\phi^r(x))=\phi^r(\Theta(x))$. Consider the ring 
\begin{eqnarray*}
&&\Q(k)=K(k)\{\{t\}\}\otimes_{W(k)[[t]][\frac{1}{p}]}K(k)\{\{t\}\}
\end{eqnarray*}
and let $\theta\in\GL(I/K(\f_{p^r}))(\Q(k))$ be the isomorphism making the diagram:
$$\begin{CD}
\tilde M_0\otimes_{W(k)[[t]]}\Q(k)@>\Theta_1>>
I\otimes_{K(\f_{p^r})}\Q(k)\\
@A{=}AA@A{\theta}AA\\
\tilde M_0\otimes_{W(k)[[t]]}\Q(k)
@>\Theta_2>>I\otimes_{K(\f_{p^r})}\Q(k)
\end{CD}$$
commutative, where $\Theta_i=\Theta\otimes_{K(k)\{\{t\}\},q_i}1_{\Q(k)}$ are obtained from the Dwork trivialization via scalar extensions 
\begin{eqnarray}
\label{projone}
&&q_1:K(k)\{\{t\}\}\rightarrow\Q(k);\alpha\mapsto\alpha\otimes_{K(k)\{\{t\}\}}1\\
\label{projtwo}
&&q_2:K(k)\{\{t\}\}\rightarrow\Q(k);\alpha\mapsto1\otimes_{K(k)\{\{t\}\}}\alpha.
\end{eqnarray}
It is important to note that the trivial deformation satisfies $\theta=1$. Observe that we also have 
$\theta(\tau^r(x))=\tau^r(\theta(x))$. Now we are in a position to give a partial justification for the term ``$\sy$-structure'':

\begin{lem}
\label{velf}
Assume $p\neq2$ and let $(\tilde M_\sigma,\tilde N_\sigma,N_{i,\sigma},\zeta_\sigma)$ be as in part (ii) of lemma \ref{suff}. Assume that  
$\bigoplus_\sigma M_\sigma$, $\bigoplus_\sigma N_\sigma$, $\bigoplus_\sigma N_{i,\sigma}$ are isoclinal with integral graded slopes, 
and write $I$, $J$, $J_i$ for their skeletons. Let $\theta_{\tilde M}$ and $\theta_{\tilde N}$ be the Dwork descent data as above. Then:
\begin{itemize}
\item[(i)]
$\theta_{\tilde N}\in\SL(J/K(\f_{p^r}))(\Q(k))$ and no smaller $K(\f_{p^r})$-algebraic subgroup of $\SL(J/K(\f_{p^r}))$ contains $\theta_{\tilde N}$.
\item[(ii)]
Let $\pi$ denote the representation of $\SL(J/K(\f_{p^r}))$ on the $K(\f_{p^r})$-vector space $I$, that comes from the natural isomorphism
$$\zeta_0:I\stackrel{\cong}{\rightarrow}\bigoplus_{i=1}^cJ_i\otimes_{K(\f_{p^r})}\sy_{K(\f_{p^r})}^{b_i}J.$$
Then we have $\pi(\theta_{\tilde N})=\theta_{\tilde M}$.
\end{itemize}
\end{lem}
\begin{proof}
Recall that the Frobenius operator $\tilde\phi:\tilde N_{\sigma-1}\rightarrow\tilde N_\sigma$ is of the form $u_\sigma\circ\phi$. Here the automorphisms $u_\sigma$ 
arise from a choice of deformation sequence, and $(\tilde N_\sigma,\phi)$ is the trivial deformation of $(N_\sigma,\phi)$. As $\det(u_\sigma)=1$ it follows that the Dwork 
trivializations of the windows $(\det(\tilde N_\sigma),\det(\tilde\phi))$ and $(\det(\tilde N_\sigma),\det(\phi))$ agree, clearly leading to $\det(\theta_{\tilde N})=1$.\\

Consider a maximal proper $K(k)$-algebraic subgroup $H\subset\SL(J/K(\f_{p^r}))$ with $\theta_{\tilde N}\in H(\Q(k))$. In case $H$ is the Borel subgroup 
stabilizing a line $R$ in $J$ we get immediately an induced descent datum $\theta:R\otimes_{K(\f_{p^r})}\Q(k)\stackrel{\cong}{\rightarrow}R\otimes_{K(\f_{p^r})}\Q(k)$, 
of which the cocycle condition follows from the one of $\theta_{\tilde N}$. Hence we obtain a submodule
\begin{eqnarray*}
&&\tilde V=\{x\in R\otimes_{K(\f_{p^r})}K(k)\{\{t\}\}|\\
&&\theta(x\otimes_{K(k)\{\{t\}\},q_2}1_{\Q(k)})=x\otimes_{K(k)\{\{t\}\},q_1}1_{\Q(k)}\},
\end{eqnarray*}
of $\q\otimes\tilde N_0$ that is projective of rank $1$ over $W(k)[[t]][\frac{1}{p}]$. Analogously we have a factor module $\tilde V'=\q\otimes\tilde N_0/\tilde V$. 
The equation $\theta_{\tilde N}(\tau^r(x))=\tau^r(\theta_{\tilde N}(x))$ implies that $\tilde V$ and $\tilde V'$ are stable under the 
operator $\tilde\phi^r=p^{z/2}\tau^r$. This means that their graded slopes must be constant contradicting lemma \ref{suff}, (i).\\

If $H$ is not a Borel subgroup then $\theta_{\tilde N}$ is a similitude with respect to some non-degenerate 
$K(\f_{p^r})$-valued symmetric pairing $(.,..)$, i.e. $(x,y)=\theta^{-1}(\theta_{\tilde N}(x),\theta_{\tilde N}(y))$ for some 
$\theta\in\Q(k)^\times$. Again this implies a cocycle condition for $\theta$, giving rise to a $W(k)[[t]][\frac{1}{p}]$-module
\begin{eqnarray*}
&&\tilde V=\{x\in K(k)\{\{t\}\}|\\
&&\theta(x\otimes_{K(k)\{\{t\}\},q_2}1_{\Q(k)})=x\otimes_{K(k)\{\{t\}\},q_1}1_{\Q(k)}\},
\end{eqnarray*}
which inherits a $\tilde\phi^r$-operation from the map $p^z\tau^r$ on $K(k)\{\{t\}\}$. Moreover, the $K(k)\{\{t\}\}$-linear extension of $(.,..)$ defines a 
$\tilde\phi^r$-equivariant pairing $\q\otimes\tilde N_0\times\q\otimes\tilde N_0\rightarrow\tilde V$. By passing to the determinant of the pairing we find 
that $\q\otimes\det(\tilde N_0)^{\otimes2}$ regarded as a $W(k)[[t]][\frac{1}{p}]$-module with $\tilde\phi^r$-operation is isomorphic to $\tilde V^{\otimes2}$. 
If we fix a generator $w$ of $\tilde V$, and write $p^z\tau^r(w)=\alpha w$ for some $\alpha\in W(k)[[t]][\frac{1}{p}]$, this means that $\alpha^2$ is in fact equal 
to $p^{2z}\frac{\tau^r(x)}{x}$ where $x\in1+tW(k)[[t]]$. Hence $\alpha=\pm p^z\frac{\tau^r(\sqrt x)}{\sqrt x}$. After adjusting $w$ with a $p^r-1$th root of $-1$ 
we get an isomorphism $\q\otimes\det(\tilde N_0)\cong\tilde V$ of $W(k)[[t]][\frac{1}{p}]$-modules with $\tilde\phi^r$-operation, implying $\theta=1$. Hence 
$(\theta_{\tilde N}x,y)=(x,(\tr(\theta_{\tilde N})-\theta_{\tilde N})y)$ for all $x,y\in J\otimes_{K(\f_{p^r})}\Q(k)$, due to $\det(\theta_{\tilde N})=1$ and $\dim_{K(\f_{p^r})}J=2$. 
As the algebra $\{f\in\End_{K(\f_{p^r})}(J)|(fx,y)=(x,(\tr(f)-f)y)\}$ is a quadratic extension of $K(\f_{p^r})$, all of its elements commute with $\theta_{\tilde N}$ and thus 
give rise to endomorphisms of $\q\otimes\tilde N_0$. This is again a contradiction because due to slope considerations such endomorphisms can not exist.\\

Let $\Theta_{\tilde M}$ and $\Theta_{\tilde N}$ be the Dwork trivializations of $\tilde M_\sigma$ and $\tilde N_\sigma$. It will not cause confusion to write 
$\pi(\Theta_{\tilde N})$ for the isomorphism from $\tilde M_0\otimes_{W(k)[[t]]}K(k)\{\{t\}\}$ to $I\otimes_{K(\f_{p^r})}K(k)\{\{t\}\}$ that is induced by the isomorphism 
$$\Theta_{\tilde N}:\tilde N_0\otimes_{W(k)[[t]]}K(k)\{\{t\}\}\rightarrow J\otimes_{K(\f_{p^r})}K(k)\{\{t\}\}$$
Clearly we have to show that $\pi(\Theta_{\tilde N})=\Theta_{\tilde M}$, as $\theta_{\tilde M}$ is 
$$(\Theta_{\tilde M}\otimes_{K(k)\{\{t\}\},q_1}1_{\Q(k)})\circ(\Theta_{\tilde M}\otimes_{K(k)\{\{t\}\},q_2}1_{\Q(k)})^{-1},$$ 
and $\theta_{\tilde N}$ is
$$(\Theta_{\tilde N}\otimes_{K(k)\{\{t\}\},q_1}1_{\Q(k)})\circ(\Theta_{\tilde N}\otimes_{K(k)\{\{t\}\},q_2}1_{\Q(k)})^{-1}.$$ 
Let $\nabla$ be the Dieudonn\'e connection on the graded window $\tilde N_\sigma$. This induces immediately a connection on 
$\bigoplus_{i=1}^cN_{i,\sigma}\otimes_{W(k)}\sy_{W(k)[[t]]}^{b_i}\tilde N_\sigma$, by the formula
$$\nabla(x_i{x'}^b{x''}^{b_i-b})=x_i(b\nabla(x'){x'}^{b-1}{x''}^{b_i-b}+(b_i-b)\nabla(x''){x'}^b{x''}^{b_i-b-1}).$$
By transport of structure we get a quasi-connection on $\tilde M_\sigma$ whichis easily verified to satisfy \eqref{die}, so that it agrees with the Dieudonn\'e connection 
on $\tilde M_\sigma$, by lemma \ref{das}. It follows that $\pi(\Theta_{\tilde N})$ is horizontal and reduces to \eqref{newzwei} $\pmod t$, which is all we want.
\end{proof}

\section{$Y^{(k)}$ in Characteristic $0$} 
In this chapter we introduce the $k$th exterior power of an abelian variety of $\U(n-1,1)$ type. We start with an elementary 
discussion of the complex Hodge structure of $Y^{(k)}$. Then we sum up which properties of the $p$-adic Tate module 
can be deduced from it using methods of $p$-adic Hodge theory of Blasius, Faltings, Fontaine, Lafaille and Messing.

\subsection{Hodge Structure of $Y^{(k)}$ and Results of Blasius}
\label{exotic}
Let us recall objects and concepts from \cite[section 2]{ball} we fix:

\begin{itemize}
\item
torsion free, finitely generated $\O_L$-modules $V^{(k)}$ ($k=0,\dots,n$), each equipped with a $\q$-valued $(L,*)$-skew-Hermitian form $\psi^{(k)}$ which one can write as
\begin{equation}
\label{polzwei}
\psi^{(k)}(x,y)=\tr_{L/\q}\Psi^{(k)}(x,y)
\end{equation}
for a unique form $\Psi^{(k)}$ with $\Psi^{(k)}(\alpha x,y)=\Psi^{(k)}(x,\alpha^*y)=\alpha\Psi^{(k)}(x,y)$, for all $x,y\in V_\q^{(k)}$. 
We require that $(V^{(k)},\Psi^{(k)})$ is gotten from $V^{(0)}$, $V^{(1)}$, $\Psi^{(0)}$, and $\Psi^{(1)}$ according to the assignments 
\begin{eqnarray}
\label{vnull}
&&V^{(k)}={V^{(0)}}^{\otimes_{\O_L}1-k}\otimes_{\O_L}\bigwedge_{\O_L}^kV^{(1)},\\
\label{polka}
&&\Psi^{(k)}(x_0^{1-k}x_1\wedge\dots\wedge x_k,y_0^{1-k}y_1\wedge\dots\wedge y_k)=\\
\label{polnull}
&&(\Psi^{(0)}(x_0,y_0))^{1-k}
\det\left(\begin{matrix}\Psi^{(1)}(x_1,y_1)&\dots&\Psi^{(1)}(x_1,y_k)\\\vdots&\ddots&\vdots\\\Psi^{(1)}(x_k,y_1)&\dots&\Psi^{(1)}(x_k,y_k)\end{matrix}\right)
\end{eqnarray}
where $x_1,\dots,x_k,y_1,\dots,y_k\in V_\q^{(1)}$, and $x_0,y_0\in V_\q^{(0)}$,
\item 
We need certain group homomorphisms $g^{(k)}:G^{(0\times1)}\rightarrow G^{(k)}$, where the target reductive $\q$-groups $G^{(k)}$ represent the functors:
$$Q\mapsto\{(\gamma,\mu)\in\End_{L\otimes Q}^\times(V_Q^{(k)})\times Q^\times|\psi^{(k)}(\gamma x,\gamma y)=\mu\psi^{(k)}(x,y)\}$$
on the category of all $\q$-algebras $Q$, and where the source is the reductive $\q$-group $G^{(0\times1)}$ making
$$\begin{CD}
G^{(0\times 1)}@>g^{(1)}>>G^{(1)}\\
@Vg^{(0)}VV@V{c^{(1)}}VV\\
G^{(0)}@>{c^{(0)}}>>\g_m
\end{CD}$$
into a cartesian diagram. Here notice that the variable $\mu$ gives natural characters $c^\delta$ on each of the groups $G^\delta$, where $\delta$ stands for 
a symbol in $\{(0\times1),(0),\dots,(n)\}$. We define $g^{(k)}$ on $G^{(0\times1)}$ by sending say $(\gamma^{(0)},\gamma^{(1)},\mu)$ in $G^{(0\times1)}(Q)$ 
to $(\gamma^{(k)},\mu)\in G^{(k)}(Q)$, where the action of $\gamma^{(k)}$ on $V_Q^{(k)}$ is defined by
\begin{equation*}
\gamma^{(k)}(x_0^{1-k}x_1\wedge\dots\wedge x_k)=\gamma^{(0)}(x_0)^{1-k}\gamma^{(1)}(x_1)\wedge\dots\wedge\gamma^{(1)}(x_k).
\end{equation*}
Notice that $c^{(k)}\circ g^{(k)}=c^{(0\times1)}$. Throughout the whole article we prefer to work with specific connected, smooth $\z$-models $G_\z^\delta$ of the groups 
$G^\delta$: They are obtained by taking the schematic closure in the group $\z$-schemes $\GL(V^\delta/\z)$, where we write $V^{(0\times1)}$ for $V^{(0)}\oplus V^{(1)}$. 
It is easy to see that our homomorphisms $g^{(k)}$ give rise to integral maps $G_\z^{(0\times1)}\rightarrow G_\z^{(k)}$ and the same is true for $c^\delta$.
\item
We need a specific conjugacy class $X^{(0\times1)}$ of homomorphisms $\c^\times\rightarrow G^{(0\times1)}(\r)$ such that $(G^{(0\times1)},X^{(0\times1)})$ and 
$(G^{(k)},X^{(k)})$ are both Shimura data, in fact of PEL type, where $X^{(k)}$ denotes the $G^{(k)}(\r)$-conjugacy class containing $g^{(k)}(X^{(0\times1)})$, for every 
$k\in\{0,\dots,n\}$. To this end we require that $X^{(0\times1)}$ is chosen such that some $h^{(0)}$ in $X^{(0)}$ and $h^{(1)}$ in $X^{(1)}$ satisfy the following properties:\\
Firstly the forms $\psi^{(1)}(x,h^{(1)}(i)y)$ and $\psi^{(0)}(x,h^{(0)}(i)y)$ are positive definite, and 
secondly the Hodge structures $V^{(0)}$ and $V^{(1)}$ are of type $\{(-1,0),(0,-1)\}$ and satisfy:
\begin{eqnarray}
\label{kottnull}
&&\tr(x|_{{V^{(0)}}^{-1,0}})=\Phi^{(0)}(x)\\
\label{kotteins}
&&\tr(x|_{{V^{(1)}}^{-1,0}})=(n-1)\Phi^{(0)}(x)+\Phi^{(n)}(x),
\end{eqnarray}
for all $x\in L$, where $\Phi^{(0)}$ and $\Phi^{(n)}$ are CM traces arising from CM-types $|\Phi^{(0)}|,|\Phi^{(n)}|\subset\Hom_\q(L,\c)$. 
Under these assumptions it turns out that firstly $\psi^{(k)}(x,h^{(k)}(i)y)$ is positive definite for every $k$, and that secondly the 
Hodge structure $V^{(k)}$ is of type $\{(-1,0),(0,-1)\}$ as well. Moreover, the action of $L$ on its tangent space is given by
\begin{equation*}
\tr(x|_{{V^{(k)}}^{-1,0}})=\Phi^{(k)}(x)=\binom{n-1}{k}\Phi^{(0)}(x)+\binom{n-1}{k-1}\Phi^{(n)}(x).
\end{equation*}
Finally let us denote by $E$ the subfield of $\c$ that is generated by all embeddings $L\hookrightarrow\c$. 
It contains the reflex fields of all of the $(G^\delta,X^\delta)$ for $\delta\in\{(0\times1),(0),\dots,(n)\}$.
\item
A level $l\geq3$, gives rise to neat compact open groups 
\begin{equation*}
K^\delta=\{\gamma\in G_\z^\delta(\zd)|\gamma\equiv1\pmod l\}
\end{equation*}
for every $\delta$. Having fixed it we write $M^\delta$ for the weakly canonical model of 
$G^\delta(\q)\backslash(X^\delta\times G^\delta(\a^\infty))/K^\delta$ over $E$. By abuse of notation we also write 
\begin{equation*}
g^{(k)}:M^{(0\times1)}\rightarrow M^{(k)},
\end{equation*} 
for the induced map, cf. \cite[Corollaire 5.4]{deligne1}. 
\end{itemize}

\subsubsection{Moduli interpretations}
We remark that $(G^{(0\times1)},X^{(0\times1)})$ can very nicely be written as a cartesian product of the toric Shimura datum $(G^{(0)},X^{(0)})$ with a certain 
preabelian Shimura datum $(G,X)$ belonging to the unitary group of $V={V^{(0)}}^{\otimes_{\O_L}-1}\otimes_{\O_L}V^{(1)}$, cf. \cite[Remark 2.4]{ball}. 
For sake of completeness we give the Hodge weights of $\bigwedge_\c^kV_\iota$, they are:
$$(p,q)=\begin{cases}{\binom{n-1}{k}\times(0,0), \binom{n-1}{k-1}\times(1,-1)}&\text{if $\iota\in|\Phi^{(0)}|-|\Phi^{(n)}|$}\\
{\binom{n-1}{k-1}\times(-1,1), \binom{n-1}{k}\times(0,0)}&\text{if $\iota\in|\Phi^{(n)}|-|\Phi^{(0)}|$}\\
{\binom{n}{k}\times(0,0)}&\text{otherwise}\end{cases}$$
observe that, unlike $V^{(k)}$, none of the Hodge structures $\bigwedge_{\O_L}^kV$ are of type $\{(-1,0),(0,-1)\}$, hence are not 
period lattices of abelian varieties. Also, the Shimura variety to $(G,X)$, which in fact looks more natural to work with, has no moduli 
interpretation. However $M^{(0\times1)}$ does have one, let us sketch it briefly. The morphisms $S\rightarrow M^{(0\times1)}$ 
parameterize tuples $(Y^{(1)},Y^{(0)},\q_{>0}\lambda^{(0)}\times\lambda^{(1)},\iota^{(1)},\iota^{(0)},\ebar^{(0\times1)})$ where:
\begin{itemize}
\item[(a)]
$Y^{(1)}$ and $Y^{(0)}$, abelian schemes over $S/E$, up to isogeny,
\item[(b)]
$\q_{>0}\lambda^{(0)}\times\lambda^{(1)}$, a homogeneous class of polarizations on the product $Y^{(0\times1)}=Y^{(0)}\times_SY^{(1)}$
\item[(c)]
Rosati invariant operations $\iota^{(1)}:L\rightarrow\End^0(Y^{(1)})$ and $\iota^{(0)}:L\rightarrow\End^0(Y^{(0)})$, 
such that the formula $\tr_{\Lie(Y^{(k)})}(\iota^{(k)}(x))=\Phi^{(k)}(x)$ holds for $k\in\{0,1\}$
\item[(d)]
level-$K^{(0\times1)}$-structure $\ebar^{(0\times1)}$, i.e. for some choice of geometric point $\xi$ of $S$ one has a 
$\pi_1(S,\xi)$-invariant  $K^{(1)}$- (resp. $K^{(0)}$-) class of $\O_L\otimes\a^\infty$-linear symplectic similitudes:
$$\eta^{(1)}:V^{(1)}\otimes\a^\infty\rightarrow H_1(Y_\xi^{(1)},\a^\infty)$$
and $$\eta^{(0)}:V^{(0)}\otimes\a^\infty\rightarrow H_1(Y_\xi^{(0)},\a^\infty)$$ with the same multipliers. 
\end{itemize}
The Shimura variety $M^{(k)}$ has a similar moduli interpretation, moreover the map $g^{(k)}$ gives rise to 
a quadruple $(Y^{(k)},\iota^{(k)},\q_{>0}\lambda^{(k)},\ebar^{(k)})$ over the scheme $M^{(0\times1)}$. 

\subsubsection{Computation of $H_{dR}^1(Y_\xi^{(k)})$}
Unless otherwise said we will always work with the representative of the isogeny class of abelian schemes $Y^{(k)}$ which satisfies 
$\eta^{(k)}(V_{\zd}^{(k)})=H_1(Y_\xi^{(k)},\zd)$, where $Y_\xi^{(k)}$ denotes the fibre of $Y^{(k)}$ over a $\c$-valued point $\xi$ of $M^{(0\times1)}$. If 
$\lambda^{(0)}\times\lambda^{(1)}$ denotes a quasi-polarization of $Y^{(0\times1)}$ that represents the homogeneous class $\q_{>0}\lambda^{(0)}\times\lambda^{(1)}$, 
then the formulae \eqref{polnull}, \eqref{polka} determine a corresponding representative $\lambda^{(k)}$ in the homogeneous class $\q_{>0}\lambda^{(k)}$. 
Moreover, a suitable multiple turns all of them into effective ones. Let us therefore fix such polarizations
\begin{equation}
\label{pol}
\lambda^\delta:Y^\delta\rightarrow{Y^\delta}^t
\end{equation}
once and for all. By definition the $\z$-Hodge structure
$$H_1(Y_\xi^{(0)}(\c),\z)^{\otimes_{\O_L}1-k}\otimes_{\O_L}\bigwedge_{\O_L}^kH_1(Y_\xi^{(1)}(\c),\z)$$
is furnished with a specific $\O_L$-linear isomorphism to:
$$H_1(Y_\xi^{(k)}(\c),\z)$$
which we denote by $m_{\xi,B}^{(k)}$. Dualizing the above yields a cohomological pendant $t_{\xi,B}^{(k)}$, and indeed it will prove useful to 
occasionally use $H^1(Y_\xi^{(k)}(\c),\z)$ instead of $H_1(Y_\xi^{(k)}(\c),\z)$. Observe also that the polarization $\lambda^{(k)}$ induces a map: 
\begin{equation}
\label{polvier}
H_1(Y_\xi^{(k)}(\c),\z)\rightarrow H_1({Y_\xi^{(k)}}^t(\c),\z)\cong H^1(Y_\xi^{(k)}(\c),\z)(1)
\end{equation}
carrying $m_{\xi,B}^{(k)}$ to $t_{\xi,B}^{(k)}$. By using \cite[Corollaire 5.4]{deligne1} and \cite{deligne2} one obtains easily the 
following key property of $t_{\xi,B}^{(k)}$, see \cite[section 2]{ball} for some more details:

\begin{lem}
\label{vvier}
Let $V^\delta,X^\delta,\psi^\delta,K^\delta$, be as above. Let $F\subset\c$ be a field containing $E$. Let
$$\xi:\Spec F\rightarrow M^{(0\times1)} $$
be a $F$-valued point, then:
\begin{itemize}
\item[(i)]
the $\Od=\O_L\otimes\zd$-linear isomorphism $t_{\xi,B}^{(k)}\otimes1_{\zd}$ between
$$H^1(Y_\xi^{(0)}\times_FF^{ac},\zd)^{\otimes_\Od1-k}\otimes_\Od\bigwedge_\Od^kH^1(Y_\xi^{(1)}\times_FF^{ac},\zd)$$
and
$$H^1(Y_\xi^{(k)}\times_FF^{ac},\zd)$$
is $\Gal(F^{ac}/F)$-equivariant.
\item[(ii)]
the $\c\otimes L$-linear isomorphism between
$$(H_{dR}^1(Y_\xi^{(0)})^{\otimes_{F\otimes L}1-k}\otimes_{F\otimes L}\bigwedge_{F\otimes L}^kH_{dR}^1(Y_\xi^{(1)}))\otimes_F\c$$
and
$$H_{dR}^1(Y_\xi^{(k)})\otimes_F\c$$
obtained from $t_{\xi,B}^{(k)}$ is defined over $F\otimes L$
\end{itemize}
\end{lem}

Let us thus write $t_{\xi,dR}^{(k)}$ for the isomorphism of the filtered $F\otimes L$-modules. Finally, consider a $K(k)$-valued point $\xi$, where $k$
is a perfect field of characteristic $p$. Notice the functorial comparison isomorphism of Faltings and Fontaine, \cite{fontaine2}, \cite{faltings}:
$$\begin{CD}
H^1(Y_\xi^{(k)}\times_{K(k)}K(k)^{ac},\q_p)\otimes_{\q_p}B_{dR}(K(k)^{ac})\\
@V{I_{dR}}VV\\
H_{dR}^1(Y_\xi^{(k)})\otimes_{K(k)}B_{dR}(K(k)^{ac})
\end{CD}$$
Let us now explain a method of Blasius to prove certain compatibilities. We prefer to reproduce his beautiful argument for two reasons: In the situation 
we have his proof simplifies a little, yet we need a slight generalization of the result. The reader who is familiar with the details of \cite{blasius} and
\cite{deligne2} should certainly skip the rest of this section, see also \cite[Chapter 4]{ogus}.

\begin{prop}
\label{dR}
Let $\xi$ be a $K(k)$-valued point of $M^{(0\times1)}$, and let $K(k)\hookrightarrow\c$ be an imbedding ($\Card(k)\leq2^{\aleph_0}$). Then the map from
$$(H_{dR}^1(Y_\xi^{(0)})^{\otimes_{K(k)\otimes L}1-k}\otimes_{K(k)\otimes L}\bigwedge_{K(k)\otimes L}^kH_{dR}^1(Y_\xi^{(1)}))\otimes_{K(k)}B_{dR}(K(k)^{ac})$$
to
$$H_{dR}^1(Y_\xi^{(k)})\otimes_{K(k)}B_{dR}(K(k)^{ac})$$
induced by $t_{\xi,dR}^{(k)}\otimes_{K(k)}1_{B_{dR}(K(k)^{ac})}$ agrees with $I_{dR}\circ(t_{\xi,B}^{(k)}\otimes1_{B_{dR}(K(k)^{ac})})\circ I_{dR}^{-1}$.
\end{prop}
\begin{proof}
Recall once more that $t_{\xi,dR}^{(k)}$ was defined from $t_{\xi,B}^{(k)}$ in the first place: Starting with the archimedian comparison isomorphisms
$$H_{dR}^1(Y_\xi^{(k)})\otimes_{K(k)}\c\cong H^1(Y_\xi^{(k)}(\c),\q)\otimes\c,$$
the map $t_{\xi,B}^{(k)}\otimes1_\c$ gives rise to $t_{\xi,dR}^{(k)}\otimes_{K(k)}1_\c$ and the $K(k)$-rationality of 
that map was checked in the second part of lemma \ref{vvier}. Write $Y/M^{(0\times1)}$ for the abelian scheme
$${Y^{(0)}}^{\times_{M^{(0\times1)}}k-1}\times_{M^{(0\times1)}}Y^{(1)}\times_{M^{(0\times1)}}Y^{(k)},$$
and let $s:Y\rightarrow M^{(0\times1)}$ be the structural morphism. By some custommary 
manipulation of tensor products and duals one can regard $t_{\xi,B}^{(k)}$ as an element in
$$H^{2k}(Y_\xi(\c),\q)(k).$$
As $\xi$ varies these form a global section which we denote by
$$t_B^{(k)}\in H^0(M^{(0\times1)}(\c),R^{2k}s_*\q)(k).$$
We want to choose (over $K(k)$) a smooth compactification $j:Y\times_EK(k)\hookrightarrow\Y/K(k)$, with $\Y-Y\times_EK(k)=\Ybar/K(k)$ a normal 
crossings divisor. Due to Deligne's th\'eor\`eme de la partie fixe \cite[Th\'eor\`eme 4.1.1]{deligne3}, we deduce the existence of 
a cohomology class $T_B^{(k)}\in H^{2k}(\Y(\c),\q)(k)$ such that $j^*(T_B^{(k)})$ maps to $t_B^{(k)}$, in the Leray spectral sequence 
to $s:Y\rightarrow M^{(0\times1)}$. The restriction of $T_B^{(k)}$ to each fibre $Y_\xi\subset\Y$ is $t_{\xi,B}^{(k)}$. Analogously one 
can find a $T_{dR}^{(k)}\in H_{dR}^{2k}(\Y)(k)$ that restricts to $t_{\xi,dR}^{(k)}\in H_{dR}^{2k}(Y_\xi)(k)$ for all $\xi$ (N.B.: The 
images of $T_B^{(k)}$ and $T_{dR}^{(k)}$ in the group $H_{dR}^{2k}(\Y\times_{K(k)}\c)(k)$ may well be different, 
but this is irrelevant). Now let $t_{\xi,B}^{(k)}\otimes1_{\q_p}$ and $t_B^{(k)}\otimes1_{\q_p}$ be the images of $t_{\xi,B}^{(k)}$ 
and $t_B^{(k)}$ in $p$-adic \'etale cohomology. The assertion in question is equivalent to the statement 
$I_{dR}(t_{\xi,B}^{(k)}\otimes1_{B_{dR}(K(k)^{ac})})=t_{\xi,dR}^{(k)}\otimes_{K(k)}1_{B_{dR}(K(k)^{ac})}$.\\

The key idea of Blasius is that an element in $H_{dR}^{2k}(\Y)(k)$ restricts to $0$ in $H_{dR}^{2k}(Y_\xi)(k)$ if and only if it restricts to $0$ in 
$H_{dR}^{2k}(Y_{\xi'})(k)$, where $\xi'$ is any other $K(k)$-valued point lying in the same geometric connected component of $M^{(0\times1)}$. When 
applying this to the element $I_{dR}(T_B^{(k)}\otimes1_{B_{dR}(K(k)^{ac})})-T_{dR}^{(k)}\otimes_{K(k)}1_{B_{dR}(K(k)^{ac})}$ we find that our assertion is 
true for $\xi$ if and only if it is true for $\xi'$, here again, we denote the image of $T_B^{(k)}$ in $H^{2k}(\Y\times_{K(k)}K(k)^{ac},\q_p)(k)
\otimes_{\q_p}B_{dR}(K(k)^{ac})$ by $T_B^{(k)}\otimes1_{B_{dR}(K(k)^{ac})}$. There exist such points $\xi'$ with $Y_{\xi'}^{(1)}$ isogenous to
$${Y_\xi^{(0)}}^{\times_{K(k)}n-1}\times_{K(k)}Y_\xi^{(n)}$$ 
as a polarized abelian variety with $L$-operation. For those $\xi'$ there is an isogeny between 
${Y_\xi^{(0)}}^{\times_{K(k)}\binom{n-1}{k}}\times_{K(k)}{Y_\xi^{(n)}}^{\times_{K(k)}\binom{n-1}{k-1}}$ and $Y_{\xi'}^{(k)}$ as well. 
Now observe that the isomorphism $I_{dR}$ for $Y_{\xi'}^{(k)}$ is just the direct sum of $\binom{n-1}{k}$ (resp. $\binom{n-1}{k-1}$ 
many) copies of the corresponding isomorphism for $Y_\xi^{(0)}$ (resp. $Y_\xi^{(n)}$). This proves it for $\xi'$ and we are done.
\end{proof}

\subsection{The Categories of Fontaine}
\label{crys}
Our reference for good reduction of Shimura varieties of $PEL$-type is the work \cite{kottwitz1} of Kottwitz. In our case 
the unramifiedness hypotheses of loc.cit. boil down to the requirement that $L$, $V^\delta$, $\psi^\delta$, $l$ satisfy:

\begin{itemize}
\item
$\psi^\delta$ induces a $\z_{(p)}$-valued perfect pairing of $V_{\z_{(p)}}^\delta$.
\item
$l$ is coprime to $p$.
\item
$L$ is unramified at all primes over $p$.
\end{itemize}

Here $p$ is some fixed prime. If this is fulfilled then Kottwitz obtains an extension of $M^\delta/E$ to a scheme $\M^\delta$ smooth over $\O_E\otimes\z_{(p)}$, 
moreover $E$ is unramified over $p$ too. One constructs such models $\M^\delta$ via a moduli interpretation, which is very similar to the one given in (a), (b), 
(c), and (d) above. However, one has to replace the constraint (c) by the so-called determinant condition, see \cite[chapter 5]{kottwitz1} for details. In fact the above 
unramifiedness hypothesis is fulfilled for all $\delta$ if and only if it is fulfilled for $\delta=(0\times1)$, due to \cite[Lemma 5.1]{ball}. Moreover loc.cit. shows that the 
polarizations in \eqref{pol} can be chosen to be $p$-principal ones, which we assume from now on. It will be useful to have a criterion for the compactness of $\M^\delta$:

\begin{lem}
\label{compact}
Assume that $L$, $V^{(k)}$, $\psi^{(k)}$, $l\geq3$, are unramified at $p$, and assume 
$\Phi^{(n)}\neq\Phi^{(0)}\circ*$. Then the scheme $\M^{(k)}$ is proper over $\O_E\otimes\z_{(p)}$.
\end{lem}
\begin{proof}
As $\M^{(k)}$ is smooth it is flat and the generic points of all irreducible components are mapped to the generic point of $\Spec\O_E\otimes\z_{(p)}$. So it 
suffices to verify the valuative criterion of properness \cite[Corollaire 7.3.10(ii)]{egaii} only for discrete valuation rings $V$ over $\O_E\otimes\z_{(p)}$ whose 
field $F$ of fractions has characteristic $0$. Consider the quadruple $(Y,\q_{>0}\lambda,\dots)$ corresponding to some $\eta:\Spec F\rightarrow M^{(k)}$. 
Observe that the $L$-operation that is induced on the tangent space $\Lie Y$ of the abelian variety has trace $\Phi^{(k)}$. Consequently we can apply 
the theory in \cite[Theorem 3]{morita}: As $\Phi^{(n)}\neq\Phi^{(0)}\circ*$ we can find a $\iota_0\notin|\Phi^{(0)}|\cup|\Phi^{(n)}|$ so that $\Phi^{(k)}$ does 
never involve the embedding $\iota_0$. By loc.cit. this forces $Y$ to have potentially good reduction. However, as the abelian variety in question also 
has a $F$-rational level $l$-structure, we may use the Crit\`ere de Raynaud \cite[Proposition 4.7]{grothendieck2}, to deduce it has in fact good reduction.
\end{proof}

We usually want to consider the scheme $\M_\gp^\delta:=\M^\delta\times_{\O_E\otimes\z_{(p)}}\O_{E_\gp}$, where $\gp$ is a 
fixed prime of $\O_E$. The study of the local properties of $\M_\gp^\delta$ is significantly eased if a further assumption is made:
\begin{itemize}
\item[(T.1)]
$\Phi^{(0)}\circ*=\tr_{L_\gq/\q_p}|_L\neq\Phi^{(n)}$
\end{itemize}
for some prime $\gq$ of $\O_L$. If we fix an embedding
\begin{equation}
\label{nzwei}
\iota:\O_{L_\gq}\stackrel{\cong}{\rightarrow}W(\f_{p^r})=\O_{E_\gp},
\end{equation}
then (T.1) means that 
\begin{equation}
\label{sechsz}
\{\tau^\sigma\circ\iota\circ*|\sigma=0,\dots,r-1\}=|\Phi^{(0)}|,
\end{equation}
while there exists a proper subset $\Omega\subset\{0,\dots,r-1\}$ with 
\begin{equation}
\label{vierz}
\{\tau^\sigma\circ\iota|\sigma\in\Omega\}\cup\{\tau^\sigma\circ\iota\circ*|\sigma\notin\Omega\}=|\Phi^{(n)}|.
\end{equation}
Under the validity of this assumption $\M_\gp^{(0\times1)}$ is proper, all fibres of the $p$-divisible group $Y^{(1)}[\gq^\infty]$ over $\M_\gp^{(0\times1)}$ 
are unipotent and $Y^{(0)}[\gq^\infty]$ is \'etale. Moreover their duals are identified with $Y^{(1)}[\gq^{*\infty}]$ and $Y^{(0)}[\gq^{*\infty}]$.

\subsubsection{Dieudonn\'e theory}
\label{contrava} 
Let us sketch one of the numerous ways to define the contravariant Dieudonn\'e module of a $p$-divisible group $\G/W(k)$: According to 
\cite{mazur} one attaches a crystal to $\G$, for our purposes it is enough to only consider the value of this crystal on the object $\Spec W(k)$ being:
$$\d^*(\G)=\Ext^{crys/\z_p}(\G,\g_a)$$
which is defined to be the abelian group of isomorphism classes of the (by \cite[Corollary 7.8]{mazur} rigid) category
$$\EXT^{crys/\z_p}(\G,\g_a)=\lim\limits_{\leftarrow N}\EXT^{crys/\z_p}(\G[p^N],\g_a).$$
Moreover, the category $\EXT^{crys/\z_p}(\G[p^N],\g_a)$ consists of objects in $\TORS^{crys/\z_p}(\G[p^N],\g_a)$ -the 
$\g_a$-torseurs $P$ on the small crystalline site $Crys(\G[p^N]/\z_p)$- together with isomorphisms:
$$(p_1+p_2)^*(P)\cong p_1^*(P)\wedge p_2^*(P)$$
of objects in $\TORS^{crys/\z_p}(\G[p^N]\times_{W(k)}\G[p^N],\g_a)$, that satisfy commutativity 
and associativity constraints. Here $\wedge$ means contracted sum of $\g_a$-torseurs and 
$$p_1,p_2:\G[p^N]\times_{W(k)}\G[p^N]\rightarrow\G[p^N]$$
are the coordinate maps. The resulting group $\d^*(\G)$ is a $W(k)$-module because $W(k)$ acts on $\g_a$, regarded 
as a sheaf of abelian groups on $Crys(\G[p^N]/\z_p)$. From this definition one can read off immediately the crystalline 
functoriality, in particular we have a $\tau$-linear map $\phi$ on it that is induced from the relative Frobenius:
$$Frob:\Gbar\rightarrow\Gbar\times_{k,\tau}k;x\mapsto x^p$$
where $\Gbar$ is $\G\times_{W(k)}k$. Finally, according to 
\cite[corollary 7.13]{mazur} there is an exact sequence of finitely generated torsion free $W(k)$-modules:
$$0\rightarrow\omega(\G)\rightarrow\Ext^{crys/\z_p}(\G,\g_a)\rightarrow\Ext(\G,\g_a)\rightarrow0$$
where $\omega(\G)=\lim\limits_{\leftarrow j}\omega(\G\times_{W(k)}W_j(k))$ is the cotangent space of $\G$, giving $\d^*(\G)$ a filtration, by 
\begin{equation*}
Fil^i\d^*(\G)=\begin{cases}\d^*(\G)&i<1\\\omega(\G)&i=1\\0&i>1\end{cases}.
\end{equation*}
Let us describe Fontaine's $W(k)$-linear $\otimes$-category $MF^{fd}$ as well: A object in $MF^{fd}$ is a tuple $(M,Fil^iM,\phi)$ with:
\begin{itemize}
\item
$M$ is a finitely generated, torsion free $W(k)$-module, $\dots\subset Fil^{i+1}M\subset Fil^iM\subset\dots$ 
is a separating and exhausting descending filtration, such that $M/Fil^iM$ is torsion free.
\item
$\phi:\q\otimes M\rightarrow\q\otimes M$ is a $\tau$-linear bijection
\item
$\sum_{i\in\z}p^{-i}\phi(Fil^iM)=M$
\end{itemize}
and a morphism in $MF^{fd}$ is a $W(k)$-linear map which preserves the filtration and the $\phi$-operation. For every integer $w\geq0$ 
there is an important full subcategory $MF^{fd,w}$ which consists of objects with $M=Fil^0M$ and $Fil^wM=0$. It is a fact that 
$(\d^*(\G),Fil^i\d^*(\G),\phi)$ obtained from a $p$-divisible group $\G/W(k)$ constitutes an object in $MF^{fd,2}$, moreover $\d^*$ is an 
anti-equivalence between the category of $p$-divisible groups over $W(k)$ and $MF^{fd,2}$.\\ 
Finally observe that a $p$-divisible group $\Gbar$ over $k$ gives rise to a Dieudonn\'e module $\d^*(\Gbar)$ over $W(k)$ too, it lacks a filtration, but 
does however still satisfy the crucial property $p\d^*(\Gbar)\subset\phi\d^*(\Gbar)\subset\d^*(\Gbar)$, in this setting $\d^*$ is again an anti-equivalence of categories.\\
In \cite[7.14]{fontaine} a fully faithful contravariant exact functor $U$ from a certain full subcategory $MF^{fd,p'}$ of $MF^{fd,p}$ to the category of 
torsion free finitely generated $\z_p$-modules with continuous $\Gal(K(k)^{ac}/K(k))$-operation is constructed. An object of $MF^{fd,p}$ is 
in $MF^{fd,p'}$ if and only if it satisfies the following subtle extra condition: Whenever a subobject $M'$ of $M$ and a maximal subobject $M''$ of
$M'$ is given we have that $Fil^{p-1}M'/M''\neq M'/M''$. This is implied by demanding that $p^{p-1}\phi^{-1}$ acts topologically nilpotent on $M$. 
According to the discussion in \cite[3.8, Remarque]{fontaine3} the functor $U$ can be described as follows: Consider $W_j(\O_{K(k)^{ac}}/p\O_{K(k)^{ac}})$, 
the Witt vectors of length $j$ with coefficients in $\O_{K(k)^{ac}}/p\O_{K(k)^{ac}}$. Let $W_j^{DP}(\O_{K(k)^{ac}}/p\O_{K(k)^{ac}})$ be the divided power envelope of 
the ideal consisting of all $(a_0,\dots,a_{j-1})\in W_j(\O_{K(k)^{ac}}/p\O_{K(k)^{ac}})$ with $a_0^{p^j}=0$, relative to the standard divided powers on the subideal formed 
by all Witt vectors $(a_0,\dots,a_{j-1})$ with $a_0=0$. Finally set 
$$\Sd'=\lim\limits_{\leftarrow j}W_j^{DP}(\O_{K(k)^{ac}}/p\O_{K(k)^{ac}}),$$
this is a $W(k)$-algebra which has a natural $\Gal(K(k)^{ac}/K(k))$-operation a natural filtration, and a Frobenius $\tau$, so one puts:
\begin{equation}
\label{fact2}
U(M)=Fil^0\Hom_{W(k)}(M,\Sd')^{\phi=1},
\end{equation}
for every $M$ in $MF^{fd,p'}$. Its significance to this paper stems from its relation to $p$-divisible groups over 
$W(k)$: If $\G/W(k)$ is such that $\d^*(\G)$ lies in $MF^{fd,p'}$ (which means $p$ odd or $\G$ unipotent), then
\begin{equation}
\label{fact1}
U(\d^*(\G))\cong T_p\G
\end{equation}
holds, where $T_p\G=\lim\limits_{\leftarrow N}\G[p^N](K(k)^{ac})$ is the Tate module, see \cite[Proposition 9.12]{fontaine} for this.\\

The following consequence of Fontaine and Lafaille's theory is basic to this work. Let $k$ be a perfect field extension of $\f_{p^r}$. Notice that a $\O_{L_\gq}$-operation on 
any object $M$ of $MF^{fd}$ is equivalent to giving a grading $M=\bigoplus_{\sigma\in\z/r\z}M_\sigma$ with $\q\otimes\phi(M_\sigma)=\q\otimes M_{\sigma+1}$. This is done 
by making $\alpha\in\O_{L_\gq}$ act on $M_\sigma$ through $x\mapsto\tau^\sigma(\iota(\alpha))x$.

\begin{lem}
\label{BARS}
Assume that $p\geq k+1$. Let $\G_\gq^{(1)}$ be a $p$-divisible group over $W(k)$ with $\O_{L_\gq}$-operation and of $\O_{L_\gq}$-height $n$. Assume that 
$$\rk_{W(k)}\omega_\sigma(\G_\gq^{(1)})=\begin{cases}1&\sigma\in\Omega\\0&\sigma\notin\Omega\end{cases}$$
Fix some \'etale $p$-divisible group $\G_\gq^{(0)}/W(k)$ with $\O_{L_\gq}$-operation, and with $\O_{L_\gq}$-height 
equal to $1$. Then there exists a $p$-divisible group $\G_\gq^{(k)}/W(k)$ with $\O_{L_\gq}$-operation such that one has 
$$T_p\G_\gq^{(k)}\cong(T_p\G_\gq^{(0)})^{\otimes_{\O_{L_\gq}}1-k}\otimes_{\O_{L_\gq}}\bigwedge_{\O_{L_\gq}}^k(T_p\G_\gq^{(1)})$$
as $\O_{L_\gq}[\Gal(K(k)^{ac}/K(k))]$-modules. Moreover, the graded Dieudonn\'e module $M_\sigma^{(k)}$ of $\G_\gq^{(k)}$ is canonically isomorphic to
\begin{equation}
\label{vfuenfz}
{M_\sigma^{(0)}}^{\otimes_{W(k)}1-k}\otimes_{W(k)}\bigwedge_{W(k)}^kM_\sigma^{(1)},
\end{equation} 
where $M_\sigma^{(1)}$ and $M_\sigma^{(0)}$ are the ones of $\G_\gq^{(1)}$ and $\G_\gq^{(0)}$. In particular we have
$$\rk_{W(k)}\omega_\sigma(\G_\gq^{(k)})=\begin{cases}\binom{n-1}{k-1}&\sigma\in\Omega\\0&\sigma\notin\Omega\end{cases}$$
and the $\O_{L_\gq}$-height is $\binom{n}{k}$. 
\end{lem}
\begin{proof}
Write $M_\sigma^{(k)}$ for the $\z/r\z$-graded object of $MF^{fd}$, that is obtained from the filtered Dieudonn\'e modules of $\G_\gq^{(1)}$, and 
$\G_\gq^{(0)}$ according to equation \eqref{vfuenfz}. By analyzing its Hodge weights one finds analogously to \cite[proof of lemma 2.1]{ball}:
\begin{itemize}
\item
$\rk_{W(k)}Fil^0M_\sigma^{(k)}=\rk_{W(k)}M_\sigma^{(k)}=\binom{n}{k}$
\item
$\rk_{W(k)}Fil^1M_\sigma^{(k)}=\begin{cases}\binom{n-1}{k-1}&\sigma\in\Omega\\0&\sigma\notin\Omega\end{cases}$
\item
$\rk_{W(k)}Fil^2M_\sigma^{(k)}=0$.
\end{itemize}
Using the fully faithfulness of $U$ we are left with proving that
\begin{eqnarray}
\label{foxy}
&&\bigwedge_{\O_{L_\gq}}^kU(M^{(1)})\cong U(\bigoplus_\sigma\bigwedge_{W(k)}^kM_\sigma^{(1)})\\
\label{fix}
&&\cong U(M^{(0)})^{\otimes_{\O_{L_\gq}}k-1}\otimes_{\O_{L_\gq}}U(M^{(k)})
\end{eqnarray}
The isomorphism \eqref{foxy} follows in two steps: Firstly, we take the appropriate $\O_{L_\gq}^{\otimes_{\z_p}k}$-eigenspace of the isomorphism:
$$U(M^{(1)})^{\otimes_{\z_p}k}\cong U({M^{(1)}}^{\otimes_{W(k)}k}),$$
which in turn can be justified by \cite[Remarque 6.13(b)]{fontaine}, if we only know that ${M^{(1)}}^{\otimes_{W(k)}k}$ is in $MF^{fd,p'}$. However, this 
object is certainly in $MF^{fd,p}$, because of $p\geq k+1$ and $M^{(1)}\in MF^{fd,2}$. Moreover, $p\phi^{-1}$ acts topologically nilpotently on $M^{(1)}$ 
by (T.1). So $p^k\phi^{-1}$, hence $p^{p-1}\phi^{-1}$ act topologically nilpotent on ${M^{(1)}}^{\otimes_{W(k)}k}$ making it a $MF^{fd,p'}$-object. 
Secondly, we pass to the eigenspace on which the symmetric group operates according to the sign character $S_k\rightarrow\{\pm1\}$.\\

In a similar way we establish that ${M^{(0)}}^{\otimes_{W(k)}k-1}\otimes_{W(k)}M^{(k)}\in MF^{fd,p'}$. We obtain, using \cite[Remarque 6.13(b)]{fontaine} again :
$$U({M^{(0)}}^{\otimes_{W(k)}k-1}\otimes_{W(k)}M^{(k)})\cong U(M^{(0)})^{\otimes_{\z_p}k-1}\otimes_{\z_p}U(M^{(k)}),$$
Then we consider an appropriate $\O_{L_\gq}^{\otimes_{\z_p}k}$-eigenspace to arrive at the isomorphism \eqref{fix}.
\end{proof}

\subsubsection{Computation of $\d^*(Y_\xi^{(k)}[\gq^\infty])$} 
Let us return to some $k$-valued point $\xi:\Spec k\rightarrow\M_\gp^\delta$. By a lift of $\xi$ we mean:
\begin{itemize}
\item
a point $\eta$ of $\M_\gp^\delta$ over a complete discrete valuation ring $R$ with perfect residue field $k'$, and 
\item
an imbedding of $k$ into $k'$, such that the restriction of $\eta$ to $\Spec k'\hookrightarrow\Spec R$ 
agrees with the composition of $\xi$ with $\Spec k'\rightarrow\Spec k$.
\end{itemize}
In case $\delta=(0\times1)$ and $\ch(R)=0$, every lift determines $R[\frac{1}{p}]$-valued points on 
$M^{(k)}\times_EE_\gp$, by applying $g^{(k)}$. Interchangeably one can think of them as objects:
\begin{equation}
\label{Yps}
(\Y_\eta^{(k)},\iota^{(k)},\q_{>0}\lambda^{(k)},\ebar^{(k)})
\end{equation}
over $R$, because the moduli spaces $\M_\gp^{(k)}$ are all proper. We need to study the $\gr$-adic and 
crystalline cohomologies of these $\Y_\eta^{(k)}$'s. In the rest of this chapter we confine ourself to the case:
\begin{itemize}
\item
$R=W(k')$, and
\item
$k$ is algebraically closed in $k'$.
\end{itemize}
If this holds we call $\eta$ unramified, we consequently have graded contravariant Dieudonn\'e modules $M^{(k)}=\d^*(\Y_\eta^{(k)}[\gq^\infty])$. Observe that:
$$H_{dR}^1(\Y_\eta^{(k)})\cong\d^*(\Y_\eta^{(k)}[p^\infty])\cong\d^*(\Y_\eta^{(k)}[\gq^\infty])\oplus\d^*(\Y_\eta^{(k)}[\gq^{*\infty}])$$
which follows from \cite[chapter V, Theorem 2.1]{crystals} and \cite[Corollary 7.13]{mazur}. The above isomorphisms are $L$-equivariant 
and so give rise to $H_{dR}^1(\Y_\eta)_\sigma\cong M_\sigma$, where we use $H_{dR}^1(\dots)_\sigma$ (resp. $\d^*(\dots)_\sigma$) 
to denote the $\tau^\sigma\circ\iota$-eigenspace of $H_{dR}^1(\dots)$ (resp. $\d^*(\dots)$), if $\iota$ is as in \eqref{nzwei}. Now recall our 
$K(k')\otimes L$-linear isomorphism $t_{\eta,dR}^{(k)}$ of lemma \ref{vvier}, which by passing to the eigenspaces yields an isomorphism: 
\begin{equation}
\label{lattice}
\q\otimes({M_\sigma^{(0)}}^{\otimes_{W(k')}1-k}\otimes_{W(k')}\bigwedge_{W(k')}^kM_\sigma^{(1)})\stackrel{t_{\eta,dR,\sigma}^{(k)}}{\rightarrow}\q\otimes M_\sigma^{(k)}.
\end{equation}
Note that the two formulas \eqref{kotteins} and \eqref{kottnull} that entered crucially into the determinant condition (c) of the 
modular definition of $\M^{(0\times1)}$, imply immediately that the two $p$-divisible groups $\G_\gq^{(1)}=\Y_\eta^{(1)}[\gq^\infty]$ 
and $\G_\gq^{(0)}=\Y_\eta^{(0)}[\gq^\infty]$ with $\O_{L_\gq}$-operation satisfy the assumptions of lemma \ref{BARS}.

\begin{lem}
\label{vdreiz}
Assume that $p\geq k+1$, and let $\eta\in\M_\gp^{(0\times1)}(W(k'))$ be an unramified lift of $\xi\in\M_\gp^{(0\times1)}(k)$. Then 
there exists a $k$-valued point, say $(Y_\xi^{(k)},\iota^{(k)},\q_{>0}\lambda^{(k)},\ebar^{(k)})$ on $\M_\gp^{(k)}$, such that:
\begin{itemize}
\item[(i)]
\eqref{Yps} is a lift of $(Y_\xi^{(k)},\iota^{(k)},\q_{>0}\lambda^{(k)},\ebar^{(k)})$
\item[(ii)]
there exists an isomorphism of graded $k$-Dieudonn\'e modules:
\begin{equation*}
\d^*(Y_\xi^{(0)}[\gq^\infty])_\sigma^{\otimes_{W(k)}1-k}\otimes_{W(k)}\bigwedge_{W(k)}^k\d^*(Y_\xi^{(1)}[\gq^\infty])_\sigma
\stackrel{t_{\xi,cris,\sigma}^{(k)}}{\rightarrow}\d^*(Y_\xi^{(k)}[\gq^\infty])_\sigma
\end{equation*}
which induces $t_{\eta,dR,\sigma}^{(k)}$, when base changed to $K(k')$.
\end{itemize}
\end{lem}
\begin{proof}
Let us write $\Ybar_\eta^{(1)}$, $\Ybar_\eta^{(0)}$, and $\Ybar_\eta^{(k)}$ for the special fibres of $\Y_\eta^{(1)}$, $\Y_\eta^{(0)}$, and $\Y_\eta^{(k)}$. 
Let us assume for a while that $k'=k$, in which case clearly $\Ybar_\eta^{(k)}=Y_\xi^{(k)}$, and $t_{\eta,dR,\sigma}^{(k)}=t_{\xi,cris,\sigma}^{(k)}$. 
All we have to do is check that the isomorphism \eqref{lattice} preserves $\phi$-structure and the $W(k)$-lattice.\\

We begin with the $\phi$-structure. Let us regard the collection of maps $t_{\eta,dR,\sigma}^{(k)}$ as a single element, 
say $t_{\eta,dR,\gq}^{(k)}$, in the filtered $K(k)$-isocrystal $\q\otimes M$, where $M$ is the filtered $W(k)$-module:
$${M^{(0)}}^{\otimes_{W(k)}k-1}\otimes_{W(k)}(\bigwedge_{W(k)}^k{M^{(1)}})^*\otimes_{W(k)}M^{(k)}.$$ 
First of all we have $T_p\Y_\eta^{(k)}[\gq^\infty]=U(M^{(k)})$, by \eqref{fact1}. Write
$$m_{\xi,\gq}^{(k)}:U(M^{(0)})^{\otimes_{\O_{L_\gq}}1-k}\otimes_{\O_{L_\gq}}\bigwedge_{\O_{L_\gq}}^kU(M^{(1)})\stackrel{\cong}{\rightarrow}U(M^{(k)}),$$
for the $\gq$-component of $m_{\xi,B}^{(k)}\otimes1_{\zd}$. By lemma \ref{vvier}, part (i), it is 
$\Gal(K(k)^{ac}/K(k))$-equivariant. However, exposing $t_{\eta,dR,\gq}^{(k)}$ to the functor 
$$?\mapsto Fil^0\Hom_{K(k)}(\q\otimes ?,B_{cris}(K(k)^{ac}))^{\phi=1}=\q\otimes U(?),$$
gives $m_{\xi,\gq}^{(k)}$ too, by proposition \ref{dR}. Therefore $t_{\eta,dR,\gq}^{(k)}$ is actually contained in the $\Gal(K(k)^{ac}/K(k))$-invariants of the space 
$$Fil^0\Hom_{K(k)}(\q\otimes M^*,B_{cris}(K(k)^{ac}))^{\phi=1}.$$
As all $\Gal(K(k)^{ac}/K(k))$-invariants of $B_{cris}(K(k)^{ac})$ are in $K(k)$, we infer that $t_{\eta,dR,\gq}^{(k)}$ is 
contained in $Fil^0\Hom_{K(k)}(\q\otimes M^*,K(k))^{\phi=1}$, which is precisely the set of homomorphisms from 
$({M^{(0)}}^{\otimes_{W(k)}k-1})^*\otimes_{W(k)}\bigwedge_{W(k)}^k{M^{(1)}}$ to $M^{(k)}$, as filtered isocrystals.\\

Now write $N_\sigma^{(k)}\subset\q\otimes M_\sigma^{(k)}$ for the image lattice of $t_{\eta,dR,\sigma}^{(k)}$. By lemma \ref{BARS} 
this is a graded object in $MF^{fd,2}$ and there exists a $p$-divisible group $\G_\gq^{(k)}/W(k)$ with $\O_{L_\gq}$-operation 
such that $\d^*(\G_\gq^{(k)})=N^{(k)}$, moreover as $N^{(k)}$ is isogenous to $M^{(k)}$ there is a canonical quasi-isogeny: 
$$f:\Y_\eta^{(k)}[\gq^\infty]\rightarrow\G_\gq^{(k)}.$$
This quasi-isogeny induces an isomorphism:
$$T_p\Y_\eta^{(k)}[\gq^\infty]\cong T_p\G_\gq^{(k)},$$
which one sees by applying $U$ to both $N^{(k)}$ and $M^{(k)}$. It follows that $f$ is an isomorphism, so that $N^{(k)}=M^{(k)}$.\\

We now come back to the case of a possibly smaller field $k$. As we know the result over $k'$ all we have to do is check that 
$t_{\eta,dR,\gq}^{(k)}$ and $\Ybar_\eta^{(k)}$ are defined over $k$. We begin with $t_{\eta,dR,\gq}^{(k)}$: Observe that the objects $M_\sigma^{(1)}$ 
and $M_\sigma^{(0)}$ have natural $W(k)$-structures, say $M_{\xi,\sigma}^{(1)}$ and $M_{\xi,\sigma}^{(0)}$, when regarded as non-filtered 
Dieudonn\'e modules. Using $t_{\eta,dR,\sigma}^{(k)}$ this carries over to a $W(k)$-form $M_{\xi,\sigma}^{(k)}$ of $M_\sigma^{(k)}$. Let us 
denote the corresponding $p$-divisible groups by $Y_\xi^{(k)}[\gq^\infty]$. The existence of $t_{\xi,cris,\sigma}^{(k)}$ is then trivial.\\

To finish the proof we have to show that the $k$-form $Y_\xi^{(k)}[\gq^\infty]$ just exhibited comes from a $k$-form of $\Ybar_\eta^{(k)}$. To this end consider the 
$\Gal(K(k')^{ac}/K(k'))$-operation on $H_1(Y_\eta^{(1)}\times_{K(k')}K(k')^{ac},\z_\ell)$. It factors through the quotient $\Gal(k^{ac}/k)$, because the $\pmod p$-reduction 
of $Y_\eta^{(1)}$ is just $Y_\xi^{(1)}\times_kk'$. Consequently the $\Gal(K(k')^{ac}/K(k'))$-operation on $H_1(Y_\eta^{(k)}\times_{K(k')}K(k')^{ac},\z_\ell)$ factors 
through $\Gal(k^{ac}/k)$ too, by part (i) of lemma \ref{vvier}. By using a theorem of Grothendieck \cite[Proposition 4.4]{grothendieck1} we deduce the existence of 
an abelian variety $\Abar$ over $k$ together with a $p$-isogeny $f:\Abar\times_kk'\rightarrow\Ybar_\eta^{(k)}$, i.e. the abelian variety $\Ybar_\eta^{(k)}$ is isotrivial. 
Now note that the map $\d^*(f)$ from $\d^*(\Ybar_\eta^{(k)})=W(k')\otimes_{W(k)}(\bigoplus_\sigma M_{\xi,\sigma}^{(k)}\oplus(\bigoplus_\sigma M_{\xi,\sigma}^{(k)})^*))$ 
to $W(k')\otimes_{W(k)}\d^*(\Abar)$ is defined over $W(k)$ by lemma \ref{skel} below. This provides us with $Y_\xi^{(k)}=\Abar/\ker(f[p^\infty])$.
\end{proof}

\begin{lem}
\label{skel}
Let $k'$ be a perfect field of characteristic $p$. Let $k\subset k'$ be algebraically closed in $k'$. Let $M_1$ and $M_2$ be effective 
$W(k)$-isocrystals. Let $f':M_1'\rightarrow M_2'$ be a morphism where $M_1'$ and $M_2'$ are $W(k')$-isocrystals obtained by extension 
of scalars. Then $f'$ preserves the $W(k)$-structures of $M_1'$ and $M_2'$ and therefore induces a unique morphism $f:M_1\rightarrow M_2$.
\end{lem}
\begin{proof}
Assume for a while that $k'$ and $k$ are algebraically closed. Let $d=\max\{\rk_{W(k)}M_1,\rk_{W(k)}M_2\}!$, and write $I_i\subset\q\otimes M_i$ for the 
$K(\f_{p^d})$-vector spaces $\bigoplus_z\{x\in\q\otimes M_i'|\phi^d(x)=p^zx\}$. It is well known that these are $K(\f_{p^d})$-forms of both $\q\otimes M_i'$ 
and $\q\otimes M_i$. However $f'$ sends $I_1$ to $I_2$, and therefore it sends $\q\otimes M_1$ to $\q\otimes M_2$. Evidently $M_i=M_i'\cap\q\otimes M_i$ 
so $f'$ restricts to a map $f$ from $M_1$ to $M_2$.\\

Now let $k'$ be arbitrary, and let $k^{'ac}$ be an algebraic closure of $k'$. Let also $k^{ac}$ be the algebraic closure of $k$ in $k^{'ac}$. By the preceding 
argument we know already that $f^{'ac}:M_1^{'ac}\rightarrow M_2^{'ac}$, the scalar extension of $f'$ to $W(k^{'ac})$, preserves the $W(k^{ac})$-structures 
$M_1^{ac}$ and $M_2^{ac}$ on $M_1^{'ac}$ and $M_2^{'ac}$. As $k=k'\cap k^{'ac}$ we get $M_i=M_i^{ac}\cap M_i'$ so this suffices.
\end{proof}

\section{$Y^{(k)}$ in Characteristic $p$}
\label{mock2}
In this chapter we study the integral properties of $Y^{(k)}$, we show that our map $g^{(k)}$ can be 
extended into characteristic $p$, and we show further that the map so obtained preserves isogeny classes.
 
\subsection{Displays and their Crystals}
Let $p$ be a prime, let $R$ be a $W(\f_{p^r})$-algebra. We assume that it is $p$-adically complete and separated, i.e. that $R\rightarrow\lim\limits_{\leftarrow j}R/p^jR$ 
is an isomorphism. In ~\cite{zink2} displays and $3n$-displays were introduced. We need to endow them with gradings:

\begin{defn}
\label{E8}
Let $P_\sigma$ and $Q_\sigma$ be $W(R)$-modules for $\sigma\in\z/r\z$. A $2r+2$-tuple $(P_\sigma,Q_\sigma,F,V^{-1})$ is called a $\z/r\z$-graded 
$3n$-display (resp. display), if $(P,Q,F,V^{-1})$ is a $3n$-display (resp. display) in the usual sense where $P=\bigoplus_{\sigma\in\z/r\z}P_\sigma$, 
$Q=\bigoplus_{\sigma\in\z/r\z}Q_\sigma$ and $F$ and $V^{-1}$ act homogeneously of degree one, i.e.
$$F:P_\sigma\rightarrow P_{\sigma+1}$$ 
and 
$$V^{-1}:Q_\sigma\rightarrow P_{\sigma+1}.$$
We write $\End_{W(\f_{p^r})}(P)$ for the ring of endomorphism that preserve the grading and similarly for $\Hom_{W(\f_{p^r})}(.,..)$.
\end{defn}

\begin{rem}
\label{E7}
One has graded versions of the usual properties, for example:
\begin{itemize}
\item
We have $I_RP_\sigma\subset Q_\sigma\subset P_\sigma$, where $I_R$ denotes the kernel of the $0$th ghost map. Moreover 
$P_\sigma$ is a finitely generated projective module and $P_\sigma/Q_\sigma$ is a direct factor of $P_\sigma/I_RP_\sigma$.
\item
From $W(R)$-linearization of $F$ we obtain a map of $W(R)$-modules
\begin{equation}
\label{Kskel}
F^\sharp:W(R)\otimes_{F,W(R)}P_{\sigma-1}\rightarrow P_\sigma;w\otimes x\mapsto w\otimes F(x).
\end{equation}
If $pR=0$ then one can consider the pull back of $(P_\sigma,Q_\sigma,F,V^{-1})$ by means of the absolute Frobenius. Moreover, 
this graded $3n$-display has the shape $(W(R)\otimes_{F,W(R)}P_\sigma,\dots)$, and there exists a map of $3n$-displays, denoted by
$$Ver_P:(W(R)\otimes_{F,W(R)}\bigoplus_\sigma P_{\sigma-1},\dots)\rightarrow(\bigoplus_\sigma P_\sigma,\dots)$$
in \cite[example 23]{zink2}, which is the map \eqref{Kskel} on the underlying modules. Notice that $Ver_P$ shifts the grading by one.
\item
every graded $3n$-display has a graded normal decomposition $P_\sigma=L_\sigma\oplus T_\sigma$
\end{itemize}
\end{rem}

From now on all displays will be graded. By Zink's theory displays have a crystalline nature: Assume that one is given a graded display $(P_\sigma,Q_\sigma,F,V^{-1})$ 
over the ring $S$ with pd-ideal $\ga$. Let $(\Pbar_\sigma,\Qbar_\sigma,F,V^{-1})$ be the base change to $R=S/\ga $. According to \cite[chapter 1.4]{zink2} 
one obtains a lift of $\ga$ to $W(S)$, because $\ga$ has divided powers, moreover $V^{-1}$ has a unique extension to a $^F$-linear map:
$$V^{-1}:\Qd_\sigma=Q_\sigma+\ga P_\sigma\rightarrow P_{\sigma+1},$$
such that $V^{-1}(\ga P_\sigma)=0$. The structure $(P_\sigma,\Qd_\sigma,F,V^{-1})$ thus obtained we will call a graded triple. Its significance stems from the 
following crystalline functoriality property which is proved in \cite[Theorem 46]{zink2}: If $p$ is nilpotent in $S$ and if $(P_{1,\sigma},Q_{1,\sigma},F,V^{-1})$ and 
$(P_{2,\sigma},Q_{2,\sigma},F,V^{-1})$ are two graded displays over $S$ with a morphism of graded displays $\alpha_\sigma:\Pbar_{1,\sigma}\rightarrow\Pbar_{2,\sigma}$ 
over $R$, then there are unique $W(S)$-linear maps $\tilde\alpha_\sigma:P_{1,\sigma}\rightarrow P_{2,\sigma}$ which send $\Qd_{1,\sigma}$ to $\Qd_{2,\sigma}$, 
commute with $F$ and $V^{-1}$, and have $\pmod{W(\ga)}$-reduction equal to $\alpha_\sigma$.

\begin{prop}
\label{E1}
Let $(P_\sigma^{(1)},Q_\sigma^{(1)},F,V^{-1})$ be a graded $3n$-display, with $\rk_R(Q_\sigma^{(1)}/I_RP_\sigma^{(1)})\leq1$ 
for all $\sigma$. Let $P_\sigma^{(1)}=L_\sigma^{(1)}\oplus T_\sigma^{(1)}$ be a graded normal decomposition. Write 
$$P_\sigma^{(k)}=\bigwedge_{W(R)}^kP_\sigma^{(1)}=L_\sigma^{(k)}\oplus T_\sigma^{(k)},$$
with $L_\sigma^{(k)}=L_\sigma^{(1)}\otimes_{W(R)}\bigwedge_{W(R)}^{k-1}T_\sigma^{(1)}$, and $T_\sigma^{(k)}=\bigwedge_{W(R)}^kT_\sigma^{(1)}$. 
Define maps $V^{-1}:L_\sigma^{(k)}\rightarrow P_{\sigma+1}^{(k)}$, and $F: P_\sigma^{(k)}\rightarrow P_{\sigma+1}^{(k)}$ by
\begin{eqnarray}
\label{Versch}
&&V^{-1}(x_1\wedge\dots\wedge x_k)=V^{-1}(x_1)\wedge F(x_2)\wedge\dots\wedge F(x_k)\\
\label{Froben}
&&F(x_1\wedge\dots\wedge x_k)=F(x_1)\wedge\dots\wedge F(x_k)
\end{eqnarray}
then:
\begin{itemize}
\item[(i)]
there is a unique extension of $V^{-1}$ to $Q_\sigma^{(k)}=L_\sigma^{(k)}\oplus I_RT_\sigma^{(k)}$ such that $(P_\sigma^{(k)},Q_\sigma^{(k)},F,V^{-1})$ 
is a $\z/r\z$-graded $3n$-display, with graded normal decomposition $L_\sigma^{(k)}\oplus T_\sigma^{(k)}$.
\item[(ii)]
$(P_\sigma^{(k)},Q_\sigma^{(k)},F,V^{-1})$ is independent of the choice of the normal decomposition, $L_\sigma^{(1)}\oplus T_\sigma^{(1)}$.
\end{itemize}
Moreover, if $\rk_R(Q_{\sigma_0}^{(1)}/I_RP_{\sigma_0}^{(1)})=0$ for at least one $\sigma_0$, then $P^{(k)}$ is a display.
\end{prop}
\begin{proof}
Note that $I_RT_\sigma^{(k)}=I_R\otimes_{W(R)}T_\sigma^{(k)}$. Hence we are allowed to define the operator $V^{-1}$ on it by the formula $V^{-1}(^Vwx)=wFx$, for 
$w\in W(R)$ and $x\in T_\sigma^{(k)}$. It is evident that $F$ and $V^{-1}$ thus defined are $^F$-linear maps on $P^{(k)}$ and $Q^{(k)}$, satisfying $V^{-1}(^Vwx)=wFx$. 
It is also evident that $P^{(k)}$ is a finitely generated projective $W(R)$-module and that $P^{(k)}/Q^{(k)}$ is a projective $R$-module. Now consider the operator 
$U:P_\sigma^{(1)}\rightarrow P_{\sigma+1}^{(1)}$ defined by $V^{-1}$ on $L_\sigma^{(1)}$ and by $F$ on $T_\sigma^{(1)}$. According to \cite[Lemma 9]{zink2} 
it is a $^F$-linear isomorphism (in fact $\z/r\z$-homogeneous of degree one), as is therefore the operator $U:P_\sigma^{(k)}\rightarrow P_{\sigma+1}^{(k)}$ that 
we define by $U(x_1\wedge\dots\wedge x_k)=U(x_1)\wedge\dots\wedge U(x_k)$. It follows that $V^{-1}:Q_\sigma^{(k)}\rightarrow P_{\sigma+1}^{(k)}$ is a 
$^F$-linear epimorphism, because $U(l+t)=V^{-1}(l+{^V1}t)$. This shows $P^{(k)}$ is a graded $3n$-display with normal decomposition $L^{(k)}\oplus T^{(k)}$.\\

In order to check the second assertion just notice that $Q_\sigma^{(k)}$ is generated by the elements of the form $x_1\wedge\dots\wedge x_k$, where 
$x_1\in Q_\sigma^{(1)}$ and $x_2,\dots,x_k\in P_\sigma^{(1)}$. Finally it is clear that the nilpotence condition for $P^{(k)}$ is granted if some $L_{\sigma_0}^{(k)}$ is $0$, 
because the image of $V^\sharp:P_{\sigma_0+1}^{(k)}\rightarrow W(R)\otimes_{F,W(R)}P_{\sigma_0}^{(k)}$ is contained in $pW(R)\otimes_{F,W(R)}P_{\sigma_0}^{(k)}$ then.
\end{proof}

\begin{rem}
\label{E2}
Let $i:R\rightarrow R'$ be a homomorphism of $W(\f_{p^r})$-algebras which are separated and complete with respect to the $p$-adic topology. Assume the graded 
display $(P_\sigma^{(1)},Q_\sigma^{(1)},F,V^{-1})$ satisfies the assumption of proposition \ref{E1}, and thus gives rise to unnormalized graded exterior power 
displays $(P_\sigma^{(k)},Q_\sigma^{(k)},F,V^{-1})$. Then, the same is true for its base change $(P_\sigma^{'(1)},Q_\sigma^{'(1)},F,V^{-1})$ to $R'$, and it gives 
rise to $(P_\sigma^{'(k)},Q_\sigma^{'(k)},F,V^{-1})$, with a graded normal decomposition of the form: $L_\sigma^{'(k)}=W(R')\otimes_{W(R)}L_\sigma^{(k)}$ and 
$T_\sigma^{'(k)}=W(R')\otimes_{W(R)}T_\sigma^{(k)}$, i.e. the base change to $R'$ of $(P_\sigma^{(k)},Q_\sigma^{(k)},F,V^{-1})$ is $(P_\sigma^{'(k)},Q_\sigma^{'(k)},F,V^{-1})$. We will frequently use this for the special case where $R'=R$ is a $\f_{p^r}$-algebra and where $i$ is the absolute Frobenius. 
\end{rem}

In the sequel we call $P_\sigma^{(k)}$ the unnormalized graded exterior power display to $P_\sigma^{(1)}$. The category of displays is an additive category, 
but observe that our rule to produce $P_\sigma^{(k)}$ from $P_\sigma^{(1)}$ is not additive. However at least it is functorial in the following sense: 

\begin{prop}
\label{E5}
Let $P_{1,\sigma}^{(1)}$ and $P_{2,\sigma}^{(1)}$ be graded displays over $R$, both of which fulfill the assumptions of the previous 
proposition (N.B.: we do not assume $\rk_R(Q_{1,\sigma}^{(1)}/I_RP_{1,\sigma}^{(1)})=\rk_R(Q_{2,\sigma}^{(1)}/I_RP_{2,\sigma}^{(1)})$). Let
$$\alpha_\sigma^{(1)}:P_{1,\sigma}^{(1)}\rightarrow P_{2,\sigma}^{(1)}$$
be a homomorphism between them. Then:
\begin{itemize}
\item[(i)]
the graded $W(R)$-linear map
$$\alpha_\sigma^{(k)}:P_{1,\sigma}^{(k)}\rightarrow P_{2,\sigma}^{(k)};x_1\wedge\dots\wedge x_k\mapsto\alpha_\sigma^{(1)}(x_1)\wedge\dots\wedge\alpha_\sigma^{(1)}(x_k)$$
is a homomorphism of graded displays too.
\item[(ii)]
Assume $pR=0$. If $P_{1,\sigma}^{(1)}=W(R)\otimes_{F,W(R)}P_{2,\sigma-1}$ and $\alpha^{(1)}=Ver_{P_2^{(1)}}$, 
in the sense of \cite[example 23]{zink2} then the same holds for $P_{1,\sigma}^{(k)}$ and $\alpha^{(k)}$.
\item[(iii)]
Let $T_{i,\sigma}^{(1)}$ be the graded triples of $P_{i,\sigma}^{(1)}$ over some pd-thickening $S\rightarrow R$ in which $p$ 
is nilpotent. Let $\tilde\alpha_\sigma^{(1)}$ be the (by \cite[Theorem 46]{zink2} unique) morphism of graded triples that lifts 
$\alpha_\sigma^{(1)}$. Then $\bigwedge_S^kT_{i,\sigma}^{(1)}$ are the graded triples of $P_{i,\sigma}^{(k)}$ over $S$, and
$$\tilde\alpha_\sigma^{(k)}:\bigwedge_S^kT_{1,\sigma}^{(1)}\rightarrow\bigwedge_S^kT_{2,\sigma}^{(1)};
x_1\wedge\dots\wedge x_k\mapsto\tilde\alpha_\sigma^{(1)}(x_1)\wedge\dots\wedge\tilde\alpha_\sigma^{(1)}(x_k)$$
is the morphism of graded triples lifting $\alpha_\sigma^{(k)}$.
\end{itemize}
\end{prop}
\begin{proof}
The proof of the previous proposition discloses that $Q_{1,\sigma}^{(k)}$ is the $W(R)$-submodule of $P_{1,\sigma}^{(k)}$ that is generated by the elements 
$x_1\wedge x_2\wedge\dots\wedge x_k$ where $x_1\in Q_{1,\sigma}^{(1)}$ and $x_2,\dots,x_k\in P_{1,\sigma}^{(1)}$. It is clear that $\alpha_\sigma^{(k)}$ 
maps this into $Q_{2,\sigma}^{(k)}$, and equations \eqref{Froben} and \eqref{Versch} show it commutes with $F$ and $V^{-1}$, showing (i).\\

Now (ii) follows from equation \eqref{Froben}, together with the description of the $Ver_{P_2^{(k)}}$'s in \eqref{Kskel}. Statement (iii) is immediately clear.
\end{proof}

\begin{rem}
\label{E10}
Fix a graded $W(R)$-module $P^{(0)}$ such that $P_\sigma^{(0)}$ is projective of rank one together with $^F$-linear isomorphisms:
$$F:P_\sigma^{(0)}\rightarrow P_{\sigma+1}^{(0)}.$$
We will call this a ``multiplicative display'', following \cite[example 16]{zink2}. In the sequel we frequently need to twist and untwist displays by this kind of modules: 
If $(P_\sigma,Q_\sigma,F,V^{-1})$ is a graded display, then so is $(P_\sigma^{(0)}\otimes_{W(R)}P_\sigma,P_\sigma^{(0)}\otimes_{W(R)} Q_\sigma,F,V^{-1})$, 
where for every $x_0\in P_\sigma^{(0)}$ the operators $F$, and $V^{-1}$ act as: $F(x_0\otimes x)=F(x_0)\otimes F(x)$, if $x\in P_\sigma$, and 
$V^{-1}(x_0\otimes x)=F(x_0)\otimes V^{-1}(x)$, if $x\in Q_\sigma$. In particular, if $(P_\sigma^{(1)},Q_\sigma^{(1)},F,V^{-1})$ satisfies the 
assumptions of proposition \ref{E1} and if $(P_\sigma,Q_\sigma,F,V^{-1})$ is the unnormalized graded exterior power, then we call 
$$(P_\sigma^{(0)^{\otimes_{W(R)}1-k}}\otimes_{W(R)}P_\sigma,P_\sigma^{(0)^{\otimes_{W(R)}1-k}}\otimes_{W(R)}Q_\sigma,F,V^{-1})$$ 
the normalized graded exterior power.
\end{rem}

\subsection{Exterior powers of formal groups}
\label{vneun}
From now on the ground $W(\f_{p^r})$-algebra $R$ is an excellent local ring. 
\begin{fact}
\label{was}
The functor $\BT$ from displays over $R$ to $p$-divisible formal groups over $R$, as constructed in \cite[chapter 3]{zink2}, is an equivalence of categories, 
cf. \cite[Theorem 103]{zink2}. Giving an operation of $W(\f_{p^r})$ on $\BT(P)$ is equivalent to giving a $\z/r\z$-grading on $P$, in the sense of definition \ref{E8}.
\end{fact}
\begin{fact}
\label{E14}
The tangent space of the formal group can be canonically recovered from the display by
\begin{equation*}
\Lie(\BT(\bigoplus_\sigma P_\sigma))_\sigma\cong P_\sigma/Q_\sigma,
\end{equation*}
where a subscript $\sigma\in\z/r\z$ means the eigenspace with respect to the $W(\f_{p^r})$-operation. 
\end{fact}
\begin{fact}
\label{E15}
Assume $pR=0$. Then the natural isogeny:
\begin{equation*}
\BT(Ver_P):\BT(\bigoplus_\sigma P_{\sigma-1})\times_{R,F}R\rightarrow\BT(\bigoplus_\sigma P_\sigma)
\end{equation*}
is identical to the Verschiebung of the $p$-divisible group, by virtue of 
\cite[Proposition 87]{zink2}.
\end{fact}
\begin{fact}
\label{E16}
The result \cite[Lemma 93]{zink2} tells us that the Lie-algebra of the Grothendieck-Messing universal vector extension can be recovered as:
\begin{equation*}
H_1^{dR}(\BT(\bigoplus_\sigma P_\sigma))_\sigma\cong P_\sigma/I_RP_\sigma,
\end{equation*}
this isomorphism is compatible with the natural projections to the left and right hand side of the isomorphism in fact \ref{E14}. 
\end{fact}
\begin{fact}
\label{wieso}
Let $S\rightarrow R$ is a pd-thickening in which $p$ is nilpotent. Then by \cite[Theorem 94]{zink2} the $\pmod{I_S}$-reduction of a lift of $\bigoplus_\sigma P_\sigma$ 
to a triple over $S$ is canonically isomorphic to the $S$-value of the crystal to $\BT(\bigoplus_\sigma P_\sigma)$ as defined in Messing's book \cite{crystals}.
\end{fact}

Let us derive some consequences for $p$-divisible groups of the kind $\G^{(k)}=\BT(\bigoplus_\sigma P_\sigma^{(k)})$, where 
$P_\sigma^{(k)}={P_\sigma^{(0)}}^{\otimes_{W(R)}1-k}\otimes_{W(R)}\bigwedge_{W(R)}^kP_\sigma^{(1)}$ is the (normalized) 
exterior power as in proposition \ref{E1} , (remark \ref{E10}). Fact \ref{E16} gives a specific isomorphism, say:
$$\md_\sigma^{(k)}:H_1^{dR}(\G^{(0)})_\sigma^{\otimes_R1-k}\otimes_R\bigwedge_R^kH_1^{dR}(\G^{(1)})_\sigma\rightarrow H_1^{dR}(\G^{(k)})_\sigma,$$
and it has the feature that it induces an isomorphism:
$$\Lie(\G^{(0)})_\sigma^{\otimes_\Rd1-k}\otimes_\Rd\bigwedge_\Rd^k\Lie(\G^{(1)})_\sigma\rightarrow\Lie(\G^{(k)})_\sigma,$$
Let us pin down the crystalline nature of $\md_\sigma^{(k)}$:

\begin{lem}
\label{vzehn}
Assume that $R$ is a mixed characteristic complete discrete valuation ring with uniformizer $\varpi$ and perfect residue field $k$. 
Consider two graded displays $(P_{i,\sigma}^{(1)},Q_{i,\sigma}^{(1)},F,V^{-1})$, ($i\in\{1,2\}$) as in proposition \ref{E1}, and let 
$(P_{i,\sigma}^{(0)},I_RP_{i,\sigma}^{(0)},F,V^{-1})$ be displays of the kind considered in remark \ref{E10}. Let $\G_i^{(k)}/R$ 
be the $p$-divisible formal group corresponding to the normalized exterior powers $(P_{i,\sigma}^{(k)},Q_{i,\sigma}^{(k)},F,V^{-1})$, 
with $W(\f_{p^r})$-action, of course. Write $\Gbar_i^{(k)}$ for the special fibre over $k$. Assume that there is a $K(\f_{p^r})$-linear quasi-isogeny
$$\alpha^{(0)}\in\Hom_{K(\f_{p^r})}^0(W(k)\otimes_{W(R)}P_1^{(0)},W(k)\otimes_{W(R)}P_2^{(0)}),$$
and an element
$$\alpha^{(1)}\in\Hom_{K(\f_{p^r})}^0(W(k)\otimes_{W(R)}P_1^{(1)},W(k)\otimes_{W(R)}P_2^{(1)})$$
and let
$$\alpha^{(k)}\in\Hom_{K(\f_{p^r})}^0(W(k)\otimes_{W(R)}P_1^{(k)},W(k)\otimes_{W(R)}P_2^{(k)})$$
be induced by part (i) of proposition \ref{E5} on the functoriality of the unnormalized graded exterior power, composed with the action of $\alpha^{(0)}$. Then the diagram
$$\begin{CD}
\q\otimes H_1^{dR}(\G_1^{(0)})_\sigma^{\otimes_R1-k}\otimes_R\bigwedge_R^kH_1^{dR}(\G_1^{(1)})_\sigma
@>{\md_{1,\sigma}^{(k)}}>>\q\otimes H_1^{dR}(\G_1^{(k)})_\sigma\\
@VVV@VVV\\
\q\otimes H_1^{dR}(\G_2^{(0)})_\sigma^{\otimes_R1-k}\otimes_R\bigwedge_R^kH_1^{dR}(\G_2^{(1)})_\sigma
@>{\md_{2,\sigma}^{(k)}}>>\q\otimes H_1^{dR}(\G_2^{(k)})_\sigma\\
\end{CD}$$
commutes for all $\sigma$, where the vertical maps are obtained from $\BT(\alpha^{(1)})\in\Hom_{K(\f_{p^r})}^0(\Gbar_1^{(1)},\Gbar_2^{(1)})$, 
$\BT(\alpha^{(0)})\in\Hom_{K(\f_{p^r})}^0(\Gbar_1^{(0)},\Gbar_2^{(0)})$, and $\BT(\alpha^{(k)})\in\Hom_{K(\f_{p^r})}^0(\Gbar_1^{(k)},\Gbar_2^{(k)})$ 
by using the crystalline functoriality of $H_1^{dR}$ in the sense of Grothendieck-Messing.
\end{lem}
\begin{proof}
Without loss of generality we assume that $\alpha^{(0)}$ is an isomorphism, because multiplication by a scalar in $K(\f_{p^r})^\times$ 
does not change the statement. For the same reason we assume without loss of generality that $\alpha_\sigma^{(1)}$ comes from a graded 
homomorphism from $W(k)\otimes_{W(R)}P_{1,\sigma}^{(1)}$ to $W(k)\otimes_{W(R)}P_{2,\sigma}^{(1)}$ and in fact lifts to, say 
$$\gamma_\sigma^{(1)}:W(R/pR)\otimes_{W(R)}P_{1,\sigma}^{(1)}\rightarrow W(R/pR)\otimes_{W(R)}P_{2,\sigma}^{(1)}.$$
over the artinian ring $R/pR$, cf. \cite[Proposition 40]{zink2}. By proposition \ref{E5} there are associated morphisms $\gamma_\sigma^{(k)}$ between the $W(R/pR)\otimes_{W(R)}P_{i,\sigma}^{(k)}$'s. Now $S=R/\varpi^jR\rightarrow R/pR$ is a pd-thickening, if $p$ divides $\varpi^j$. Let us fix such a $j$. Then we get a map of graded triples:
$$\widetilde{\gamma_\sigma^{(k)}}:W(S)\otimes_{W(R)}P_{1,\sigma}^{(k)}\rightarrow W(S)\otimes_{W(R)}P_{2,\sigma}^{(k)},$$
induced from $\gamma_\sigma^{(k)}$'s action on the displays over $R/pR$. Passage from $W(S)$ to $S$, as in fact \ref{E16}, brings us to the maps 
$$\widetilde{\gamma_\sigma^{(k)}}\pmod{I_S}:H_1^{dR}(\G_1^{(k)}\times_RS)_\sigma\rightarrow H_1^{dR}(\G_2^{(k)}\times_RS)_\sigma,$$
which agree with the ones defined by the procedure in \cite{crystals}, recall that $\G_i^{(k)}\times_RS$ 
is $\BT(\bigoplus_\sigma W(S)\otimes_{W(R)}P_{i,\sigma}^{(k)})$. However, the diagrams
$$\begin{CD}
W(S)\otimes_{W(R)}(P_{1,\sigma}^{(0)^{\otimes_{W(R)}1-k}}\otimes_{W(R)}\bigwedge_{W(R)}^kP_{1,\sigma}^{(1)})
@>{=}>>W(S)\otimes_{W(R)}P_{1,\sigma}^{(k)}\\
@V{\tilde\gamma_\sigma^{(k)}}VV@V{\widetilde{\gamma_\sigma^{(k)}}}VV\\
W(S)\otimes_{W(R)}(P_{2,\sigma}^{(0)^{\otimes_{W(R)}1-k}}\otimes_{W(R)}\bigwedge_{W(R)}^kP_{2,\sigma}^{(1)})
@>{=}>>W(S)\otimes_{W(R)}P_{2,\sigma}^{(k)}\\
\end{CD}$$
do commute, which one sees by  applying part (iii) of proposition \ref{E5} to the graded displays 
$W(S)\otimes_{W(R)}(P_{i,\sigma}^{(0)^{\otimes_{W(R)}-1}}\otimes_{W(R)}P_{i,\sigma}^{(1)})$. If we let $j$ tend to $\infty$ we get the result.
\end{proof}

\subsection{An Extension Theorem}

Before we state our theorem on the extension of $Y^{(k)}$ over $\M_\gp^{(0\times1)}$, we would like to say that there has already been a lot of activity in the subject of 
extending abelian schemes over higher dimensional base schemes by C.-L. Chai, G. Faltings, O. Gabber, A. Grothendieck, J. Milne, B. Moonen, A. Ogus, M. Raynaud, 
A. Vasiu and many others, we certainly do not pretend to prove anything new in this section. Let us begin with some rather trivial remarks on the local structure of the 
integral model $\M_\gp^{(0\times1)}$: We fix a closed point $x\in\M_\gp^{(0\times1)}$ and write $R_x=\O_{\M_\gp^{(0\times1)},x}$ for the local ring. Let $k$ be the residue 
field of $R_x$, it is naturally a finite extension of $\O_E/\gp=\f_{p^r}$, write $\xi:\Spec k\rightarrow\M_\gp^{(0\times1)}$ for the associated morphism. Assume that (T.1) holds. When writing $\Lie(\dots)_\sigma$ (resp. $H_1^{dR}(\dots)_\sigma$) for the $\tau^\sigma\circ\iota\circ*$-eigenspace of $\Lie(\dots)$ (resp. $H_1^{dR}(\dots)$) we have
\begin{eqnarray*}
\rk_{R_x}\Lie(Y_x^{(1)})_\sigma=\begin{cases}n-1&\text{if }\sigma\in\Omega\\n&\text{if }\sigma\notin\Omega\end{cases},&&\rk_{R_x}H_1^{dR}(Y_x^{(1)})_\sigma=n,
\end{eqnarray*}
for all $\sigma\in\{0,\dots,r-1\}$, where $Y_x^{(1)}$ is $Y^{(1)}\times_{\M_\gp^{(0\times1)}}R_x$.\\
Let also $\Rd_x$ be the completion of $R_x$ at the maximal ideal $\gm$. According to the Serre-Tate theorem we have the following description of the 
deformation functor prorepresented by $\Rd_x$: It sends $N\in nil$ to the set $\Df_x(N)$ of deformations of $\Gbar_{x,\gq}^{(1)}:=Y_\xi^{(1)}[\gq^\infty]$ 
as a $p$-divisible group with $\O_{L_\gq}$-operation. We put $\G_{x,\gq}^{(1)}=\Yd_x^{(1)}[\gq^\infty]$ for the universal deformation over it, where 
$\Yd_x^{(1)}=Y_x^{(1)}\times_{R_x}\Rd_x$, and similarily for $\G_{x,\gq}^{(0)}$. \\
Again we need to shift from cohomology to homology now: Suppose that $\eta:\Rd_x\rightarrow R$ is a characteristic $0$ lift of $\xi$, then we have
$$H_1^{dR}(\lambda^{(k)}):H_1^{dR}(\Y_\eta^{(k)}[\gq^{*\infty}])_\sigma\stackrel{\cong}{\rightarrow}H_{dR}^1(\Y_\eta^{(k)}[\gq^\infty])_\sigma=H_{dR}^1(\Y_\eta^{(k)})_\sigma,$$
where the notation is from \eqref{Yps}. Let us therefore write $m_{\eta,dR,\sigma}^{(k)}$ for the $R[\frac{1}{p}]$-linear 
map which is obtained from $t_{\eta,dR,\sigma}^{(k)}$ by transport of structure, as in \eqref{polvier}.

\begin{thm}
\label{extend}
Assume that the data $L$, $V^{(0\times1)}$, $\psi^{(0\times1)}$, and $l$ are unramified at some prime $p$. Assume 
also that the condition (T.1) in section \ref{crys} is valid. For all $k=0,\dots,\min\{n,p-1\}$ the following holds:
\begin{itemize}
\item[(i)]
The map $g^{(k)}:M^{(0\times1)}\rightarrow\M^{(k)}$ extends to the whole of the integral model $\M_\gp^{(0\times1)}$ 
and therefore gives rise to an extension of the abelian schemes $Y^{(k)}$ over $\M_\gp^{(0\times1)}$.
\item[(ii)]
For every closed point $x$ there exist $\z/r\z$-graded displays $P_{x,\sigma}^{(1)}$ and $P_{x,\sigma}^{(0)}$ over $\Rd_x$ together with $\O_{L_{\gq^*}}$-isomorphism
$$m_{x,\BT}^{(k)}:\BT(\bigoplus_\sigma{P_{x,\sigma}^{(0)}}^{\otimes_{W(\Rd_x)}1-k}\otimes_{W(\Rd_x)}
\bigwedge_{W(\Rd_x)}^kP_{x,\sigma}^{(1)})\rightarrow\Yd_x^{(k)}[\gq^{*\infty}]$$
where $\Yd_x^{(k)}=Y^{(k)}\times_{\M_\gp^{(0\times1)}}\Rd_x$.
\item[(iii)] 
For every characteristic $0$ lift $\eta:\Rd_x\rightarrow R$, the isomorphism $\md_{x,\sigma}^{(k)}$ from the graded filtered module
$$H_1^{dR}(\Yd_x^{(0)}[\gq^{*\infty}])_\sigma^{\otimes_{\Rd_x}1-k}\otimes_{\Rd_x}\bigwedge_{\Rd_x}^kH_1^{dR}(\Yd_x^{(1)}[\gq^{*\infty}])_\sigma$$ 
to the graded filtered module
$$H_1^{dR}(\Yd_x^{(k)}[\gq^{*\infty}])_\sigma$$ 
obtained from (ii) together with the considerations in section \ref{vneun} induces $m_{\eta,dR,\sigma}^{(k)}$, when base changed to $R[\frac{1}{p}]$.
\end{itemize}
\end{thm}
\begin{proof}
Pick a closed point $x$. Let $\G_{x,\gq^*}^{(1)}$ and $\G_{x,\gq^*}^{(0)}$ be the duals of the $p$-divisible groups $\G_{x,\gq}^{(1)}$ and $\G_{x,\gq}^{(0)}$. 
According to fact \ref{was} we may pick two graded displays $P_{x,\sigma}^{(k)}$ over $\Rd_x$ (unique up to unique isomorphism) corresponding 
to the $\G_{x,\gq^*}^{(k)}$'s for $k\leq1$. Write $P_{x,\sigma}^{(k)}$ for the normalized graded exterior power $3n$-displays, which we 
are allowed to form by proposition \ref{E1} and remark \ref{E10}. Write $\G_{x,\gq^*}^{(k)}=\BT(\bigoplus_\sigma P_{x,\sigma}^{(k)})$ for 
$k\geq2$, and let $\G_{x,\gq}^{(k)}$ be the Serre-dual. Choose an isomorphism $\Rd_x\cong W(k)[[s_1,\dots,s_d]]$, let $k'$ be the perfect 
closure of the pure transcendental extension $k(t_1,\dots,t_d)$, and define $\kappa:\Rd_x\rightarrow W(k')$ by $\kappa(s_j)=p[t_j]$.\\

We next study properties of the contravariant Dieudonn\'e modules of the groups $\G_{\kappa,\gq}^{(k)}=\G_{x,\gq}^{(k)}\times_{\Rd_x,\kappa}W(k')$ 
and of their special fibres $\Gbar_{\kappa,\gq}^{(k)}$. We start with the observation that:
$$\d^*(\G_{\kappa,\gq}^{(0)})_\sigma^{\otimes_{W(k')}1-k}\otimes_{W(k')}\bigwedge_{W(k')}^k\d^*(\G_{\kappa,\gq}^{(1)})_\sigma\cong\d^*(\G_{\kappa,\gq}^{(k)})_\sigma.$$
Here is a proof using only first properties of displays: The techniques of Mazur and Messing produce the usual comparison isomorphism
$$\d^*(\G_{\kappa,\gq}^{(k)})_\sigma\cong H_1^{dR}(\G_{\kappa,\gq^*}^{(k)})_\sigma$$
of filtered $W(k')$-modules, if one combines \cite[proposition 7.12]{mazur} \cite[proposition 5.1]{mazur} and 
\cite[proposition 8.7]{mazur}. From remark \ref{E2} it is clear that we have a canonical diagram of $W(k')$-modules
$$\begin{CD}
H_1^{dR}(\G_{\kappa,\gq^*}^{(0)})_\sigma^{\otimes_{W(k')}1-k}\otimes_{W(k')}\bigwedge_{W(k')}^kH_1^{dR}(\G_{\kappa,\gq^*}^{(1)})_\sigma
@>{\cong}>>H_1^{dR}(\G_{\kappa,\gq^*}^{(k)})_\sigma\\
@VVV@VVV\\
\Lie(\G_{\kappa,\gq^*}^{(0)})_\sigma^{\otimes_{W(k')}1-k}\otimes_{W(k')}\bigwedge_{W(k')}^k\Lie(\G_{\kappa,\gq^*}^{(1)})_\sigma
@>{\cong}>>\Lie(\G_{\kappa,\gq^*}^{(k)})_\sigma
\end{CD}.$$
We only have to check that the upper horizontal map respects the various Frobenius operators. We have to think of 
them as the maps which are induced from the various Verschiebungen, according to fact \ref{E15} these are given by:
$$\BT(Ver_{P_\xi^\delta})\times_kk':\Gbar_{\kappa,\gq^*}^\delta\times_{k',\tau}k'\rightarrow\Gbar_{\kappa,\gq^*}^\delta,$$
where the graded displays $P_{\xi,\sigma}^{(k)}={P_{\xi,\sigma}^{(0)}}^{\otimes_{W(k)}1-k}\otimes_{W(k)}\bigwedge_{W(k)}^kP_{\xi,\sigma}^{(1)}$ over the residue field $k$ 
stand for the special fibre of the $P_{x,\sigma}^{(k)}$'s. We can now deduce the desired Frobenius equivariance from lemma \ref{vzehn} and part (ii) of proposition \ref{E5}.\\

We next apply lemma \ref{vdreiz} to the unramified lift $\kappa$. Associated to it we have the point $(Y_\xi^{(k)},\iota^{(k)},\lambda^{(k)},\ebar^{(k)})\in\M^{(k)}(k)$ 
with $M_\sigma^{(k)}:=\d^*(Y_\xi^{(k)}[\gq^\infty])_\sigma$, and a $t_{\xi,cris,\sigma}^{(k)}$ and a $t_{\kappa,dR,\sigma}^{(k)}$. Combining this gives:
\begin{eqnarray*}
&&W(k')\otimes_{W(k)}\d^*(Y_\xi^{(k)}[\gq^\infty])_\sigma\cong\\
&&W(k')\otimes_{W(k)}({M_\sigma^{(0)}}^{\otimes_{W(k)}1-k}\otimes_{W(k)}\bigwedge_{W(k)}^kM_\sigma^{(1)})\cong\\
&&W(k')\otimes_{W(k)}\d^*(\Gbar_{x,\gq}^{(k)})_\sigma.
\end{eqnarray*}
Again it is fairly easy to see that this isomorphism is truely defined over $W(k)$, compare everything with any 
$W(k)$-valued lift of $\xi$ for example, or just use lemma \ref{skel}, hence $Y_\xi^{(k)}[\gq^\infty]\cong\Gbar_{x,\gq}^{(k)}$.\\ 

We obtain immediately an abelian scheme $\Yd_x^{(k)}$ over $\Df_x$, by deforming $Y_\xi^{(k)}$. Clearly it inherits a polarization and a $\O_L$-operation from the corresponding data of $\G_{x,\gq}^{(k)}$ and $Y_\xi^{(k)}$. Finally the level $K^{(k)}$-structure $\ebar^{(k)}$ lifts uniquely from $(Y_\xi^{(k)},\iota^{(k)},\lambda^{(k)})$ 
to $(\Yd_x^{(k)},\iota^{(k)},\lambda^{(k)})$ by the rigidity of \'etale covers. In so doing we manufactured a $\Rd_x$-valued point of $\M^{(k)}$, extending the point 
$g^{(k)}\circ\kappa$. With the usual methods one completes the proof of (i) and (ii), please see \cite[proof of theorem 5.4]{ball} for example.\\

Assertion (iii) can be dealt with by a continuity argument: Using the Baire property one can choose a family of embeddings $\iota_i:K(k')\rightarrow\c$ 
such that $\iota_i\circ\kappa$ converges in the complex topology to $\iota_\infty\circ\eta$ where $\iota_\infty$ is some embedding of $R$. The natural diagram
$$\begin{CD}
H_1^{dR}(\Yd_x^{(0)}[\gq^{*\infty}])_\sigma^{\otimes_{\Rd_x}1-k}
\otimes_{\Rd_x}\bigwedge_{\Rd_x}^kH_1^{dR}(\Yd_x^{(1)}[\gq^{*\infty}])_\sigma
@>{\md_{x,\sigma}^{(k)}}>>H_1^{dR}(\Yd_x^{(k)}[\gq^{*\infty}])_\sigma\\
@V{H_1^{dR}(\lambda^{(0\times1)})}VV@V{H_1^{dR}(\lambda^{(k)})}VV\\
H_{dR}^1(Y_{\iota_i\circ\kappa}^{(0)})_\sigma^{\otimes_\c1-k}
\otimes_\c\bigwedge_\c^kH_{dR}^1(Y_{\iota_i\circ\kappa}^{(1)})_\sigma
@>{t_{\iota_i\circ\kappa,dR,\sigma}^{(k)}}>>
H_{dR}^1(Y_{\iota_i\circ\kappa}^{(k)})_\sigma
\end{CD}$$
commutes by construction of $\Yd_x^{(k)}$. If we pass to the limit we get the result for $\iota_\infty\circ\eta$.
\end{proof}

Throughout the whole work we will retain the notation for the graded displays $P^{(k)}_{x,\sigma}$ over $\Rd_x$ and the map $m_{x,\BT}^{(k)}$, where $x$ is a 
closed point. If $\eta$ is a $R$-valued lift we will write $P_{\eta,\sigma}^{(k)}$ and $m_{\eta,\BT}^{(k)}$ for the objects over $R$, and similarly for the fibre over $\xi$.

\begin{rem}
The attantive reader may have observed that the proofs of parts (i) and (ii) of the previous theorem do not really use the deep proposition \ref{dR}, nor does the proof of the extension theorem in \cite{ball}. However, part (iii) does make use of it, and we need it in the proof of theorem \ref{vacht}, and in all compatibility statements that build on it.
\end{rem}

\subsection{Action of $g^{(k)}$ on Quasi-Isogenies} 
\label{qisg}
Let $\eta:\Spec F\rightarrow M^{(0\times1)}\times_EE_\gp$ be a geometric point, i.e. $F$ separably 
closed. Recall that in case $\ch(F)=0$ we associated to it a specific $\Od$-linear isomorphism from
$$H_1(Y_\eta^{(0)},\zd)^{\otimes_\Od1-k}\otimes_\Od\bigwedge_\Od^kH_1(Y_\eta^{(1)},\zd)$$
to
$$H_1(Y_\eta^{(k)},\zd)$$
which, upon choosing an embedding $F\hookrightarrow\c$, could be described as $m_{\eta,B}^{(k)}\otimes1_\zd$. Observe that $\Od=\prod_\gr\O_{L_\gr}$ 
and let us write $T_\gr Y_\eta^{(k)}$ and $m_{\eta,\gr}^{(k)}$ for the $\gr$-components of $H_1(Y_\eta^{(k)},\zd)$ and $m_{\eta,B}^{(k)}\otimes1_\zd$, 
where $\gr$ is a prime of $\O_L$. We need to extend this to the case where $\ch(F)=p$ is coprime to $\gr$. To this end observe that the group 
$\Aut_{\eta}(F)$ of automorphisms of $F$ that fix the geometric point $\eta:\Spec F\rightarrow\M_\gp^{(0\times1)}$ acts naturally on $T_\gr Y_\eta^{(k)}$. 
Observe also that, if $\R$ is a strictly henselian local subring of $F$ such that $\eta$ factors through $\Spec\R$, then there is a canonical generization map: 
$$\begin{CD}
T_\gr Y_\xi^{(k)}@>{\cong_\R}>>T_\gr Y_\eta^{(k)},\\
\end{CD}$$
where $\xi$ arises from $\Spec\R\rightarrow\M_\gp^{(0\times1)}$ by passage to the closed point, leading to $\Spec k\rightarrow\M_\gp^{(0\times1)}$ 
where $k$ is the residue field of $\R$. Moreover for every $\sigma\in\Aut_\eta(F)$ that stabilizes $\R$ there is a self-explanatory commutative diagram:
$$\begin{CD}
T_\gr Y_\xi^{(k)}@>{\sbar}>>T_\gr Y_\xi^{(k)}\\
@V{\cong_\R}VV@V{\cong_\R}VV\\
T_\gr Y_\eta^{(k)}@>{\sigma}>>T_\gr Y_\eta^{(k)}
\end{CD}$$
which we will refer to as the Galois-invariance of $\cong_\R$, here $\sbar\in\Aut_\xi(k)$ is the automorphism 
of $k$ which is induced by the reduction of $\sigma|_\R$ modulo the maximal ideal of $\R$.\\
If $k$ has characteristic $0$ it is easy to see that the diagram
$$\begin{CD}
{T_\gr Y_\xi^{(0)}}^{\otimes_{\O_{L_\gr}}1-k}\otimes_{\O_{L_\gr}}\bigwedge_{\O_{L_\gr}}^kT_\gr Y_\xi^{(1)}@>{m_{\xi,\gr}^{(k)}}>>T_\gr Y_\xi^{(k)}\\
@V{\cong_\R}VV@V{\cong_\R}VV\\
{T_\gr Y_\eta^{(0)}}^{\otimes_{\O_{L_\gr}}1-k}\otimes_{\O_{L_\gr}}\bigwedge_{\O_{L_\gr}}^kT_\gr Y_\eta^{(1)}@>{m_{\eta,\gr}^{(k)}}>>T_\gr Y_\eta^{(k)}
\end{CD}$$
is commutative, (pick embeddings $F,k\hookrightarrow\c$). Now let $\xi$ be a geometric point in characteristic $p$. Let us simply define $m_{\xi,\gr}^{(k)}$ 
to be the unique upper horizontal map rendering the previous diagram commutative, where $F\supset\R\rightarrow k$ is a choice of a characteristic-zero 
generization. The commutativity of the aforementioned diagram in the equal characteristic-zero case shows $m_{\xi,\gr}^{(k)}$ thus introduced does not depend 
on the choice of $F\supset\R\rightarrow k$. Moreover, the Galois-invariance of $\cong_\R$ implies that $m_{\xi,\gr}^{(k)}$ is $\Aut_\xi(k)$-equivariant too.\\
We next study what happens to quasi-isogenies. Assume that $\xi_1$ and $\xi_2$ are $\f^{ac}$-valued points on $\M_\gp^{(0\times1)}$. Assume that
$$\gamma^{(0\times1)}:Y_{\xi_1}^{(0)}\times_{\f^{ac}}Y_{\xi_1}^{(1)}\rightarrow Y_{\xi_2}^{(0)}\times_{\f^{ac}}Y_{\xi_2}^{(1)}$$ 
is $L\oplus L$-linear and preserves the homogeneous polarizations. This quasi-isogeny induces a graded map
$$\gamma_{cris,\sigma}^{(0\times1)}:P_{\xi_1,\sigma}^{(0)}\oplus P_{\xi_1,\sigma}^{(1)}\rightarrow P_{\xi_2,\sigma}^{(0)}\oplus P_{\xi_2,\sigma}^{(1)}.$$ 
By slight abuse of notation we denote by $\gamma_{cris,\sigma}^{(k)}$ the map which is induced according 
to part (i) of proposition \ref{E5}. Similarily we let $\gamma_{et,\gr}^{(k)}$ be the map that arises from letting
$$\gamma_{et,\gr}^{(0\times1)}:T_\gr Y_{\xi_1}^{(0)}\oplus T_\gr Y_{\xi_1}^{(1)}\rightarrow T_\gr Y_{\xi_2}^{(0)}\oplus T_\gr Y_{\xi_2}^{(1)}$$
act on tensors of the form $x_0^{1-k}x_1\wedge\dots\wedge x_k$.

\begin{thm}
\label{vacht}
Let $\xi_1$ and $\xi_2$ be two points of $\M_\gp^{(0\times1)}$ over the field $\f^{ac}$. Let 
$\gamma^{(0\times1)}$ be a quasi-isogeny as above. Then there exists a unique quasi-isogeny 
$$g^{(k)}(\gamma^{(0\times1)}):Y_{\xi_1}^{(k)}\rightarrow Y_{\xi_2}^{(k)}$$
preserving the homogeneous polarizations, and inducing the map such that the diagram
$$\begin{CD}
\BT(\bigoplus_\sigma{P_{\xi_1,\sigma}^{(0)}}^{\otimes_{W(\f^{ac})}1-k}\otimes_{W(\f^{ac})}\bigwedge_{W(\f^{ac})}^kP_{\xi_1,\sigma}^{(1)})
@>{m_{\xi_1,\BT}^{(k)}}>>Y_{\xi_1}^{(k)}[\gq^{*\infty}]\\
@V{\BT(\gamma_{cris}^{(k)})}VV
@V{g^{(k)}(\gamma^{(0\times1)})[\gq^{*\infty}]}VV\\
\BT(\bigoplus_\sigma{P_{\xi_2,\sigma}^{(0)}}^{\otimes_{W(\f^{ac})}1-k}\otimes_{W(\f^{ac})}\bigwedge_{W(\f^{ac})}^kP_{\xi_2,\sigma}^{(1)})
@>{m_{\xi_2,\BT}^{(k)}}>>Y_{\xi_2}^{(k)}[\gq^{*\infty}]
\end{CD}$$
commutes, moreover 
$$m_{\xi_2,\gr}^{(k)}\circ\gamma_{et,\gr}^{(k)}\circ m_{\xi_1,\gr}^{(k)^{-1}}$$ 
agrees with the map from $T_\gr Y_{\xi_1}^{(k)}$ to $T_\gr Y_{\xi_2}^{(k)}$ that is induced by $g^{(k)}(\gamma^{(0\times1)})[\gr^\infty]$.
\end{thm}

\begin{proof}
Let $\eta_1,\eta_2:\Spec R\rightarrow\M_\gp^{(0\times1)}$ be lifts of $\xi_1$ and $\xi_2$ over a mixed characteristic complete discrete valuation ring with uniformizer 
$\varpi$ and residue field $\f^{ac}$. If we apply lemma \ref{vzehn} to the groups $\G_i^{(k)}=\BT(\bigoplus_\sigma P_{\eta_i,\sigma}^{(k)})$, we get the diagram:
$$\begin{CD}
\q\otimes H_1^{dR}(\G_1^{(0)})_\sigma^{\otimes_R1-k}\otimes_R\bigwedge_R^kH_1^{dR}(\G_1^{(1)})_\sigma
@>{\md_{1,\sigma}^{(k)}}>>\q\otimes H_1^{dR}(\G_1^{(k)})_\sigma\\
@V{\tilde\gamma_{cris,\sigma}^{(k)}}VV
@V{\widetilde{\gamma_{cris,\sigma}^{(k)}}}VV\\
\q\otimes H_1^{dR}(\G_2^{(0)})_\sigma^{\otimes_R1-k}\otimes_R\bigwedge_R^kH_1^{dR}(\G_2^{(1)})_\sigma
@>{\md_{2,\sigma}^{(k)}}>>\q\otimes H_1^{dR}(\G_2^{(k)})_\sigma
\end{CD}$$
here the symbol $\widetilde{\gamma_{cris,\sigma}^{(k)}}$ indicates the Grothendieck-Messing functoriality of the map $\BT(\gamma_{cris}^{(k)})$ and the 
symbol $\tilde\gamma_{cris,\sigma}^{(k)}$ means the effect of$\widetilde{\gamma_{cris,\sigma}^{(1)}}$ and $\widetilde{\gamma_{cris,\sigma}^{(0)}}$ 
on tensors of the kind $x_0^{1-k}x_1\wedge\dots\wedge x_k$. If we compose the rows of the diagram with the $H_1^{dR}$ of the maps 
$m_{\eta_i,\BT}^{(k)}:\G_i^{(k)}\rightarrow Y_{\eta_i}^{(k)}[\gq^{*\infty}]$, and use part (iii) of theorem \ref{extend}, we get the diagram:
$$\begin{CD}
{H_1^{dR}(Y_{\eta_1}^{(0)})_\sigma^{\otimes_R1-k}\otimes_R\bigwedge_R^kH_1^{dR}(Y_{\eta_1}^{(1)})_\sigma}
@>{m_{\eta_1,dR}^{(k)}}>>{H_1^{dR}(Y_{\eta_1}^{(k)})_\sigma}\\
@V{\tilde\gamma_{cris,\sigma}^{(k)}}VV
@V{\widetilde{\gamma_{cris,\sigma}^{(k)}}}VV\\
{H_1^{dR}(Y_{\eta_2}^{(0)})_\sigma^{\otimes_R1-k}\otimes_R\bigwedge_R^kH_1^{dR}(Y_{\eta_2}^{(1)})_\sigma}
@>{m_{\eta_2,dR}^{(k)}}>>{H_1(Y_{\eta_2}^{(k)})_\sigma}.
\end{CD}$$
All we have to do is check that there exists a quasi-isogeny $Y_{\xi_1}^{(k)}\rightarrow Y_{\xi_2}^{(k)}$ that induces the vertical right map of this diagram, and 
``commutes'' with the $m_{\xi_i,\gr}^{(k)}$'s. In case $\gamma^{(0\times1)}$ lifts to a quasi-isogeny between the $Y_{\eta_i}^{(0)}\times_RY_{\eta_i}^{(1)}$'s 
over $R$ this is evident, as one can work over the extension $\c\supset R$ to come up with a quasi-isogeny $Y_{\eta_1}^{(k)}\rightarrow Y_{\eta_2}^{(k)}$ 
``commuting'' with both $m_{\eta_i,dR}^{(k)}$ and $m_{\eta_i,\gr}^{(k)}$. In case $\gamma^{(0\times1)}$ is a product of such isogenies, this is evident too.\\

In view of the lemma \ref{verynasty} below we can write every quasi-isogeny as a product of (possibly inverses of) isogenies with 
$\deg(\gamma^{(0\times1)})=p^{(n+1)[L^+:\q]}\z_{(p)}^\times$, i.e. satisfying:
$${\gamma^{(0\times1)}}^t\circ\lambda^{'(0\times1)}\circ\gamma^{(0\times1)}=p\lambda^{(0\times1)},$$
so that we can rephrase the problem as follows: Suppose $K^{(1)}$ is a subgroup scheme over $\f^{ac}$ such that
$$\begin{CD}
Y_\xi^{(1)}[\gq^{*\infty}]@>>>Y_\xi^{(1)}[\gq^{*\infty}]/K^{(1)}\\Y_\xi^{(1)}[\gq^{*\infty}]/K^{(1)}@>p>>Y_\xi^{(1)}[\gq^{*\infty}]
\end{CD}$$
are $\O_{L_{\gq^*}}$-linear isogenies with degrees $p^{rk}$, and $p^{r(n-k)}$ -and similarily for $Y_\xi^{(0)}[\gq^{*\infty}]$- can one find lifts to $p$-divisible groups 
$Y_\eta^{(1)}[\gq^{*\infty}]$, $Y_\eta^{(0)}[\gq^{*\infty}]$ over a mixed characteristic discrete valuation ring $R$ together with lifts of subgroup schemes $\K^{(1)}$ 
and $\K^{(0)}$ with $\O_{L_{\gq^*}}$-operation? It is quite obvious how to lift $Y_\xi^{(0)}[\gq^{*\infty}]$, and $K^{(0)}$, in order to lift $Y_\xi^{(1)}[\gq^{*\infty}]$, and 
$K^{(1)}$ we use the machinery of the local models of ~\cite{rapoport}. In the case at hand $M^{loc}$ decomposes into a product of $M_\sigma^{loc}$ parameterized 
by the elements $\sigma\in\Omega$. The functor which the factor $M_\sigma^{loc}$ represents, can be described as follows: Write $\O_{E_\gp}^n=A\oplus B$ with 
$\rk_{\O_{E_\gp}}A=k$, $\rk_{\O_{E_\gp}}B=n-k$. Then over $S/\Spec\O_{E_\gp}$ the points consist of quadruples $(t_A,\phi_A,t_B,\phi_B)$ where $t_A$ and $t_B$ 
are $\O_S$-modules, and $\phi_A$ and $\phi_B$ are $\O_S$-linear surjective maps from $\O_{E_\gp}^n\otimes_{\O_{E_\gp}}\O_S$ to $t_A$ and $t_B$ such that:
\begin{itemize}
\item
$t_A$ and $t_B$ are projective of rank $n-1$
\item
the kernel of $\phi_A$ is mapped into the kernel of $\phi_B$ under the map on $\O_{E_\gp}^n$ defined by the $A\oplus B\ni a+b\mapsto a+pb$
\item
the kernel of $\phi_B$ is mapped into the kernel of $\phi_A$ under the map on $\O_{E_\gp}^n$ defined by $a+b\mapsto pa+b$
\end{itemize}
Observe that $M_\sigma^{loc}$ is naturally a subscheme of $\p_{\O_{E_\gp}}^{n-1}\times\p_{\O_{E_\gp}}^{n-1}$, let us write $(t_1,\dots,t_n)$ for 
the coordinates of the first factor and $(s_1,\dots,s_n)$ for the coordinates of the second factor. This should mean that $t_A$ ($t_B$) is the quotient of 
$\O_{E_\gp}^n\otimes_{\O_{E_\gp}}\O_S$ by the $\O_S$ module generated by the element $(t_1,\dots,t_n)$ (resp. $(s_1,\dots,s_n)$). It is easy to see that the 
open subset of $M_\sigma^{loc}$ where one of the induced maps $\ker(\phi_A)\rightarrow\ker(\phi_B)$ or $\ker(\phi_B)\rightarrow\ker(\phi_A)$ is an isomorphism, 
is actually smooth. The rest is covered by the Zariski open sets $\{s_\nu t_\mu\neq0\}$ where $\nu\leq k<\mu$. On this set $M_\sigma^{loc}$ has the equations:
\begin{eqnarray*}
&&t_1=\frac{s_1}{s_\nu}t_\nu,\dots,t_k=\frac{s_k}{s_\nu}t_\nu\\
&&s_{k+1}=\frac{t_{k+1}}{t_\mu}s_\mu,\dots,s_n=\frac{t_n}{t_\mu}s_\mu\\
&&\frac{s_\mu}{s_\nu}\frac{t_\nu}{t_\mu}=p
\end{eqnarray*}
This is because the above equations are equivalent to:
$$(t_1,\dots,t_k,pt_{k+1},\dots,pt_n)=\frac{t_\nu}{s_\nu}(s_1,\dots,s_n)$$
and
$$(ps_1,\dots,ps_k,s_{k+1},\dots,s_n)=\frac{s_\mu}{t_\mu}(t_1,\dots,t_n).$$
So this affine chart is the spectrum of the factor ring of the polynomial $\O_{E_\gp}$-algebra in the $n+1$ variables 
$$\frac{s_1}{s_\nu},\dots,\frac{s_{\nu-1}}{s_\nu},\frac{s_\mu}{s_\nu},\frac{s_{\nu+1}}{s_\nu},\dots,\frac{s_k}{s_\nu},
\frac{t_{k+1}}{t_\mu},\dots,\frac{t_{\mu-1}}{t_\mu},\frac{t_\nu}{t_\mu},\frac{t_{\mu+1}}{t_\mu},\dots,\frac{t_n}{t_\mu}$$
by the ideal generated by the single element $\frac{s_\mu}{s_\nu}\frac{t_\nu}{t_\mu}-p$, and is therefore flat over $\O_{E_\gp}$. Therefore the requested lifts do exist.
\end{proof}

\begin{lem}
\label{verynasty}
Let $k$ be a perfect field of characteristic $p$, and let $M=\bigoplus_{\sigma\in\z/r\z}M_\sigma$ be a $\z/r\z$-graded Dieudonn\'e module over $k$, such that 
\begin{equation*}
\dim_kM_{\sigma+1}/\phi(M_\sigma)\leq1
\end{equation*}
holds. Consider a $\z/r\z$-graded Dieudonn\'e submodule $N_\sigma\subset M_\sigma$ such that the lengths of the $W(k)$-modules 
$M_\sigma/N_\sigma$ are finite and independent of $\sigma$. Then at least one of the following two assertions holds:
\begin{itemize}
\item[(i)]
$pM_\sigma\subset N_\sigma$ for all $\sigma$
\item[(ii)]
$N_\sigma$ is contained in at least one other $\z/r\z$-graded Dieudonn\'e submodule $K_\sigma\subset M_\sigma$, 
such that the lengths of $K_\sigma/N_\sigma$ and $M_\sigma/K_\sigma$ are non-zero and independent of $\sigma$.
\end{itemize}
\end{lem}
\begin{proof}
Notice that $\lg_{W(k)}(M_\sigma/N_\sigma)=\lg_{W(k)}(M_{\sigma+1}/N_{\sigma+1})<\infty$ implies 
$\lg_{W(k)}(N_{\sigma+1}/\phi(N_\sigma))=\lg_{W(k)}(M_{\sigma+1}/\phi(M_\sigma))\leq1$.\\
Suppose that (i) does not hold, consider the $\z/r\z$-graded Dieudonn\'e submodule $\Nbar_\sigma=N_\sigma+pM_\sigma$, observe that the 
numbers $\lg_{W(k)}(M_\sigma/\Nbar_\sigma)=h_\sigma$ may depend on $\sigma$, but let us assume without loss of generality that they attain 
their minimum for $\sigma=0$. In view of the note above it is enough to look for $\z/r\z$-graded subcrystals $K_\sigma\supset\Nbar_\sigma$: Starting 
from $\Nbar_r=\Nbar_0=K_0$ we proceed by downward induction. Since $\lg_{W(k)}(M_\sigma/(M_\sigma\cap\phi^{-1}(K_{\sigma+1}))\leq h_0$ 
while $\lg_{W(k)}(M_\sigma/\Nbar_\sigma)\geq h_0$ we are allowed to choose some $W(k)$-submodule $K_\sigma$ sandwiched between 
$\Nbar_\sigma$ and $M_\sigma\cap\phi^{-1}(K_{\sigma+1})$ and satisfying $\lg_{W(k)}(M_\sigma/K_\sigma)=h_0$.
\end{proof}

\section{Endomorphisms}
\label{tat}
In this chapter we study the relationship between the endomorphism algebra of $Y_\xi^{(1)}$ and the endomorphism algebra of $Y_\xi^{(k)}$. In the special case of a 
$\mu$-ordinary point $\xi$ this subject was touched upon in \cite[section 6]{ball}, but here we want to be more general. However we continue to assume that $\xi$ is a 
point with values in the algebraic closure $\f^{ac}$ of $\f_{p^r}$. Our aim is to construct a canonical $*$-preserving map of $L$-algebras:

\begin{equation}
\label{wdrei}
op_\xi^{(k)}:\sy_L^k(\End_L^0(Y_\xi^{(1)}))\rightarrow\End_L^0(Y_\xi^{(k)}),
\end{equation}

where the left hand side acquires an involutive $L$-algebra structure through the natural embedding into $\End_L^0(Y_\xi^{(1)})^{\otimes_Lk}$. 
Our tools are theorem \ref{vacht} and lifts into characteristic $0$. Observe that passing to the $\gr$-adic or crystalline homology theories gives maps:
\begin{eqnarray}
\label{newfuenf}
&&\sy_{L_\gr}^k(\End_{L_\gr}^0(T_{\xi,\gr}^{(1)}))\rightarrow\End_{L_\gr}^0(T_{\xi,\gr}^{(k)})\\
\label{weshalb}
&&\sy_{K(\f_{p^r})}^k(\End_{K(\f^{ac})[F^r]}^0(P_{\xi,0}^{(1)}))\rightarrow\End_{K(\f^{ac})[F^r]}^0(P_{\xi,0}^{(k)})
\end{eqnarray}
here we write $T_{\xi,\gr}^{(k)}={T_\gr Y_\xi^{(0)}}^{\otimes_{\O_{L_\gr}}1-k}\otimes_{\O_{L_\gr}}\bigwedge_{\O_{L_\gr}}^kT_\gr Y_\xi^{(1)}$ and use the convention that a tensor
$$f=\sum_\nu f_{\nu,1}\otimes_Q\dots\otimes_Qf_{\nu,k}\in\End_Q(V)^{\otimes_Qk},$$
where $V$ is some $Q$-vector space, acts on $V\otimes_Q\dots\otimes_QV$ by mapping an 
element of the form $x_1\otimes_Q\dots\otimes_Qx_k$ to the element determined by the formula
$$\sum_\nu f_{\nu,1}(x_1)\otimes_Q\dots\otimes_Qf_{\nu,k}(x_k)\in V\otimes_Q\dots\otimes_QV.$$
If the tensor $f\in\End_Q(V)^{\otimes_Qk}$ lies in $\sy_Q^k(\End_Q(V))$ i.e. if it satisfies
$$f=\sum_\nu f_{\nu,\pi(1)}\otimes_Q\dots\otimes_Qf_{\nu,\pi(k)}$$
for all permutations $\pi\in S_k$, then it is easy to see it preserves all the $S_k$-eigenspaces of $V\otimes_Q\dots\otimes_QV$, and in particular we get a map
$$\sy_Q^k(\End_Q(V))\rightarrow\End(\bigwedge_Q^kV).$$
Using these conventions we can state:

\begin{prop}
\label{vsieben}
Let $\xi$ be as above. Then the algebra maps
\begin{equation}
\label{wfuenf}
op_\xi^{(k)}[\gr^\infty]:\sy_{L_\gr}^k(\End_{L_\gr}^0(Y_\xi^{(1)}[\gr^\infty]))\rightarrow\End_{L_\gr}^0(Y_\xi^{(k)}[\gr^\infty]),
\end{equation}
and
\begin{equation}
\label{wvier}
op_\xi^{(k)}[\gq^{*\infty}]:\sy_{L_{\gq^*}}^k(\End_{L_{\gq^*}}^0(Y_\xi^{(1)}[\gq^{*\infty}]))\rightarrow\End_{L_{\gq^*}}^0(Y_\xi^{(k)}[\gq^{*\infty}]),
\end{equation}
which are obtained from \eqref{newfuenf} and \eqref{weshalb}, send the subalgebra $\sy_L^k\End_L^0(Y_\xi^{(1)})$ 
into $\End_L^0(Y_\xi^{(k)})$, moreover the mapping so induced does not depend on $\gr$ or $\gq^*$.
\end{prop}
\begin{proof}
Let us start with the observation that the restrictions of $op_\xi^{(k)}[\gr^\infty]$ (resp. $op_\xi^{(k)}[\gq^{*\infty}]$) 
to the algebras $\End_L^0(Y_\xi^{(1)})\otimes_LQ$ define maps:
$$\sy_L^k(\End_L^0(Y_\xi^{(1)}))\otimes_LQ\rightarrow\End_L^0(Y_\xi^{(k)})\otimes_LQ,$$
for $Q\in\{L_\gr,L_{\gq^*}\}$. One can see this as follows: Pick some large finite field $k\subset\f^{ac}$ over which $\xi$ is defined, and to which all endomorphisms of 
all $Y_\xi^{(k)}$'s descend. Over this field $T_{\xi,\gr}^{(k)}$ acquires a $\Gal(\f^{ac}/k)$-operation (and the $P_{\xi,\sigma}^{(k)}$'s over $\f^{ac}$ descend to displays 
over $k$), moreover the maps $m_{\xi,\gr}^{(k)}:T_{\xi,\gr}^{(k)}\rightarrow T_\gr Y_\xi^{(k)}$ are preserved by the $\Gal(\f^{ac}/k)$-operation (resp. $m_{\xi,\BT}^{(k)}$ 
descends to $W(k)$). By Tate's theorem (cf. \cite{tate}) we deduce that $op_\xi^{(k)}[\gr^\infty]$ sends elements in $\sy_L^k(\End_L^0(Y_\xi^{(1)}))$, 
being $\Gal(\f^{ac}/k)$-invariant, at least into $\End_L^0(Y_\xi^{(k)})\otimes_LL_\gr$, and the same holds for $op_\xi^{(k)}[\gq^{*\infty}]$.\\

In order to complete the proof of the proposition it suffices to consider the elements $f\otimes_L\dots\otimes_Lf$, this is due to
$$\sum_{\pi\in S_k} f_{\pi(1)}\otimes_L\dots\otimes_Lf_{\pi(k)}=\sum_{Z\subset\{1,\dots,k\}}(-1)^{|Z|-k}(\sum_{\nu\in Z}f_\nu)\otimes_L\dots\otimes_L(\sum_{\nu\in Z}f_\nu).$$
Let us write $q:\End_L^0(Y_\xi^{(1)})\otimes_LQ\rightarrow\End_L^0(Y_\xi^{(k)})\otimes_LQ$ for the composition of $f\mapsto f\otimes_L\dots\otimes_Lf$ 
with one of $op_\xi^{(k)}[\gr^\infty]$ or $op_\xi^{(k)}[\gq^{*\infty}]$. This is a polynomial map between two affine $Q$-spaces both of which have 
a $L$-form. We claim that the $L$-rationality of $q$ can be checked on any Zariski-dense subset $S$ in $\End_L^0(Y_\xi^{(1)})$, regarded 
as affine space over $L$: Indeed if $q(S)\subset\End_L^0(Y_\xi^{(k)})$, then $\sigma(q)|_S=q|_S$ for all $L$-linear embeddings 
$\sigma:Q\rightarrow Q^{ac}$, hence $\sigma(q)=q$ for all $\sigma$ and $q$ has all its coefficients in $L$.\\

Finally observe that theorem \ref{vacht} implies that $q(f)\in\End_L^0(Y_\xi^{(k)})$, whenever $f\in S=\{f\in\End_L^0(Y_\xi^{(1)})^\times|ff^*=1\}$. However, $S$ is Zariski dense 
in $\End_L^0(Y_\xi^{(1)})$: One verifies this by checking it is in fact the group of $L^+$-rational points of some connected reductive algebraic group, whose $L$-points are 
$\End_L^0(Y_\xi^{(1)})^\times$. Accordingly, the density follows from \cite[corollary 18.3]{borel}, as $\End_L^0(Y_\xi^{(1)})^\times$ is clearly dense in $\End_L^0(Y_\xi^{(1)})$.\\

The independence of $\gr$ and $\gq^*$ follows analogously.
\end{proof}

\subsection{Isoclinal Points}
\label{supo}

We need to set up a further Hodge structure $V$ of type $\{(-1,0),(0,-1)\}$ with $\O_L$-operation. Pick some purely imaginary element $v\in L^\times$. 
Let the underlying torsion-free $\O_L$-module of $V$ be $\O_L$ itself, and let the Hodge structure be given by the $\r$-linear homomorphism:
\begin{equation}
\label{newvier}
h:\c\rightarrow\End_L(L_\r);i\mapsto\frac{v}{\sqrt{-v^2}}.
\end{equation}
Let us write $X=V_\r/\O_L$ for the corresponding CM type abelian variety and let us write $\Phi(x)=\tr(x|_{V^{-1,0}})$ for the corresponding CM trace. Let us assume that
\begin{itemize}
\item[(T.2)]
$\Card(\{\sigma|\tau^\sigma\circ\iota\in|\Phi|\})=a:=\frac{w}{n}$,
\end{itemize}
where $w$ is the integer $\Card(|\Phi^{(n)}|-|\Phi^{(0)}|)=\Card(\Omega)$. A Hodge structure $(V,h)$ with this property clearly 
exists if and only if $w$ is divisible by $n$. Observe that the non-degenerate $*$-skew-Hermitian pairing which is given by 
\begin{equation}
\label{poleins}
\psi:V\times V\rightarrow\q;(x,y)\mapsto\tr_{L/\q}(vxy^*),
\end{equation}
endows $X$ with a homogeneous polarization $\q_{>0}\lambda$ of which the Rosati-involution stabilizes $L\subset\End^0(X)$, this is because $\psi(x,h(i)y)$ 
is positive definite on $V_\r$. Accordingly, the Weil pairing $TX\times TX\rightarrow2i\pi\q$ is a scalar multiple of $\psi$, as is the $\gr^+$-adic one on 
$T_{\gr^+}\Xbar=\O_{L_{\gr^+}^+}\otimes_{\O_{L^+}}TX$, where $\Xbar/\f^{ac}$ is the special fibre of a good model of $X$ over $\O_{E_\gp^{ac}}$. We say that a point 
$\xi$ is isoclinal if $Y_\xi^{(1)}$ is isogenous to $\Xbar^{\times n}$, as an abelian variety with $L$-operation up to isogeny. This is actually equivalent to saying that 
the $p$-divisible group $Y_\xi^{(1)}[\gq^\infty]$ is isoclinal, as follows for example by specializing the result \cite[Corollary 6.29]{rapoport} to our concrete situation.\\

Throughout the rest of this chapter we fix such a $\xi$, and we also bear in mind that the $L$-space $B=\Hom_L^0(\Xbar,Y_\xi^{(1)})$ 
is equipped with a positive definite sesquilinear pairing: If an element $g$ of $B$ is given one can consider its dual 
$g^*=\lambda^{-1}\circ g^t\circ\lambda^{(1)}\in\Hom_L^0(Y_\xi^{(1)},\Xbar)=B^*$, where $B^*=\Hom_L(B,L)$ is the dual $L$-vector space. Now set:
\begin{equation}
\label{ses}
(f,g)=g^*\circ f\in\Hom_L^0(\Xbar,\Xbar)=L,
\end{equation}
for $f,g\in B$. It is easy to see that the crystalline and $\gr$-adic homology theories can be recovered from $B$ by:
\begin{eqnarray}
\label{E0}
&&B\otimes_{\O_L}T_\gr\Xbar\cong\q\otimes T_\gr Y_\xi^{(1)}\\
\label{E3}
&&B\otimes_{\O_L}P_\sigma\cong\q\otimes P_{\xi,\sigma}^{(1)},
\end{eqnarray}
where $P_\sigma$ is the graded display of $\Xbar[\gq^{*\infty}]$. In fact this means that $P_{\xi,\sigma}^{(1)}$ has skeleton isomorphic to: 
\begin{equation}
\label{wnull}
I^{(1)}\cong B\otimes_{L,\iota\circ*}H,
\end{equation}
as $H=\{x\in\q\otimes P_0|\phi^r(x)=p^ax\}$ is the skeleton of $P_\sigma$. If we pass to the exterior powers we get from \eqref{E0} the isomorphism:
\begin{equation}
\bigwedge_L^kB\otimes_{\O_L}({T_{\xi,\gr}^{(0)}}^{\otimes_{\O_{L_\gr}}1-k}\otimes_{\O_{L_\gr}}T_\gr\Xbar^{\otimes_{\O_{L_\gr}}k})\cong\q\otimes T_{\xi,\gr}^{(k)},
\end{equation}
giving rise to the pleasant formula:
\begin{equation}
\label{wzwei}
\End_{L_\gr}^0(T_{\xi,\gr}^{(k)})\cong\End_L(\bigwedge_L^kB)\otimes_LL_\gr.
\end{equation}
And analogously \eqref{wnull} implies:
\begin{equation}
\label{weins}
\End_{K(\f_{p^r})}^0(I^{(k)})\cong\End_L(\bigwedge_L^kB)\otimes_{L,\iota\circ*}K(\f_{p^r}),
\end{equation}
where $I^{(k)}$ is the skeleton of $P_{\xi,\sigma}^{(k)}$. This consideration sets up ``natural'' $L$-forms of the above algebras. It is 
trivial to see that in the $k\leq1$-case this $L$-form is nothing else than $\End_L(Y_\xi^{(k)})$. That pertains to the $k\geq2$-case:

\begin{lem}
\label{wsechs}
The images of $\End_L^0(Y_\xi^{(k)})$ in $\End_{K(\f_{p^r})}^0(I^{(k)})$ and in $\End_{L_\gr}^0(T_{\xi,\gr}^{(k)})$ are equal to $\End_L(\bigwedge_L^kB)$ embedded 
therein by means of the isomorphisms \eqref{weins} and \eqref{wzwei}. This sets up one and only one isomorphism $\End_L(\bigwedge_L^kB)\cong\End_L^0(Y_\xi^{(k)})$.
\end{lem}
\begin{proof} 
The maps $op_\xi^{(k)}[\gr^\infty]$ and $op_\xi^{(k)}[\gq^{*\infty}]$ that were introduced in \eqref{wfuenf} and \eqref{wvier} preserve 
certainly the $L$-structures which the isomorphism \eqref{wzwei} and \eqref{weins} exhibit on their source and target. Appealing to proposition 
\ref{vsieben} we infer that on the one hand the image of $op_\xi^{(k)}$ is contained in $\End_L(\bigwedge_L^kB)$. On the other hand that image 
is a $L$-form of say $\End_{K(\f_{p^r})}^0(I^{(k)})$, as its dimension agrees with the dimension of the image of $op_\xi^{(k)}[\gq^{*\infty}]$. We get 
$$\End_L^0(Y_\xi^{(k)})=op_\xi^{(k)}\sy_L^k(\End_L^0(Y_\xi^{(1)}))=\End_L(\bigwedge_L^kB)$$ 
for free.
\end{proof}

\subsection{Lifts in characteristic $p$}
Let $R$ be a complete discrete valuation ring with uniformizer $\omega$, and perfect residue field. Lemma \ref{wsechs} makes the following definition possible:

\begin{defn}
\label{newdrei}
Let $G^+$ be an algebraic subgroup of $\U(B/L)$. A lift $\eta$ over $R$ of $\xi$ is said to have a mock $G^+$-structure 
if the (injective!) reduction-$\pmod\omega$-map identifies $\End_L^0(Y_\eta^{(k)})$ with the subalgebra of $\End_L^0(Y_\xi^{(k)})$, 
that corresponds to $\End_G(\bigwedge_L^kB)\subset\End_L(\bigwedge_L^kB)$, for all $k\leq\min\{n,p-1\}$.
\end{defn}

From now on we will always assume that our lifts are $R=\f^{ac}[[t]]$-valued ones. 

\begin{rem}
\label{E9}
Over the perfect field $\f^{ac}$ the notion of a display is just another synonym for ``unipotent Dieudonn\'e module'' or ``$W(\f^{ac})$-window''. More generally, Zink 
constructs an equivalence between the category of $A$-windows and the category of displays over $R=A/pA$, for an arbitrary frame $A$, cf. \cite[Theorem1.6]{zink1}. 
Let us explain the dictionary in the context we need, namely the $\z/r\z$-graded one over $A=W(\f^{ac})[[t]]$. Starting with a graded $W(\f^{ac})[[t]]$-window 
$(\tilde M_\sigma,\tilde M_{\sigma,1},\tilde\phi)$ with graded normal decomposition $\tilde M_\sigma=\tilde L_\sigma\oplus\tilde T_\sigma$, extension of scalars yields modules 
\begin{eqnarray}
\label{DSP1}
&&P_\sigma=W(\f^{ac}[[t]])\otimes_{\varkappa,W(\f^{ac})[[t]]}\tilde M_\sigma\\
\label{DSP2}
&&Q_\sigma=W(\f^{ac}[[t]])\otimes_{\varkappa,W(\f^{ac})[[t]]}\tilde L_\sigma\oplus I_{\f^{ac}[[t]]}\otimes_{\varkappa,W(\f^{ac})[[t]]}\tilde T_\sigma,
\end{eqnarray}
where $\varkappa$ was explained in \eqref{car}. To define $^F$-linear operators $F$ and $V^{-1}$ on them one proceeds as:
\begin{eqnarray*}
&&F(\alpha\otimes_{\varkappa,W(\f^{ac})[[t]]}x)={^F\alpha\otimes_{\varkappa,W(\f^{ac})[[t]]}}\tilde\phi(x)\\
&&V^{-1}(\alpha\otimes_{\varkappa,W(\f^{ac})[[t]]}l)={^F\alpha\otimes_{\varkappa,W(\f^{ac})[[t]]}}\tilde\phi_1(l)\\
&&V^{-1}(^V\alpha\otimes_{\varkappa,W(\f^{ac})[[t]]}t)=\alpha\otimes_{\varkappa,W(\f^{ac})[[t]]}\tilde\phi(t)
\end{eqnarray*}
This is an unambiguous definition as worked out in \cite[section 1]{zink1}, and the result is a graded display $(P_\sigma,Q_\sigma,F,V^{-1})$ over $\f^{ac}[[t]]$.
\end{rem}

If $\eta$ is understood, then we write $\tilde M_\sigma^{(k)}$ for the $\z/r\z$-graded $W(\f^{ac})[[t]]$-window that corresponds to the display $P_{\eta,\sigma}^{(k)}$ 
by the above, and we write $M_\sigma^{(k)}=P_{\xi,\sigma}^{(k)}$ for the special fibre. It is straightforward to see that we have a canonical isomorphism:
\begin{equation}
\label{WW}
\tilde M_\sigma^{(0)^{\otimes_{W(\f^{ac})[[t]]}1-k}}\otimes_{W(\f^{ac})[[t]]}\bigwedge_{W(\f^{ac})[[t]]}^k\tilde M_\sigma^{(1)}\cong\tilde M_\sigma^{(k)}.
\end{equation}
In particular the isomorphism $m_{\eta,\BT}^{(k)}$ in part (ii) of theorem \ref{extend} translates into an isomorphism 
between $\tilde M_\sigma^{(k)}$ and the window that corresponds to the $p$-divisible formal group $Y_\eta^{(k)}[\gq^{*\infty}]$ 
by \cite[Theorem 4]{zink1}.\\Let us also sum up gadgets from subsection \ref{vdrei}: We trivialize $\tilde M_\sigma^{(k)}$ 
over the ring extension $K(\f^{ac})\{\{t\}\}$, by using the skeleton $I^{(k)}$ we have:
\begin{equation*}
\tilde M_0^{(k)}\otimes_{W(\f^{ac})[[t]]}K(\f^{ac})\{\{t\}\}\cong I^{(k)}\otimes_{K(\f_{p^r})}K(\f^{ac})\{\{t\}\},
\end{equation*}
with corresponding 
$$\theta^{(k)}\in(\End_{K(\f_{p^r})}(I^{(k)})\otimes_{K(\f_{p^r})}\Q(\f^{ac}))^\times.$$
as constructed in \ref{vdrei}. According to \eqref{weins} we may think of these elements as sitting in:
$$(\End_L(\bigwedge_L^kB)\otimes_{L,\iota\circ*}\Q(\f^{ac}))^\times,$$
indeed we have $\theta^{(k)}(x_1\wedge\dots\wedge x_k)=\theta^{(1)}(x_1)\wedge\dots\wedge\theta^{(1)}(x_k)$.

\begin{lem}
\label{vsechs}
Let $\eta$ be a lift of $\xi$, and let $\theta^{(1)}\in(\End_L(B)\otimes_{L,\iota\circ*}\Q(\f^{ac}))^\times$ be the corresponding Dwork descent 
datum. Let $G^+$ be the smallest $L^+$-subgroup of $\U(B/L)$ containing $\theta^{(1)}$. Then $\eta$ has a mock $G^+$-structure.
\end{lem}
\begin{proof}
Observe firstly, that the image of the map 
$$\End_L^0(Y_\xi^{(k)})\rightarrow\End_{K(\f_{p^r})}(I^{(k)})$$
(coming from $m_{\xi,\BT}^{(k)}$) is $\End_L(\bigwedge_L^kB)$. Moreover, the preimage of some $f\in\End_L(\bigwedge_L^kB)$ induces a 
quasi-endomorphism on the (isogeny class of the) $p$-divisible group $Y_\eta^{(k)}[\gq^{*\infty}]$ if and only if it preserves the $W(\f^{ac})[[t]][\frac{1}{p}]$-lattice 
$\q\otimes\tilde M_0^{(k)}$ within $I^{(k)}\otimes_{K(\f_{p^r})}K(\f^{ac})\{\{t\}\}$. But by lemma \ref{ff} one can apply descent theory to recover it as the module:
\begin{eqnarray*}
&&\q\otimes\tilde M_0^{(k)}=\{x\in I^{(k)}\otimes_{K(\f_{p^r})}K(\f^{ac})\{\{t\}\}|\\
&&\theta^{(k)}(x\otimes_{K(\f^{ac})\{\{t\}\},q_2}1_{\Q(\f^{ac})})=x\otimes_{K(\f^{ac})\{\{t\}\},q_1}1_{\Q(\f^{ac})}\},
\end{eqnarray*}
where the maps $q_1,q_2:K(\f^{ac})\{\{t\}\}\rightarrow\Q(\f^{ac})$ are as in \eqref{projone} and \eqref{projtwo}. 
Hence $\q\otimes\tilde M_0^{(k)}$ is preserved if and only if $f$ commutes with $\theta^{(k)}$.\\

Observe secondly that $f^*$ induces a quasi-endomorphism on the (isogeny class of the) $p$-divisible group $Y_\eta^{(k)}[\gq^{*\infty}]$ if and only 
if actually $f$ induces an endomorphism on the (isogeny class of the) dual $p$-divisible group $Y_\eta^{(k)}[\gq^\infty]$. By the Serre-Tate theorem 
we find that the preimage of $f$ lies in the subalgebra $\End_L^0(Y_\eta^{(k)})$ if and only if both $f$ and $f^*$ commute with $\theta^{(k)}$.\\

By the criterion of lemma \ref{R1}, we have what we wanted.
\end{proof}

\section{Proof of Main Theorem}
\label{proof}
\subsection{Choice of $\sy$-Structure} 
\label{symG}
Consider our non-zero imaginary element $v$ in our CM field $L$, and our primes $\gq,\gq^*,\gq^+$ and $p\neq2$ of $L$, $L^+$, and $\q$ as in definition 
\ref{vzwei}. Write $E$ for the smallest subfield of $\c$ that contains all Galois conjugates of $L$, and fix a prime $\gp$ of $E$ above $p$ with residue field, say 
$\f_{p^r}$. Fix an isomorphism $\iota:L_\gq\rightarrow K(\f_{p^r})$, and write $a$ for whichever is the smaller of $\Card(\{\sigma|\Im(\tau^\sigma\circ\iota(v))>0\})$ 
and $\Card(\{\sigma|\Im(\tau^\sigma\circ\iota(v))<0\})$. Let $(X,\lambda)$ be a polarized abelian variety with complex multiplication by $\O_L$, such that 
the CM type $|\Phi|$, satisfies (T.2), and such that the Weil-pairing is equal to \eqref{poleins} up to rational scalars.\\
Now let the data consisting of our $n$-dimensional $L$-vector space $B$, positive definite sesquilinear pairing $(.,..)$ and unitary representation 
$\rho^+:G^+\hookrightarrow\U(B/L)$ satisfy the $v$-$\gq$-flexibility conditions of definition \ref{vzwei}. Recall the element $\rho(u)$ coming up in condition (X.2) and write 
it as a direct sum of cyclic unipotents of ranks, say $b_1+1,\dots,b_c+1$. The assumption (X.2) on $B$ implies $\max\{b_i|i=1,\dots,c\}\leq a$ so that the triple $z=2$, $r$ 
and $a$ fulfills the requirements of subsection \ref{freakish}. In what follows our ground field is an algebraic closure $\f^{ac}$ of $\f_{p^r}$: We can choose an example of 
a $\sy$-structure $(N_\sigma,N_{i,\sigma},\zeta_\sigma)$ on an isoclinal $\z/r\z$-graded Dieudonn\'e module $M_\sigma$ of integral graded slope $a$, so that moreover: 
\begin{eqnarray*}
\dim_{\f^{ac}}N_\sigma/N_{\sigma,1}=\begin{cases}1&\sigma\in\Sigma\\2&\sigma\notin\Sigma\end{cases},
&&\dim_{\f^{ac}}M_\sigma/M_{\sigma,1}=\begin{cases}n-1&\sigma\in\Omega\\n&\sigma\notin\Omega\end{cases}.
\end{eqnarray*}
Furthermore, the cardinality of $\Sigma$ is $z=2$, and the cardinality of $\Omega$ is $w=an<r$. Finally 
we invoke the skeletons $I$ and $J$ of $M_\sigma$ and $N_\sigma$, and the representation 
\begin{equation}
\label{warum}
\pi:\SL(J/K(\f_{p^r}))\rightarrow\GL(I/K(\f_{p^r})),
\end{equation}
as in part (ii) of lemma \ref{velf}.

\subsection{Choice of Shimura Variety}
\label{group} 
Our first step is to build the integral Shimura varieties $\M^\delta$: To this end we introduce the CM types $|\Phi^{(0)}|$ and $|\Phi^{(n)}|$ by 
means of the equations \eqref{sechsz} and \eqref{vierz}. In order to complete the list of input data for our kind of PEL moduli problems, we have 
to say what the $*$-skew-Hermitian modules $(V^{(1)},\psi^{(1)})$ and $(V^{(0)},\psi^{(0)})$ are like. We prefer to deal with the sesquilinear forms 
$\frac{1}{v}\Psi^{(1)}$ and $\frac{1}{v}\Psi^{(0)}$ where $\Psi^{(1)},\Psi^{(0)}$ are as in \eqref{polzwei}. To ease notation we denote these forms 
by $(.,..)$, this will not cause confusion with the form on $B$ that was also denoted by $(.,..)$. For each inert place of $L^+$ we have a local requirement:
\begin{itemize}
\item[(L.1)] 
For a real embedding $\iota^+:L^+\rightarrow\r$ the signature of the form $(.,..)$ on $V^{(1)}\otimes_{L^+,\iota^+}\r$ is:
$$=\begin{cases}
(n,0)&\text{ if }\iota\in|\Phi^{(0)}|\cap|\Phi^{(n)}|\\
(n-1,1)&\text{ if }\iota\in|\Phi^{(0)}|-|\Phi^{(n)}|\\
(1,n-1)&\text{ if }\iota\in|\Phi^{(n)}|-|\Phi^{(0)}|\\
(0,n)&\text{ if }\iota\notin|\Phi^{(0)}|\cup|\Phi^{(n)}|
\end{cases}$$
where $\iota$ denotes the unique element in $|\Phi|$ with $\iota|_{L^+}=\iota^+$. 
\item[(L.2)]
For an inert prime $\gr^+$ of $L^+$, there exists an isometry between $B\otimes_{L^+}L_{\gr^+}^+$ and 
$V^{(1)}\otimes_{L^+}L_{\gr^+}^+$, i.e. an $L\otimes_{L^+}L_{\gr^+}^+$-linear map preserving the forms $(.,..)$.
\end{itemize}
Let us check that $V^{(1)}$ exists. By the local-global principle, all we have to do is to prove that the number of infinite places at which the condition (L.1) results in 
a negative discriminant is even. If $n$ is even this is clear because the number of these places is $w$ in that case. If $n$ is odd we have to compute the cardinality 
of $|\Phi^{(n)}|-|\Phi|\pmod2$. To check the parity note that $|\Phi^{(n)}|$ differs from $|\Phi^{(0)}|$ at $w$ real places and that $|\Phi^{(0)}|$ differs from $|\Phi|$ at 
$a$ places and that $w=an\equiv a\pmod2$. Finally we choose any $V^{(0)}$ such that the signature of the form $(.,..)$ on $V^{(0)}\otimes_{L^+,\iota^+}\r$ is
$$=\begin{cases}(n,0)&\text{ if }\iota\in|\Phi^{(0)}|\\(0,n)&\text{ if }\iota\notin|\Phi^{(0)}|\end{cases},$$
where $\iota\in|\Phi|$ extends the real embedding $\iota^+:L^+\rightarrow\r$, as above. Now we have $\M^\delta$, 
$\M_\gp^\delta=\M^\delta\times_{\O_E\otimes\z_{(p)}}\O_{E_\gp}$, and $\Mbar^\delta=\M_\gp^\delta\times_{\O_{E_\gp}}\f_{p^r}$ at our disposal. 
Observe that the above CM types $|\Phi^{(0)}|$, $|\Phi^{(n)}|$ and $|\Phi|$ satisfy the condition (T.1) of section \ref{crys}, and condition (T.2) of 
section \ref{supo}. Hence we are allowed to use the maps $g^{(k)}:\M_\gp^{(0\times1)}\rightarrow\M^{(k)}$ of theorem \ref{extend}, and the whole 
machinery of chapter \ref{tat}. In particular we retain the conventions about $\Xbar[\gq^{*\infty}]=\BT(\bigoplus_\sigma P_\sigma)$, $H$, etc.

\subsection{Choice of Point} 
The second step is to pick a carefully chosen $\xi\in\Mbar^{(0\times1)}(\f^{ac})$. We start with the isogeny class 
$B\otimes_{\O_L}X$. It carries a $L$-action $\iota^{(1)}$ that is stabilized by the Rosati involution coming from the natural 
polarization $\lambda^{(1)}=(.,..)\otimes_L\lambda$. We also assign to it a level structure $\ebar^{(1)}=\eta^{(1)}{K^{(1)}}^p$, where 
$$\eta^{(1)}:\a^{p,\infty}\otimes V^{(1)}\rightarrow\a^{p,\infty}\otimes B\otimes_{\O_L}TX$$
is any choice of $L$-linear symplectic similitude, it exists due to (L.2). In order to constitute a point in the isogeny class 
$(B\otimes_{\O_L}\Xbar,\q_{>0}\lambda^{(1)},\iota^{(1)},\ebar^{(1)})$ we have to fix a specific $p$-divisible group with $\O_{L_{\gq^*}}$-action 
in the isogeny class $B\otimes_{\O_L}\Xbar[\gq^{*\infty}]$ (use the self-duality to get rid of $B\otimes_{\O_L}\Xbar[\gq^\infty]$).\\

We proceed by bringing the results of chapter \ref{wsieben} into play, let us therefore pick a homomorphism 
$$\pi':\SL(J/K(\f_{p^r}))\rightarrow G\times_{L,\iota\circ*}K(\f_{p^r})(=G^+\times_{L^+,\iota\circ*}K(\f_{p^r}))$$
with properties as granted by lemma \ref{R2}, i.e. 
\begin{itemize}
\item
$\pi'$ does not factor through any proper $L^+$-subgroup of $G^+$.
\item
$u$ is conjugated to an element of the form $\pi'(v)$.
\end{itemize}
If we base change $\rho$ to the field $K(\f_{p^r})$ and compose with $\pi'$ we obtain a representation
$\rho\circ\pi':\SL(J/K(\f_{p^r}))\rightarrow \GL(B/L)\times_{L,\iota\circ*}K(\f_{p^r})$ which is abstractly isomorphic to $(I,\pi)$. Eventually we fix an isomorphism:
$$\epsilon_0:(I,\pi)\stackrel{\cong}{\rightarrow}(B\otimes_{L,\iota\circ*}H,\rho\circ\pi')$$
of $\SL(J/K(\f_{p^r}))$-representations. This gives rise to an isogeny of graded displays:
$$\epsilon_\sigma:\q\otimes M_\sigma\rightarrow B\otimes_{\O_L}P_\sigma$$
observe that the displays in question have the same slope. This provides us with the requested $p$-divisible group
$$\BT(\epsilon):\q\otimes\BT(\bigoplus_\sigma M_\sigma)\stackrel{\cong}{\rightarrow}B\otimes_{\O_L}\Xbar[\gq^{*\infty}]$$
and hence with a $\f^{ac}$-point on $\Mbar^{(1)}$. Let finally
$$\xi:\Spec\f^{ac}\rightarrow\Mbar^{(0\times1)}$$ 
be an arbitrary lifting to the catalyst covering Shimura variety. Let us sum up what we have:

\begin{itemize}
\item[(M.1)]
The $\f^{ac}$-valued point $\xi$ is isoclinal and gives rise to the  vector space $B=\Hom_L^0(\Xbar,Y_\xi^{(1)})$ over $L$. This is 
equipped with a positive definite sesquilinear form $(.,..)$, as introduced in 
\eqref{ses}.
\item[(M.2)]
The positive definite vector space in (M.1) is equipped with a unitary representation $\rho^+:G^+\rightarrow\U(B/L)$.
\item[(M.3)]
The $\z/r\z$-graded $W(k)$-window $M_\sigma^{(k)}=P_{\xi,\sigma}^{(k)}$ is equipped 
with a formal $\sy$-structure, i.e. there is an isomorphism of isocrystals with grading:
$$\zeta_\sigma:\q\otimes M_\sigma^{(1)}\rightarrow\q\otimes\bigoplus_{i=1}^cN_{i,\sigma}\otimes_{W(\f^{ac})}\sy_{W(\f^{ac})}^{b_i}N_\sigma$$
with properties as in definition \ref{symD}.
\item[(M.4)]
The representation $\pi$ on the $K(\f_{p^r})$-vector space $B\otimes_{L,\iota\circ*}H\cong I^{(1)}$ (cf. \eqref{wnull}) which is induced 
by the map $\zeta_0$ of (M.3) allows a factorization of the form $\rho\circ\pi'$. Here $\rho$ comes from the representation in (M.2).
\item[(M.5)]
No proper $L^+$-algebraic subgroup of $G^+$ has the property stated in (M.4).
\end{itemize}

One easily derives the following conclusion:

\begin{lem}
\label{mock3}
Let $\xi\in\M_\gp^{(0\times1)}(\f^{ac})$ be a point with the additional structure (M.1) to (M.5). Let $\eta$ be a lift such that the 
(according to \cite[Theorem 4]{zink1}) induced window $\tilde M_\sigma^{(1)}$ is a sufficient deformation (in the sense of lemma 
\ref{suff}) of the $\sy$-structure that is given on $M_\sigma=P_{\xi,\sigma}^{(1)}$. Then $\eta$ has a mock $G^+$-structure.\\ 
In particular, lifts with mock $G^+$-structure do exist under the assumptions at the beginning of this chapter.
\end{lem}
\begin{proof}
Let $H^+$ be a subgroup containing $\theta^{(1)}$. Then lemma \ref{velf} implies that 
$H\times_{L,\iota\circ*}K(\f_{p^r})$ contains the group $\pi(\SL(J/K(\f_{p^r}))$. However (M.5) implies $H^+=G^+$ then.
\end{proof}

\subsection{Choice of Curve} 

For the rest of this chapter we deal with our main theorem:\\

{\em Proof of theorem \ref{newnull}.} We begin with the special case $B=B'$. Lemma \ref{mock3} allows us to pick a lift $\eta$ with mock 
$G^+$-structure. The function field of the schematic image of $\eta$ possesses some finite extension $F\subset\f^{ac}((t))$ to which all 
endomorphisms of all of the $Y_\eta^{(k)}$'s descend. Let us denote the corresponding $F$-valued point by $\eta_0$. Notice that the rigidity 
of endomorphisms implies that $Y_{\eta_0}^{(k)}$ and $Y_{\eta_0}^{(k)}\times_FF^{sep}$ have the same endomorphism algebras.\\
 
Let $Z$ be the normalization of $\M^{(0\times1)}$ in $F$. Consider a point $\chi\in Z$ with residue field $K$, and let us fix the corresponding 
(strictly) henselian local rings $\O_{Z,\chi}^{sh}$, and $\O_{Z,\chi}^h$. This gives rise to an isomorphism of $\Gal(K^{sep}/K)$ with a certain 
subquotient $\Gal(F_\chi^{nr}/F_\chi)$ of $\Gal(F^{sep}/F)$, because $Z$ is normal with function field $F$. Here, $F_\chi^{nr}/F_\chi$ is the 
extension of the fraction fields of $\O_{Z,\chi}^{sh}$, and $\O_{Z,\chi}^h$. Let $H_\gr(\chi)$ (resp. $H_\gr(\eta_0)$) be the $L_\gr$-algebraic 
groups arising as the Zariski closure of the action of $\Gal(K^{sep}/K)$ (resp. $\Gal(F^{sep}/F)$) on $T_\gr Y_\chi^{(1)}\times_KK^{sep}$ 
(resp. $T_\gr Y_{\eta_0}^{(1)}\times_FF^{sep}$). Observe also that the left hand side specialization map in the diagram
$$\begin{CD}
T_\gr Y_\chi^{(1)}\times_KK^{sep}@<{\cong_{\O_{Z,\chi}^{sh}}}<<T_\gr Y_{\eta_0}^{(1)}\times_FF^{sep}@>{\cong_{\f^{ac}[[t]]}}>>T_\gr Y_\xi^{(1)},\\
\end{CD}$$
(as in section \ref{qisg}) turns the $\Gal(K^{sep}/K)$-action into the $\Gal(F_\chi^{nr}/F_\chi)$-one. The group $H_\gr(\eta_0)$ 
commutes with $\End_L^0(Y_{\eta_0}^{(k)})$, which coincides with the preimage of $\End_G(\bigwedge_L^kB)$ under the natural map 
$$\End_L^0(Y_\xi^{(k)})\rightarrow\End_{L_\gr}^0(T_{\xi,\gr}^{(k)})$$
(coming from $m_{\xi,\gr}^{(k)}$). According to part (i) of lemma \ref{R3} we obtain natural inclusions:
$$H_\gr(\chi)^\circ\hookrightarrow H_\gr(\eta_0)^\circ\hookrightarrow G\times_LL_\gr\cdot\mbox{torus}.$$ 
If we apply the theorem of Mori-Zarhin (cf. \cite[Chapitre XII, Th\'eor\`eme 2.5]{mori}) and use part (ii) of lemma \ref{R3} we get 
$$G\times_LL_\gr\subset H_\gr(\eta_0).$$
By a Bertini argument (cf. \cite[Theorem 6.3/Remarque 6.3.18]{bertini}) there exist one-dimensional 
points $\chi$ such that $H_{\gr_0}(\chi)$ and $H_{\gr_0}(\eta_0)$ agree, for some fixed $\gr_0$.\\
 
Together with $G\times_LL_{\gr_0}\subset H_{\gr_0}(\eta_0)$ this tells us that 
$\End_L^0(Y_\chi^{(1)}\times_KK^{sep})$ can be no larger than $\End_G(B)$. We can repeat the argument to get 
$$G\times_LL_\gr\subset H_\gr(\chi)$$
as well. Finally Serre's letter to Ribet yields a Galois extension $K^\circ/K$ with canonical diagrams
$$\begin{CD}
\Gal(K^{sep}/K)@>>>H_\gr(\chi)(L_\gr)\\
@VVV@VVV\\
\Gal(K^\circ/K)@>\cong>>H_\gr(\chi)/H_\gr(\chi)^\circ(L_\gr),
\end{CD}$$
that commute simultaneously for all $\gr$ (cf. \cite[Proposition 1.1]{larsen}). Set $S$ to be the normalization of $Z$ in $K^\circ$, and we are done.\\

If $B'\neq B$ one just uses that the projector $B\oplus B^\perp\mapsto B$ is in $\End_L^0(Y_\eta^{(1)})=\End_L^0(Y_\chi^{(1)})$.
\qed

The following corollary is noteworthy:

\begin{cor}
\label{R5}
Let $G/\q_\ell$ be a connected semisimple group, and write $\Gamma$ for the automorphism group of its root datum 
(over $\q_\ell^{ac}$). Let $\rho:G\rightarrow\GL(C/\q_\ell)$ be a faithful, finite dimensional, linear representation, and 
$p\notin\{2,\ell\}$ be a prime. Then there exists a polarized abelian scheme $(Y,\lambda)$ of dimension smaller than or equal to
\begin{eqnarray*}
&&\Card(\Gamma)^2((\dim_{\q_\ell}C)^3-(\dim_{\q_\ell}C)^2)\\
&&+\Card(\Gamma)((1+\dim_{\q_\ell}G)(\dim_{\q_\ell}C)^2-\dim_{\q_\ell}G\dim_{\q_\ell}C)
\end{eqnarray*}
over a projective and smooth pointed $\f_p^{ac}$-curve $S\stackrel{\xi}{\leftarrow}\Spec\f_p^{ac}$, together with an embedding $f:C\hookrightarrow V_\ell Y_\xi$, such that:
\begin{itemize}
\item[(i)]
the image of $f$ is $\pi_1(S,\xi)$-invariant and totally isotropic with respect to the $\ell$-adic Weil pairing,
\item[(ii)]
the image of the $\pi_1(S,\xi)$-operation on $C$ (pulled back by means of $f$) is a compact open subgroup of $\rho(G(\q_\ell))$.
\end{itemize}
\end{cor}
\begin{proof}
We use theorem \ref{newnull}, lemma \ref{R4} and follow the ideas of \cite{bogomolov}, which in the case at hand is nothing but observing that $[G,G]=G$
\end{proof}
   
\begin{appendix}
\section{Lemmas for finding mock structures}
\begin{lem}
\label{R1}
Let $(B,(.,..),\rho^+)$ be a faithful unitary representation of the $L^+$-group $G^+$ with respect to a quadratic extension $L/L^+$. Let $R$ be a $L$-algebra and 
let $\theta\in G^+(R)\subset\GL_R(B\otimes_LR)$ be a $R$-valued point which is not contained in any strictly smaller $L^+$-group $H^+\subset G^+$. Then one has: 
$$\End_G(B)=\{f\in\End_L(B)|f\mbox{ and }f^*\mbox{ commute with }\theta\}$$
\end{lem}
\begin{proof}
Write $C$ for the right hand side algebra and write $C'$ for the commutant. Both are subalgebras 
of $\End_L(B)$ which are closed under $*$. The inclusion $\End_G(B)\subset C$ is clear. For the converse consider the $L^+$-algebraic group $H^+$ described by the functor 
$Q\mapsto\{\gamma\in(C'\otimes_{L^+}Q)^\times|\gamma\gamma^*=1\}$. As $\theta\in H^+(R)$, we have $G^+\subset H^+$. This shows $C\subset\End_H(B)\subset\End_G(B)$.
\end{proof}

\begin{lem}
\label{R2}
Let $G$ be a connected semisimple $k$-group, where $k$ is a number field, and let $k_\nu$ be the completion at a place. Assume there is a 
unipotent element $u\in G(k_\nu)$ that is not killed by any of the projections onto the $k$-simple factors. Then there exists a homomorphism 
$\pi:\SL_2\times k_\nu\rightarrow G\times_kk_\nu$ which does not factor through any strictly smaller $k$-rational subgroup $H\times_kk_\nu$ 
of $G\times_kk_\nu$, and which maps the element $v=\left(\begin{matrix}1&1\\0&1\end{matrix}\right)$ into the conjugacy class of $u$.
\end{lem}
\begin{proof}
Let $G_1,\dots,G_n$ be the almost $k$-simple normal $k$-subgroups of $G$. Write $u=u_1\dots u_n$ with $u_i\in G_i(k_\nu)$. By the Jacobson-Morozov 
theorem \cite[Theorem 3]{jacobson} we may pick homomorphisms $\pi_i:\SL_2\times k_\nu\rightarrow G_i\times_kk_\nu$ that send $v$ to $u_i$. Consider 
$\pi=\prod_i\pi_i$. This particular choice of $\pi$ implies that $\{g\in G(k_\nu)|\pi(\SL_2(k_\nu))\subset gH(k_\nu)g^{-1}\}$ is nowhere dense, whenever 
$H$ is a proper $k$-subgroup of $G$. By the Baire property we deduce that there exists a $g_0$ such that $\pi(\SL_2(k_\nu))\not\subset g_0H(k_\nu)g_0^{-1}$ 
for all such groups $H$, as there are only countably many. Consequently the homomorphism $g_0^{-1}\pi g_0$ has the required property.
\end{proof}

\section{Lemmas for finding monodromy groups}

\begin{lem}
\label{R3}
Let $\rho:G\rightarrow\GL(V/k)$ be a faithful, finite dimensional, linear representation of a connected semisimple $k$-group $G$ where $k$ is a field of 
characteristic $0$. Assume that $\rho$ contains the adjoint representation on $\gg=\Lie G$ and the trivial one. Assume, moreover, that $H\subset\GL(V/k)$ is 
a connected subgroup which commutes with every element in $\End_G(V)\cup\End_G(\bigwedge_k^2V)$. Then the following holds:
\begin{itemize}
\item[(i)]
$H$ is contained in the product of $\rho(G)$ with the centre of $\End_G(V)^\times$.
\item[(ii)]
If $H$ is reductive and $\End_H(V)=\End_G(V)$, then $\rho(G)\subset H$.
\end{itemize}
\end{lem}
\begin{proof}
By our assumption the representation can be written in the form $V=U\oplus\gg\oplus W$, where we write $U$ for the trivial one-dimensional 
representation. Note that all of $U\otimes_k\gg$, $U\otimes_kW$, $\bigwedge_k^2\gg$, $\gg\otimes_kW$ are $G$-invariant subspaces of 
$\bigwedge_k^2V$. Also, observe that the Lie multiplication 
\begin{equation}
\label{tensor1}
\bigwedge_k^2\gg\rightarrow U\otimes_k\gg,
\end{equation}
as well as the map
\begin{equation}
\label{tensor2} 
\gg\otimes_kW\rightarrow U\otimes_kW,
\end{equation}
the derived action on the subspace $W$, gives rise to $G$-invariant elements in $\Hom_k(\bigwedge_k^2\gg,U\otimes_k\gg)$ 
and $\Hom_k(\gg\otimes_kW,U\otimes_kW)$ which are subspaces of $\End_k(\bigwedge_k^2V)$ in a natural way.\\

To ease notation we now assume that $H$ contains the scalars, we are allowed to do this as scalars act trivially on both spaces $\End_k(V)$ and $\End_k(\bigwedge_k^2V)$. 
In order to prove the first assertion we can assume $\rho(G)\subset H$ too. Now every element $h$ in $H(k^{ac})$ commutes with the projections onto $U$, $\gg$, and 
$W$ and hence induces maps $h_U$, $h_\gg$ and $h_W$ on these spaces. The map $h_U$ is just multiplication by a scalar which we can adjust to be one. Therefore 
\eqref{tensor1} implies that the map $h_\gg=h_U^{-1}h_\gg$ is an automorphism of $\gg$, which due to the connectedness of $H$ has to be an inner automorphism. 
Adjusting by an element in $\rho(G)$ reduces us further to $h_\gg=1$. Finally \eqref{tensor2} forces $h_W$ to be $\gg$-linear, hence $G$-linear. Such an element 
lies in $\End_G(W)$, in fact it lies in its centre, because all of $H$ commutes with$\End_G(V)$. Thus we have shown more specifically that $H$ is contained in 
$T\rho(G)$, where $T\subset\GL(V/k)$ is the torus of all elements that act as a scalar on each of the $G$-isotypic summands of $W$, along with $U\oplus\gg$.\\

Now we prove the second assertion. Let $\gh_1$ be the Lie algebra of $\rho^{-1}(TH)$. It is easy to see that $\gh_1$, regarded as a subspace of $\gg$, is invariant under 
$T$ and $\rho^{-1}(TH)$, and therefore it is invariant under the whole of $H$. The reductiveness of $H$ tells us that there exists a $H$-invariant complement, say $\gh_2$. 
Utilizing $\End_G(V)=\End_H(V)$, shows that $\gh_1$ and $\gh_2$ are ideals in $\gg$, just consider the projectors. In particular $[\gh_1,\gh_2]\subset\gh_1\cap\gh_2=0$. 
Once again, this means that every subspace in $\gh_2$, being $H$-invariant is also $G$-invariant, which is absurd unless $\gh_1=\gg$.
\end{proof}

\begin{rem}
The preceding result is certainly not the best one possible. For $G=\SL_2\times k$, for example one can show that the only faithful representations 
which fail the above assertions are the ones of the form$\sy^0\oplus\dots\oplus\sy^0\oplus\sy^3$. For $G=\PGL_2\times k$ only $\sy^2$ fails.
\end{rem}

\begin{lem}
\label{R4}
Let $H^+/M^+$ be a connected semisimple group, where $M^+$ is a finite extension of $\q_\ell$. Let $C$ be a free $M$-module of finite rank, where $M$ is a semisimple 
$M^+$-algebra of rank two, and let $(C,(.,..)_C,\sigma^+)$ be a unitary representation of $H^+$. Let $p\notin\{2,\ell\}$ be a prime, and $r$ an integer satisfying
\begin{eqnarray*}
&&r\geq[M^+:\q_\ell]\\
&&r>(\Card(\Gamma)\frac{\dim_{M^+}C}{2}+\dim H^++1)
\max\{\frac{\dim_{M^+}C}{2}-1,2\},
\end{eqnarray*}
where $\Gamma$ is the automorphism group of the root datum of $H^+$ (over $\q_\ell^{ac}$). Then there exists: 
\begin{itemize}
\item[(i)]
an isomorphism of quadratic extensions $L\otimes_{L^+}L_{\gr^+}^+/L_{\gr^+}^+\cong M/M^+$, where $[L^+:\q]=r$ and $\gr^+$ is a prime divisor of $\ell$.
\item[(ii)]
an algebraic group $G^+/L^+$ such that $G^+\times_{L^+}M^+$ is the universal cover of $H^+$.
\item[(iii)]
a unitary $G^+$-representation $(B,(.,..)_B,\rho^+)$ of dimension less than or equal to $\Card(\Gamma)\frac{\dim_{M^+}C}{2}$, 
and such that $L\oplus\Lie G\oplus B$ is $v$-$\gq$-flexible, for a choice of an imaginary element $v$ and a prime $\gq|p$ of $L$,
\item[(iv)]
an isometric embedding of unitary $G^+\times_{L^+}M^+$-representations $(C,(.,..)_C,\sigma^+)\hookrightarrow(B\otimes_LM,(.,..)_B,\rho^+\times_{L^+}M^+)$.
\end{itemize}
\end{lem}
\begin{proof}
Choose a number field $L^+$ of degree $r$, together with isomorphisms:  
$$L_\nu^+\cong\begin{cases}\r&\nu|\infty\\K(\f_{p^r})&\nu=\gq^+\\M^+&\nu=\gr^+\end{cases},$$
where $\gq^+|p$ and $\gr^+|\ell$ are fixed primes of $L^+$. Now write $\tilde H^+$ for the simply connected cover of $H^+$, and write $\prod_{i=1}^nRes_{M_i/M^+}H_i$ 
for the canonical factorization of $\tilde H^+$ into Weil restrictions of absolutely almost simple groups $H_i/M_i$. Pick further totally real field extensions $L_1,\dots,L_n$ 
of $L^+$, together with isomorphisms $M_i\cong L_i\otimes_{L^+}L_{\gr^+}^+$ which restrict to the previously chosen one on $M^+$. In each of these choose a prime 
$\gq_i$ over $\gq^+$. By \cite[Theorem B]{harder} there exists a group $G_i/L_i$ such that:
$$G_i\times_{L_i}{L_i}_\nu\cong\begin{cases}\text{ anisotropic form }&\nu|\infty\\\text{ split form }&\nu=\gq_i\\H_i&\nu=\gr_i\end{cases},$$
where $\gr_i$ is the unique prime of $L_i$ over $\gr^+$. Now put $\prod_{i=1}^nRes_{L_i/L^+}G_i=G^+$.\\
The next step is to embed $\sigma^+$ into a unitary representation of $G^+$. To this end we pick a quadratic extension $L$ of $L^+$ 
such that the previously chosen isomorphisms extend to:
$$L\otimes_{L^+}L_\nu^+\cong\begin{cases}\c&\nu|\infty\\K(\f_{p^r})\oplus K(\f_{p^r})&\nu=\gq^+\\M&\nu=\gr^+\end{cases}.$$
Observe that we demand the prime $\gq^+$ to be split in $L$ while $\gr^+$ may or may not be split, depending on whether $M=M^+\oplus M^+$ or not. We also choose 
a homomorphism $\iota:M\rightarrow\q_\ell^{ac}$ (which in the $M=M^+\oplus M^+$-case is not injective). Now let $L'\subset\q_\ell^{ac}$ be a Galois extension of $L$ 
over which the base change $G'=G^+\times_{L^+}L'$ becomes split. In fact we could take for $L'$ the field which trivializes the $\Gal(L^{ac}/L)$-operation on $\Phi(T,G)$, 
the set of roots relative to some chosen maximal torus $T\subset G=G^+\times_{L^+}L$. This gives the crude estimate $[L':L]\leq\Card(\Gamma)$, where $\Gamma$ 
is the subgroup of elements in $\GL_\z(T^*)$ which preserve $\Phi(T,G)$. Observe that the highest weights theory gives a natural $L'$-structure $C'$ on the 
representation $C^{ac}=C\otimes_{M,\iota}\q_\ell^{ac}$. Regarding $C'$ as a $L$-vector space gives us a $[L':L]\frac{\dim_{M^+}C}{2}$-dimensional representation, 
say $(B,\rho)$, of $G$. It is clear that there exists an embedding $f:(C,\sigma)\rightarrow(B\otimes_LM,\rho\times_LM)$ of $G\times_LM$-representations.\\
Finally we wish to give $\rho$ a pairing: Let $Q$ be the $L^+$-vector space of $G^+$-equivariant sesquilinear forms on $B$. For each place $\nu|\infty$ let $Q_\nu^\circ$ 
be the set of positive definite $G^+\times_{L^+}L_\nu^+$-equivariant forms on $L_\nu^+\otimes_{L^+}B$. This is a non-empty open set in $L_\nu^+\otimes_{L^+}Q$, 
as is the subset $Q_{\gr^+}^\circ\subset L_{\gr^+}^+\otimes_{L^+}Q$ of forms for which there is an embedding $f:(C,\sigma)\rightarrow(B\otimes_LM,\rho\times_LM)$ with 
$(f(x),f(y))_B=(x,y)_C$. One sees easily that there exists an element of $Q$ lying in all of these finitely many sets $Q_\nu^\circ$.\\
Now we can check that the unitary representation $L\oplus\Lie G\oplus B$ thus obtained is a flexible one: Choose nontrivial unipotent 
elements $u_i$ in the longest root-space of $G_i({L_i}_{\gq_i})$, and let $u=u_1\dots u_n\in G^+(L_{\gq^+}^+)$. It is then clear that (X.2) holds 
with $a=\max\{\frac{\dim_{M^+}C}{2}-1,2\}$, moreover (X.1) follows from $(\Card(\Gamma)\frac{\dim_{M^+}C}{2}+\dim H^++1)<\frac{r}{a}$. 
\end{proof}

\end{appendix}

\end{document}